\documentclass[11pt]{amsart}
\usepackage{mathtools}
\usepackage{amsmath}
\usepackage{amssymb}
\usepackage{amsthm}
\usepackage{yhmath}
\usepackage{graphicx}
\usepackage[colorlinks,citecolor=cyan,linkcolor=magenta]{hyperref}
\usepackage{ mathrsfs }
\usepackage{bbm}
\usepackage{xcolor}
\usepackage{tikz-cd}
\usepackage{cleveref}
\usepackage{verbatim}
\usepackage{enumitem}
\usepackage{scalerel}
\usepackage{wrapfig}
\usepackage{float}
\usepackage{manfnt}
\newtheorem{theorem}{Theorem}[section]
\newtheorem{claim}[theorem]{Claim}

\newtheorem*{theorem*}{Theorem}
\newtheorem*{definition*}{Definition}
\newtheorem*{prop*}{Proposition}
\newtheorem*{cor*}{Corollary}
\newtheorem*{lemma*}{Lemma}

\newtheorem*{claim*}{Claim}

\newtheorem{lemma}[theorem]{Lemma}

\newtheorem{corollary}[theorem]{Corollary}

\newtheorem{cor}[theorem]{Corollary}

\newtheorem{prop}[theorem]{Proposition}

\theoremstyle{definition}
\newtheorem{definition}[theorem]{Definition}

\theoremstyle{remark}
\newtheorem{remark}[theorem]{Remark}
\newtheorem*{remarks*}{Remarks}

\numberwithin{equation}{section}

\DeclareMathOperator{\imc}{imc}
\DeclareMathOperator{\per}{per}

\DeclareMathOperator{\Hom}{Hom}
\DeclareMathOperator{\length}{length}
\newcommand{\cO}{\mathcal O}
\newcommand{\inverse}{^{-1}}
\newcommand{\Y}{\mathcal Y}

\DeclareMathOperator{\Stretch}{stretch}

\newcommand{\Z}{\mathbb{Z}}

\newcommand{\Hplane}{\mathbb{H}^2}

\newcommand{\op}{\operatorname}

\newcommand{\cal}{\mathcal}

\newcommand{\SO}{\op{SO}}

\newcommand{\bH}{\mathbb H}

\newcommand{\T}{\op{T}}

\newcommand{\be}{\begin{equation}}
	\newcommand{\ee}{\end{equation}}

\renewcommand{\L}{\mathcal L}

\renewcommand{\T}{\mathsf{T}}

\newcommand{\N}{\mathbb N}

\renewcommand{\epsilon}{\varepsilon}
\newcommand{\ep}{\epsilon}

\newcommand{\R}{\mathbb{R}}
\renewcommand{\v}{\mathsf v}
\newcommand{\PSL}{\op{PSL}}

\newcommand{\Qm}{\cal{Q}}

\newcommand{\Sigmazero}{\Sigma_0}

\newcommand{\chprox}{\leadsto}
\newcommand{\chproxb}{{\leftrightsquigarrow}}
\newcommand{\ssm}{\smallsetminus}

\newcommand{\Zxy}[2]{\prescript{}{#2}{\mathsf{Z}}^{#1}}
\newcommand{\Axy}[2]{\prescript{}{#2}{\mathscr{A}}^{#1}}
\newcommand{\s}{\mathscr{S}_+}

\newcommand{\Gimc}{\mathcal{G}^{\mathrm{imc}}}
\newcommand{\G}{\mathcal{G}}

\begin{document}
	
	\graphicspath{ {} }
	
	\title{Classification of horocycle orbit closures in $ \Z $-covers}
 	\author{James Farre}
    \author{Or Landesberg}
    \author{Yair Minsky}
	\begin{abstract}
		We fully describe all horocycle orbit closures in $ \Z $-covers of compact hyperbolic surfaces. Our results rely on a careful analysis of the efficiency of all distance minimizing geodesic rays in the cover. As a corollary we obtain in this setting that all non-maximal horocycle orbit closures, while fractal, have integer Hausdorff dimension.
	\end{abstract}

    \vspace*{-0.2cm}
	\maketitle
 
		\tableofcontents
\section{Introduction}

The study of horospherical flow  in hyperbolic manifolds dates back to Hedlund in the
1930's \cite{Hedlund}, and plays a fundamental role in the modern theory of homogenous
dynamics. The behavior of the flow, both measure-theoretic and topological, is very
	well-studied in the finite-volume and even geometrically-finite cases where it is known to exhibit extreme rigidity, see e.g.~\cite{Furstenberg,Dani-Smillie,Burger,Roblin,Ratner}.  In contrast, the general infinite-volume setting is far less well-understood. 

In \cite{FLM}, we introduced some geometric techniques to study the behavior of horocycle
(and horospherical) orbit closures in the ``first'' symmetric infinite volume case,
namely hyperbolic manifolds with cocompact $\Z$-actions. In particular, for a
regular cover $\Sigma\to \Sigma_0$ with deck group  $\Z$ and $\Sigma_0$ closed, we related the behavior of
horospherical orbit closures to a ``maximal stretch lamination'' for circle valued maps of
$\Sigma_0$. 
In two dimensions this allowed us to demonstrate the sensitivity of the structure of orbit closures to the hyperbolic structure on $\Sigma_0$.

In this paper we complete this study with a classification of orbit closures, in the case that $\Sigma$ is 2-dimensional, in
terms of the geometric information of the stretch lamination, and
combinatorial data which we call the ``slack graph'' of such a lamination.

Our approach provides a new and detailed description of non-maximal horocycle closures. Among the features we obtain:

\begin{itemize}

\item  {\em Integer Hausdorff dimension:}
All horocyclic orbit closures in $\T^1\Sigma$ are of Hausdorff dimension
  1, 2 or 3 (although the only orbit closure which is also a topological manifold is $\T^1\Sigma$ itself). 

\item \emph{Finiteness:} $\T^1\Sigma$ contains finitely many horocycle orbit closures up to translation by the geodesic flow. 

\item  {\em Slack graph:} we introduce a directed graph whose vertices correspond to weak components of
  the stretch lamination and whose edges are geodesic transitions between them. A
  ``slack'' function on the ``fundamental semigroupoid'' of this graph is the basic
  organizing object controlling horocycle closures.

\item  {\em Recurrence semigroup:} There exists a closed non-discrete sub-semi-group of geodesic translations under which the horocycle orbit closures are sub-invariant, $\{t \geq 0 : a_t\overline{Nx} \subseteq \overline{Nx}\}$. This semigroup is countable and of
depth $\omega$ when the stretch lamination covers a multicurve, and it contains a ray in all other cases.

\item {\em Chain proximality:} we introduce and study this metric/dynamical generalization
  of proximality and   find that, in our setting, it is an equivalence relation whose
  classes are the weak components of the stretch lamination, playing a crucial role in reducing the analysis to a finite vertex graph.

\end{itemize}

We were struck by the delicate nature of the structures that arise in this setting, and the contrast with the geometrically finite case. 

We believe the techniques have further reach, with some indication they will apply to higher rank abelian covers, higher dimensional manifolds, and possibly manifolds with quasi-periodic symmetry, non-constant negative curvature, or those locally modeled on certain higher rank homogeneous/symmetric spaces.

A more detailed summary of our results follows.

    \subsection{Main results}
    Let $\Sigmazero$ be a closed, oriented hyperbolic surface with unit tangent bundle $\T^1\Sigmazero \cong G/\Gamma_0$, where $G = \PSL_2(\R)$ and $\Gamma_0\le G$ is a uniform lattice acting isometrically on the right.
    Let $A=\{a_t\}\le G$ denote the diagonal subgroup generating, via left multiplication, the geodesic flow, let $A_+ = \{a_t : t\ge 0\}$, and let $N\le G$ be the lower unipotent subgroup corresponding to the stable horocycle flow on $\T^1\Sigmazero$. We denote by $U$ the opposite horospherical group.

    A homotopy class of maps $\varphi=[\Sigmazero \to \R/\Z]$ determines a normal subgroup $\Gamma \lhd \Gamma_0$, namely the kernel of the map induced on $\pi_1$.
    Let $\pi_\Z: \Sigma \to \Sigmazero$ be the corresponding cover with  deck group $\Gamma_0/\Gamma \cong \Z$.
    Note that the limit set $\Lambda_\Gamma$ of $\Gamma$ is the entire circle $\partial \mathbb H^2$.

    A point $x \in \T^1\Sigma$ is \emph{quasi-minimizing} if there is a constant $c$ such that 
    \begin{equation}\label{eqn: qm}
        d(a_tx, x) \ge t-c
    \end{equation}
    for all $t\ge0$, and $x$ is \emph{minimizing} if \eqref{eqn: qm} holds with $c=0$. We say $x$ is \emph{bi-minimizing} if the entire geodesic $Ax$ is isometrically embedded in $\T^1\Sigma$. We remark that the endpoints of quasi-minimizing rays are the \emph{non-horospherical} limit points in $\partial \Hplane$.
    
    By a Theorem of Eberlein and Dal'bo \cite{Eberlein, Dal'bo}, 
    \[\overline {Nx} \not = \T^1\Sigma \Leftrightarrow\text{ $x$ is quasi-minimizing}.\]
    Let $\Qm \subset \T^1\Sigma$ denote the set of quasi-minimizing points. 

    The asymptotic behavior of quasi-minimizing rays is captured by an oriented chain recurrent \emph{distance minimizing} geodesic lamination $\lambda_0 \subset \Sigmazero$ contained in the maximally stretched set for any best Lipschitz (tight) representative of $\varphi$ \cite[Theorem 1.4]{FLM}:
    \[\bigcup_{x\in \Qm}\omega\text{-}\!\lim_{t\to \infty} \pi_\Z(a_tx) = \T^1\lambda_0,\]
    where $\T^1\lambda_0 \subset \T^1\Sigmazero$ denotes the unit vectors tangent to $\lambda_0$.
    The set $\T^1\lambda = \pi_\Z\inverse (\T^1\lambda_0) \subset \T^1\Sigma$ consists only of bi-minimizing lines. (The lamination $\lambda_0$ is the same as that obtained by Gu\`eritaud-Kassel \cite{GK:stretch} and Daskalapoulos-Uhlenbeck \cite{DU:circle}; see Section \ref{Subsec:Tight map summary of notations}).
    
    Every (isotopy class of) orientable chain recurrent geodesic lamination on $\Sigmazero$ (satisfying a mild positivity condition in homology) appears as the distance minimizing lamination for some $\Z$-cover of a closed hyperbolic surface $\Sigma\to \Sigmazero$ %
    \cite[Theorem 5.3]{FLM} (see also \Cref{rmk: extend construction of stretch laminations}).

    The following result is a corollary of the classification and structure theory for $N$-orbit closures developed throughout the paper.
    It states that the \emph{dynamical structure} of non-maximal $N$-orbit closures can be read from the \emph{topological features} of $\lambda_0$, which is itself obtained by solving a \emph{geometric optimization} problem.  

    \begin{theorem}\label{thm: dichotomy intro}
        There is a dichotomy.
        \begin{enumerate}[label = (\alph*)]
            \item $\lambda_0$ is a simple multi-curve: for all $x \in \Qm$, $\overline {Nx}$ is a countable union of horocycles, hence has Hausdorff dimension $1$.
            The set of endpoints of quasi-minimizing rays in $\partial\Hplane$ is countable. %
            \item $\lambda_0$ contains an infinite leaf: for all $x\in \Qm$, $\overline{Nx}$ has Hausdorff dimension $2$ and $\overline {Nx} \cap A_+x$ contains a ray.
            The set of  endpoints of quasi-minimizing rays in $\partial \mathbb H^2$ is an uncountable set with Hausdorff dimension 0. %
        \end{enumerate}
    \end{theorem}  

     In Theorem 1.13 of \cite{FLM}, we demonstrated that the \emph{topology} of non-maximal $N$-orbit closures is not rigid. More precisely, we constructed sequences $(\Sigmazero^i)_{i}$ of hyperbolic metrics on a closed surface $S_0$ converging to   $\Sigmazero^\infty$,  with associated  $\Z$-covers $\Sigma^i \to \Sigmazero^i$ such that no non-maximal $N$-orbit closure in $\T^1\Sigma^i$ was homeomorphic to any non-maximal $N$-orbit closure in $\T^1\Sigma^\infty$.
    \Cref{thm: dichotomy intro} above recovers a certain amount of rigidity of  $N$-orbit closures in $\Z$-covers:

    \begin{corollary}\label{cor: intro H dim}
           All $N$-orbit closures in $\Z$-covers of closed hyperbolic surfaces have integer Hausdorff dimension.
    \end{corollary}

    \begin{remarks*}\;
    	\begin{enumerate}[leftmargin=*]
    		\item The non-maximal orbit closures are never homogeneous nor even topological submanifolds, see \S\ref{subsec: non-regularity of orbit closures}.
    		\item Not all surfaces satisfy dimension rigidity for $N$-orbit closures.  It is well known that convex cocompact surfaces can have minimal horocycle orbit closures of arbitrary dimension between $2$ and $3$, coming from $2=\dim AN$ plus the dimension for the limit set. 
    		In forthcoming work with F. Dal'bo, we expect to construct a geometrically infinite hyperbolic surface (without cyclic symmetry) with a horocycle orbit closure having arbitrary Hausdorff dimension between $1$ and $2$.
    		\item The dichotomy in \Cref{thm: dichotomy intro} may also be described in terms of the existence of certain types of limit point in $\partial \Hplane$ called Garnett points, see \S\ref{subsec: dichotomy}.
      \item We believe a similar result should hold for $\Z$-covers of finite volume and geometrically finite surfaces, when considering quasi-minimizing rays eventually contained in the convex core.
    	\end{enumerate}
    \end{remarks*}

    \medskip
    The set of quasi-minimizing points is naturally partitioned $\Qm = \Qm_- \sqcup \Qm_+$, where $x\in \Qm_\pm$ if $A_+x$ exits the `$\pm$'-end of $\T^1\Sigma$.
    We will focus our attention on $N$-orbit closures facing the `$+$'-end of $\T^1\Sigma$; the analysis for the `$-$'-end is analogous (see \S\ref{subsec: minus end}). %
    
    \subsection*{Slack}

    Our description of $N$-orbit closures in $\T^1\Sigma$ is facilitated by the choice of a $1$-Lipschitz tight map 
    \[\tau_0: \Sigmazero \to \R/c\Z \in \varphi,\]
    i.e., a best Lipschitz map representing $\varphi$ (see \S\ref{Subsec:Tight map summary of notations}).  A lift of $\tau_0$ to a $\Z$-equivariant $1$-Lipschitz mapping $\tau: \Sigma \to \R$ constitutes a `ruler' that allows us to measure and compare the progress of projections of $A_+$-orbits in $\T^1\Sigma$ to $\Sigma$. 
    Abusing notation, we also denote by $\tau: \T^1\Sigma\to \R$ the pullback of $\tau$ along the tangent projection.

    Consider bi-minimizing points $x$ and $y \in \T^1_+\lambda = \T^1\lambda \cap \Qm_+$ satisfying $\tau(x) = \tau(y)$.
    Since $\overline {Nx} \subset \Qm_+$ and $\Qm_+$ is foliated by $A$-orbits,
    any good description of $\overline {Nx}$ would require a thorough understanding of the set
         \begin{equation}\label{eqn: def Zxy intro}
         \Zxy xy = \{t\in \R: a_ty \in \overline{Nx}\}.
     \end{equation}
    In \cite[\S7]{FLM}, we proved that $\Zxy xx \subset [0,\infty)$  has the structure of a non-discrete semi-group.
    
    For a more detailed description of  $\Zxy xy$, we quantify the \emph{efficiency} of certain geodesics $Az$ that join the past of $Ay$ with the future of $Ax$, as in Figure \ref{fig: slack intro}.
    The set of such $Az$ is denoted by $\Axy xy$.
    
    \begin{figure}[h]
        \centering
        \includegraphics[width=0.6\linewidth]{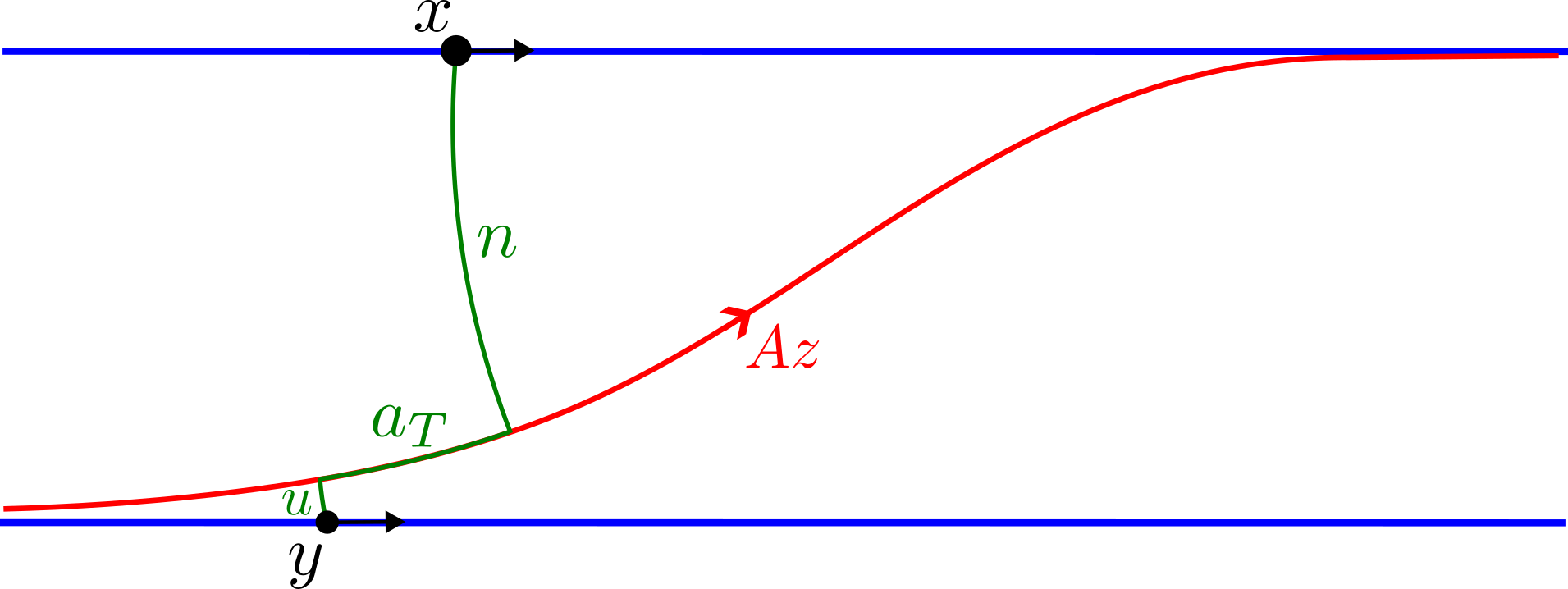}
        \caption{\small The relationship between slack, the Bruhat
        decomposition, and $N$-orbit closures.}
        \label{fig: slack intro}
    \end{figure}
    
    Efficiency is measured by an invariant of paths in $\T^1\Sigma$ called \emph{slack} (Definition \ref{def: slack}).  
    For a line $Az$, the slack $\s(Az) \in [0,\infty]$ is
    \[\s(Az) = \lim_{t\to \infty} 2t - \tau(a_t z) +\tau(a_{-t}z).\]
    The basic idea is, for $Az \in \Axy xy$ with $z$ very close to $y$, that the $A_+$ orbit of $z$ `wastes' roughly time $\s(Az)$ before catching up and becoming strongly asymptotic with $x$.  If $z$ is very close to $y$, then there will be an offset of about $\s(Az)$ corresponding to this recurrence of $Nx$ near $Ay$. 

    More formally, slack is a geometric avatar of the $A$-coordinate in the Bruhat decomposition (see \S\ref{Sec:Slack of Paths}):
    for $Az \in \Axy xy$, lift $x, y \in G/\Gamma$ to $g,h \in G$ using the path $Az$ (see Figure \ref{fig: slack intro}) and write 
    \[gh\inverse = na_Tu\in NAU.\]
    Then we have %
    \[\s(Az) = T.\]
    
    If $Az$ is very close to $y$, then $u$ is very small.
    So, if $T$ is an accumulation point of $\{\s(Az_m)\}$ where $Az_m\in \Axy xy$ are lines that tend to $Ay$ on compact sets near $y$ as $m\to \infty$, then $a_Ty \in \overline {Nx} \cap Ay$.  Conversely, every point in $\overline {Nx} \cap Ay$ arises in this way (Lemma \ref{lem:slack gives points in Nxbar}).

    \begin{remark}
        We draw the reader's attention to the similarity between the shape of the symbol `$\mathsf Z$' in \eqref{eqn: def Zxy intro} and the arrangement of lines $Ay$, $Az$, and $Ax$ in Figure \ref{fig: slack intro}.  The authors found that this notation helped us to keep track of the roles played by $x$ and $y$.
    \end{remark}

    \subsection*{The slack graph}
    A \emph{weak component} $\mu\subset \lambda \subset \Sigma$ is a sublamination with the property that the $\epsilon$-neighborhood of $\mu$ is connected for every $\epsilon>0$.
    Weak components of $\lambda$ project to components of $\lambda_0$, and the preimage of a component of $\lambda_0$ is a finite\footnote{This is not completely obvious, but can be deduced from covering space theory applied to a suitable \emph{train track neighborhood} of $\lambda_0$ or by combining \Cref{thm: chprox equiv chain recurrent} and \Cref{cor: chprox and weak components}.}
 union of weak components of $\lambda$.
    
    We define a directed graph $\G$\footnote{We should really decorate $\G$ with a `$+$'-sign, but do not for readability.} called the \emph{slack graph} whose vertex set $V(\G)$ consists of a choice of $x \in \T_+^1\mu \cap \tau\inverse (0)$, as $\mu$ ranges over the (finitely many) weak components of $\lambda$.
    For $x$ and $y\in V(\G)$, the directed edge set from $y$ to $x$ is $\Axy xy$.

    Slack extends to a morphism $\s: \Pi(\G) \to \R_{\ge0}$, where $\Pi(\G)$ is the fundamental semi-groupoid of finite directed paths in $\G$. %
    For $y$ and $x \in V(\G)$, denote by $\Hom_{\G}(y,x)$ the set of finite directed paths joining $y$ to $x$.
    \medskip

    Consider the special case that $\lambda_0$ is a multi-curve, which implies that $\T^1_+\lambda = \cup_{x \in V(\G)}Ax$ is the (finite) union of uniformly isolated leaves.
    For a set $S\subset \R$, the \emph{derived set} $S^{(1)}$ is obtained from $S$ by removing the isolated points from $S$.
    Inductively, $S^{(i)}$ is the derived set of $S^{(i-1)}$. 
    We say that $S$ has depth $\omega$ if $S^{(i)} \not = \emptyset$ for all $i$ and $\cap_{i \in \mathbb N} S^{(i)} = \emptyset$.

    \begin{theorem}\label{thm: depth omega intro}
        If $\lambda_0$ is a multi-curve, then for all $x, y \in V(\G)$, $\Zxy xy$  is countable with depth $\omega$ and satisfies
        \[\s(\Hom_{\G}(y,x)) = \Zxy xy. \]
    \end{theorem}

\begin{remarks*}
			\begin{enumerate}
				\item We, in fact, provide a precise description of the depth of each point in $\Zxy{x}{y}$ where accumulations are filtered by the combinatorial length of paths in $\G$, via the slack map $\s$. See \Cref{subsec: Depth} for details.
				\item This statement (as well as its proof, to a certain extent, given in \S\ref{sec: finite lamination}) is reminiscent of the celebrated result of J{\o}rgensen and Thurston that the set of volumes of hyperbolic $3$-manifolds is a well ordered set of ordinal type $\omega^\omega$, which, in particular, has depth $\omega$ \cite{Gromov:JT, Thurston:notes}.  By an explicit computation along the lines of \Cref{lem: non-trivial accumulation}, however, we know that $\Zxy xx$ is never well-ordered.
			\end{enumerate}
	\end{remarks*}

    The proof of this statement does not use anything specific to dimension $2$, hence generalizes to horospherical orbit closures projected to the tangent bundle in $\Z$-covers of higher dimensional hyperbolic manifolds; see \Cref{thm: d-manifolds intro}.
    \medskip

    Returning to the general case,  any $x \in \Qm_+$, $A_+x$ is eventually contained in the $\ep$-neighborhood of  $\T^1_+\mu$, for some weak component $\mu$ and every $\ep>0$ \cite[Theorem 3.4]{FLM};
    define 
    \[\v: \Qm_+ \to V(\G)\]
    accordingly. There is an $N$-invariant, upper semi-continuous function defined, for $x\in \T^1\Sigma$, by 
    \[\beta_+(x) = \tau(x) - \s(A_+x) \in [-\infty, \infty).\]
    Then $\beta_+(x)>-\infty$ if and only if $x\in \Qm_+$.
    The \emph{marked Busemann function} 
    \[\hat\beta_+: \Qm_+ \to \R\times V(\G)\]
    for the `$+$'-end of $\T^1\Sigma$ is defined by $\hat \beta_+(x) = (\beta_+(x), \v(x))$.
    
    The following reduces the problem of computing arbitrary $N$-orbit closures to a finite list, up to $A$-translation.  
    
    \begin{theorem}\label{thm: intro reduction}
        For all $x$, $y\in \Qm_+$, 
        \[\overline{Nx} = \overline{Ny} \text{ if and only if } \hat\beta_+(x) = \hat\beta_+(y).\]
        In particular, \[\overline {Nz} = a_{\beta_+(z)}\overline{N\v(z)}.\]
    \end{theorem}

    Using the symmetry that $\lambda_0$ can be computed either as the set of  $\omega$-limit points of projections of qausi-minimizing rays exiting the `$+$'-end or the `$-$'-end, Theorem \ref{thm: intro reduction} counts the number of distinct $N$-orbit closures,  up to $A$-translation.
    \begin{corollary}\label{cor: intro number N orbit closures}
        There are exactly $2|V(\G)|+1$ distinct $N$-orbit closures in $\T^1\Sigma$, up to $A$-translation.\footnote{The $+1$ is for all of $\T^1\Sigma$.} %
    \end{corollary}
    Implicit is the statement that, in the definition of $\G$, our choice of $x \in \T^1_+\mu \cap \tau\inverse(0)$ for each weak component $\mu \subset \lambda$,  was immaterial.
    A major ingredient in its proof is the study of a certain \emph{chain proximality} relation on $\T_+^1\lambda \cap \tau\inverse(0)$, discussed below. 
    This part of our analysis relies heavily on the structure of geodesic laminations on surfaces (rather than higher dimensional manifolds). %

    The following result gives a description of the $N$-orbit closure of a vertex of $\G$, which is essentially reduced to finitely many computations of the sets $\Zxy xy$ for $x$ and $y \in V(\G)$.
    \begin{theorem}\label{thm: intro N orbit vertex}
    For all $x,y \in V(\G)$
    \[\Zxy {x}{y} = \overline{\s(\Hom_{\G} (y, x))}. \]
    For each $x \in V(\G)$, 
        \[\overline{Nx} = \hat\beta_+\inverse \left ( \bigcup_{y \in V(\G)}  \overline{\s(\Hom_{\G} (y,x))} \times \{y\}\right). \]
    \end{theorem}

    In the case that $\lambda_0$ contains an infinite leaf, we obtain the following structural properties of $\Zxy {x}{y}$:
    
    \begin{theorem}\label{thm: intro contains a ray}
    For any $x,y \in V(\G)$ the shift set $\Zxy{x}{y}$ contains a ray $[\rho_{y,x}, \infty)$ if and only if $\lambda_0$ contains an infinite leaf. In that case, the constant $0 \leq \rho_{y,x} < \infty$ is given explicitly from the graph $\G$, and $\Zxy {x}{y} \setminus [\rho_{y,x},\infty)$ is at most countable with finite depth.
    \end{theorem}

    A more detailed description of $\Zxy {x}{y} \setminus [\rho_{y,x},\infty)$ is given in \S\ref{subsec: slack graph G}.
    \subsection*{Examples}

    To illustrate our main theorems, we consider the following three prototypical examples illustrated in Figure \ref{fig: prototypes}:
    \begin{enumerate}
        \item $\lambda$ has finitely many leaves, i.e., $\lambda_0$ is a multi-curve.
        \item $\lambda$ is weakly connected with countably many leaves.
        \item $\lambda$ is weakly connected with uncountably many leaves and $\lambda_0$ is minimal. 
    \end{enumerate}

    \begin{figure}[h]
        \centering
        \includegraphics[width=1\linewidth]{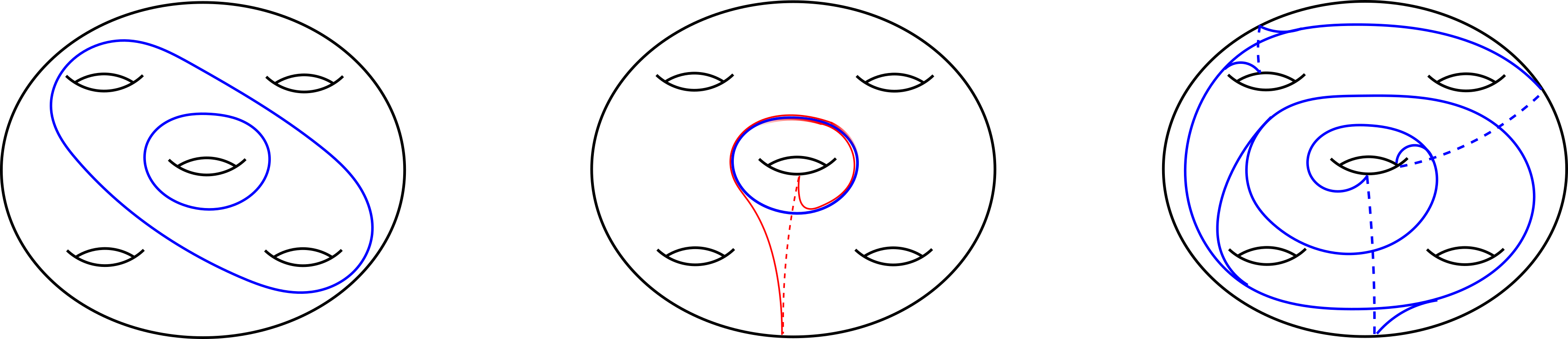}
        \caption{\small Prototypical examples (1)--(3) pictured in terms of $\lambda_0 \subset\Sigmazero$.  The middle example consists of a simple closed curve (blue) and an isolated (red) chain recurrent leaf that spirals onto it.  The rightmost example pictures a (train track approximation) of a minimal, orientable lamination with uncountably many leaves.}
        \label{fig: prototypes}
    \end{figure}

    By the structure theory for geodesic laminations on hyperbolic surfaces \cite{Thurston:notes}, an orientable chain recurrent geodesic lamination has a finite number of connected components, each consists of finitely many minimal sublaminations and chain recurrent isolated leaves spiraling onto them, so 
    an arbitrary $\lambda_0$ (hence $\lambda$) exhibits some combination of these prototypical behaviors.

    For example (1), 
    Theorems \ref{thm: depth omega intro}, \ref{thm: intro reduction}, and \ref{thm: intro N orbit vertex} imply that $\overline {Nx}$ is a countable union of $N$-orbits for every $x \in \Qm$; see \Cref{cor: multi-curve structure}.

    In examples (2) and (3), there is only one vertex in $\G$ and $\lambda$ contains infinitely many leaves.  In these examples, Theorems \ref{thm: intro reduction}--\ref{thm: intro contains a ray} assert that for every $x\in \Qm_+$ the orbit closure $\overline{Nx}$ is a ``$\beta_+$-horoball'':
    \begin{equation}\label{eqn: uniform horoball intro}
        \overline {Nx} = \beta_+\inverse([\beta_+(x),\infty)).
    \end{equation}

    \begin{remark}\label{rmk: Garnett}
        In example (2), there are countably many non-Garnett limit points in $\partial \mathbb H^2$, while in example (3), there are uncountably many such.  See \S\ref{subsec: dichotomy} for a discussion of Garnett points.
    \end{remark}
        
    In our previous work, we obtained a result \cite[Theorem 1.12]{FLM} with a similar flavor as \eqref{eqn: uniform horoball intro} under a dynamical constraint on the first return system to a fiber of $\tau_0$ induced by the geodesic flow tangent to $\lambda_0$, namely that it was minimal and weak-mixing.  In such a system, essentially every pair of points (a dense $G_\delta$ set of pairs) is \emph{proximal}.
    In this article, we show that a weaker notion of \emph{chain proximality} for pairs $x$ and $y$ is sufficient to guarantee that $\overline {Nx} = \overline {Ny}$;
    see \S\ref{subsec: chain proximality intro} and \S\S\ref{sec:chain prox min}--\ref{Section:Synchronous Transversal}, below.

    \subsection*{Partial results in higher dimensions}
    Some of the techniques developed in this paper (e.g., \S\S\ref{Sec:Slack of Paths}--\ref{sec: finite lamination}) apply more generally to the setting that $\Sigma_0$ is a hyperbolic $m$-manifold, $\Sigma\to \Sigmazero$ is a $\Z$-cover, and  $\T^1\lambda_0\subset \T^1\Sigma_0$ consists only of (finitely many) closed $A$-orbits.
    Recently, Cameron Rudd found infinitely many examples of closed hyperbolic $3$-manifolds fibering over the circle where $\lambda_0$ consists of (short) closed curves transverse to the surface fibers \cite{Rudd}.
    Thus the hypotheses of the following are verified, in dimension $3$, by many interesting examples.
    \begin{theorem}\label{thm: d-manifolds intro}
        Let $\Sigma \to \Sigma_0$ be a $\Z$-cover of a closed hyperbolic $m$-manifold and suppose that the corresponding distance minimizing lamination $\lambda_0$ is a multi-curve in $\Sigmazero$. 
        
        Then the closure of every non-dense leaf of the foliation of $\T^1\Sigma$ by horospheres is a countable union of horospheres, and the depth of the corresponding recurrence semigroup, $\Zxy x x$, is $\omega$.
    \end{theorem}

    Geodesic laminations in higher dimension do not have as strong a structure theory as in dimension 2, and for this reason 
    we are not able to establish any of our results for the infinite-leaf case.

    \subsection{About the proof}
    The slack of $Az$ behaves very differently depending on how much time $Az$ spends near $\T_+^1\lambda$.
    We therefore have two main strategies for controlling or computing slacks, depending on whether trajectories are forced to make big-slack excursions between components of $\T^1_+\lambda$ or can travel between leaves in one component with small slack.
    These strategies are combined in \S\ref{sec: structure of horo orbit closures} by way of the slack graph $\G$ to obtain our main Theorems \ref{thm: intro reduction}--\ref{thm: intro contains a ray}.

    \subsection*{Geometric limit chains}

The case that $\lambda_0$ consists of closed curve components is considered in \S\ref{sec: finite lamination}, 
where we obtain lower bounds on slack proportional to the length of excursions of $Az$ between components of $\T^1_+\lambda$,
    and prove Theorem \ref{thm: depth omega intro} (see also Theorems  \ref{thm:structure of Zxy} and \ref{thm: countable filtration Zxy} and Corollary \ref{cor: depth omega}). 
    The structure of accumulation points comes from analyzing how sequences of lines $Az_m \in \Axy xy$ with bounded slack degenerate as $m\to \infty$ when $\lambda_0$ is a multi-curve.
    
    The idea is that the bound on slack forces the geometry of $Az_m$ away from $\T^1_+\lambda$ to stabilize (up to subsequence), but subsegments can spend arbitrarily long amounts of time near different components of $\T^1_+\lambda$,  accumulating very little slack.
    By studying the possible geometric limits up to the $\Z$-action, we show that there are finitely many \emph{geometric limit chains}, each of finite length, consisting of geodesics $Az_1\cdots Az_r$ forming directed paths in $\mathcal G$, and $\sum \s (Az_i)$ is an accumulation point of $\{\s(Az_m)\}$.

    \medskip
    
    \subsection*{Chain proximality}\label{subsec: chain proximality intro} 
        When $\lambda$ has a weak component with infinitely many leaves, there are lines $Az$ that make (infinitely) many small jumps between geodesics in $\T^1_+\lambda$ building up arbitrarily \emph{small} slack.
    This phenomenon leads us to the chain proximality relation on points of $\T^1_+\lambda$ in the same $\tau$-fiber introduced and studied in \S\S\ref{sec:chain prox min}--\ref{Section:Synchronous Transversal}.
    
    It is not difficult to see that if rays $A_+x$ and $A_+y$ are \emph{proximal}, then $\overline {Nx} = \overline {Ny}$ \cite[\S8]{FLM}.
    We formulate a weaker, not necessarily symmetric notion of \emph{chain proximality} (Definition \ref{def: chain proximal}) on points in $\T_+^1\lambda \cap \tau\inverse (0)$ and prove that if $x$ is chain proximal to $y$, written $x \chprox y$, then $\overline {Nx}\subset \overline {Ny}$.
    Essentially, $x \chprox y$ means that a pseudo-orbit of the geodesic flow starting at $x$ can \emph{intercept} the geodesic orbit of $y$ in $\T^1\Sigma$ in a synchronous fashion, i.e., the pseudo-orbit starting at $x$ arrives to the $A_+$-orbit of $y$ at the same time as $y$.
    In addition, the total distance of all of the jumps made by such pseudo-orbits can be made arbitrarily small.

    The notion of chain proximality applies to an arbitrary transformation or flow on a metric space. 
    We study, in detail, the chain proximality relation on pairs of points
    for the first return 
    \[\sigma : \T_+^1\lambda_0 \cap \tau_0\inverse(0) \to \T_+^1\lambda_0 \cap \tau_0\inverse(0) \] for the geodesic flow and prove
  
    \begin{theorem}\label{thm: chain prox intro}
        $\sigma$-chain proximality is a (symmetric) equivalence relation on $\T_+^1\lambda_0 \cap \tau_0\inverse (0)$ with finitely many equivalence classes.
        Moreover, there is a positive $d$ such that in the $d$-fold cover $\pi_d : \Sigma_d \to \Sigmazero$ intermediate to $\pi_\Z: \Sigma \to \Sigmazero$, the connected components of $\lambda_d = \pi_d\inverse(\lambda_0)$ identify both the chain proximality equivalence classes and the weak components of $\lambda$ in $\Sigma$.
    \end{theorem}

    See Theorems \ref{thm: chain prox invt equiv relation} and \ref{thm: chprox equiv chain recurrent} and \Cref{cor: chprox and weak components} for precise statements.
    We point out that Theorem \ref{thm: chain prox invt equiv relation} does not use the structure of the tight map $\tau_0$ and may be of interest outside of the context of the present article.
    \medskip

    Chain proximality is invariant under, e.g., bi-Lipschitz conjugations, and is both a dynamical and geometrical concept.
    To illustrate this, note that there are minimal, orientable geodesic laminations $\mu_0 \subset \Sigma_0$ equipped with a transverse measure such that the first return $\sigma$ to a suitable transversal admits an order and measure preserving (topologically semi-) conjugacy to an irrational circle rotation.
    No two distinct points are chain proximal for any circle rotation $t \in \R/c\Z \mapsto t+\theta \mod c\Z$.  
    However, the interaction of the hyperbolic geometry of $\mu_0 \subset \Sigma_0$ together with the dynamics of the first return mapping results in only finitely many (in fact, only one) chain proximality equivalence classes for $\sigma$.  

    In a similar fashion, if the first return map $\sigma$ admits, as a continuous factor, a \emph{rational} circle rotation of order $q$, then there are at least $q$ chain proximality equivalence classes for $\sigma$.
    In particular, if $\lambda_0$ is minimal and $\sigma$ admits a non-trivial continuous rational eigenfunction to $\mathbb C$, then there are strictly more chain proximality equivalence classes than connected components of $\T_+^1\lambda_0$.
    We remark that, for topological reasons, if $\lambda_0$ is minimal and filling, then $\sigma$ does not admit a continuous rational  eigenfunction (see Remark \ref{rmk: minimal filling lamination}).

    \subsection{Chain-recurrence of the stretch lamination}\label{chain recurrence always}
For a given tight map $\tau_0: \Sigmazero \to \R/c\Z$, denote by $\Stretch(\tau_0)$ the maximally stretched locus, i.e., the set of points whose local Lipschitz constant is the global Lipschitz constant.
    In Gu\`eritaud-Kassel \cite{GK:stretch}, it is shown that the intersection of maximal stretch loci over all tight representatives of a given homotopy class $\varphi$ of maps $\Sigmazero \to \R/\Z$ is a geodesic lamination $\lambda_0(\varphi)$, and that a tight map  $\tau_0$ can be chosen such that $\lambda_0(\varphi) = \Stretch (\tau_0)$.
  When the dimension is 2, they also show that this lamination is chain-recurrent.
  
  In the appendix we will extend this result:
\begin{theorem}\label{L chain recurrent}
Let $\Sigma_0$ be a closed hyperbolic $m$-manifold.   For any nontrivial homotopy class $\varphi$ of maps   $\Sigma_0\to
\R/\Z$, the stretch lamination $ \lambda_0(\varphi) $ is chain-recurrent.
\end{theorem}
(Note that this is not a strict generalization of \cite{GK:stretch}, as they prove their
result for any hyperbolic target manifold, and our target is always a circle.)

In view of this, unless stated otherwise, we will assume that our tight map $\tau_0$
has been chosen such that $\Stretch (\tau_0) = \lambda_0([\tau_0])$.  
    
    \subsection{Organization of the paper}
    After some preliminaries and a summary of some notations following our previous work in \S\ref{sec: preliminaries}, the paper is divided into three parts:
    \subsection*{\S\S\ref{Sec:Slack of Paths}--\ref{sec: finite lamination}}
    In \S\ref{Sec:Slack of Paths}, we discuss slack and collect some of its basic properties that will be required throughout.
    In \S\ref{sec: finite lamination}, we give a detailed account of the structure of horocycle orbit closures when $\lambda_0$ is a multi-curve in terms of the slack graph $\G$. In particular we prove \Cref{thm: depth omega intro,thm: d-manifolds intro}.
    We also discuss, in \S\ref{subsec: infinite cr leaf}, what happens when $\lambda_0$ has finitely many leaves but contains an infinite leaf.
    The results in these sections hold for hyperbolic manifolds in higher dimensions, as well.

    \subsection*{\S\S\ref{sec:chain prox min}--\ref{Section:Synchronous Transversal}}
    In \S\ref{sec:chain prox min}, we discuss the chain proximality relation and prove that,  when considering the first return map to a $C^1$ transversal for the geodesic flow tangent to an orientable and minimal geodesic lamination in a hyperbolic surface, chain proximality is an equivalence relation with finitely many equivalence classes   (this may be of independent interest).  
    In \S\ref{Section:Synchronous Transversal}, we specialize to the setting that the geodesic lamination of interest is $\lambda_0$ and apply our techniques from the previous section to each minimal sublamination, after we build a nice transversal.  As a byproduct of the construction of a nice transversal, we obtain a structural result for the behavior of tight circle valued maps in a neghborhood of $\lambda_0$ (Corollary \ref{cor: structure of tight maps}).
    We then pass to a finite cover $\Sigma_d \to \Sigma_0$ intermediate to $\Sigma \to \Sigma_0$ and analyze the chain proximality equivalence relation in the cover.  In particular, we obtain Theorem \ref{thm: chprox equiv chain recurrent}, which identifies the equivalence classes as connected components of the preimage of $\lambda_0$ in $\Sigma_d$.

    \subsection*{\S\ref{sec: structure of horo orbit closures}}
    In this section we combine the work done in the previous two parts to conclude our main structural results for general surfaces, implying in particular \Cref{thm: dichotomy intro,thm: intro reduction,thm: intro N orbit vertex,thm: intro contains a ray} and \Cref{cor: intro number N orbit closures}. We also construct some sequences of closed hyperbolic surfaces and corresponding $\Z$-covers that stay in a compact set of metrics, but which have recurrence semi-groups with considerably different structures.  These examples  illustrate further non-rigidity properties of $N$-orbit closures in the category of $\Z$-covers of closed hyperbolic surfaces. Finally, we explain why non-maximal $N$-orbit closures are not manifolds.

    \subsection*{Acknowledgements}
    The first named author would like to thank Xiaolong Hans Han for pointing out Proposition 9.4 in \cite{GK:stretch}.  We thank Hee Oh for helpful remarks on an earlier draft. The first named author received funding from DFG – Project-ID 281071066 – TRR 191, and the third named author was partially supported by DMS-2005328.

\section{Preliminaries}\label{sec: preliminaries}
 
 This section contains odds and ends summarizing important results as well as a few simple observations and remarks on terminology, all of which should help facilitate reading through this manuscript. 
 The reader may find it useful to consult our previous paper \cite[\S\S2--3]{FLM} for a more thorough discussion of preliminaries including quasi-minimizing rays, geodesic laminations and the basic relationship between best Lipschitz (tight) maps and distance minimizing geodesic laminations.
 
\subsection{Background on tight maps and distance minimizing laminations}\label{Subsec:Tight map summary of notations}

Here we briefly recall some terminology, notation, and results regarding tight maps and their maximally stretched sets from \cite{FLM}.

Throughout the paper, unless otherwise indicated, $\Sigma_0$ is a closed, oriented hyperbolic surface (the discussion in this subsection also holds for arbitrary hyperbolic manifolds of dimension at least $2$), $\varphi$ is a non-trivial homotopy class of maps $\Sigma_0\to \R/\Z$, and $\pi_\Z: \Sigma \to \Sigmazero$ is the associated $\Z$-cover.
A map $\tau_0 : \Sigma_0 \to \R/\Z \in \varphi$ is called \emph{tight} if its Lipschitz constant realizes the following naive lower bound 
    \[\sup_{\gamma \subset \Sigmazero}\frac{\deg \varphi|_\gamma}{\ell(\gamma)}\]
    on the Lipschitz constant of any representative of $\varphi$,
    where $\ell(\gamma)$ is the hyperbolic length of a closed curve $\gamma \subset \Sigmazero$.
    Thus a tight map has, in particular, the smallest Lipschitz constant in its homotopy class.
    
    By composing a tight map $\Sigma_0 \to \R/\Z$ with an affine map $\R/\Z \to \R/c\Z$, we can assume that its Lipschitz constant is $1$.
    Abusing notation, we use the same letter $\tau_0: \T^1\Sigmazero \to \R/c\Z$ to denote the pullback of our $1$-Lipschitz tight map along the tangent projection $\T^1\Sigmazero\to \Sigmazero$.
    Any lift  $\tau : \T^1\Sigma \to \R$ of $\tau_0: \T^1\Sigmazero \to \R/c\Z$ is equivariant with respect to the deck group $\Z$: 
    \[\tau(k.x) = \tau(x) +kc\in \R.\]

    In \S3 of our previous paper, we associated to $\tau_0 : \T^1\Sigma_0 \to \R/c\Z$ the $A$-invariant part $\L_0 \subset \T^1\Sigmazero$ of the set of points $x \in \T^1\Sigmazero$ satisfying 
    \[|\tau_0(a_{-\delta/2}x) - \tau_0( a_{\delta/2}x)| = \delta\]
    for a suitable small parameter $c/2>\delta>0$
    and concluded that $\L_0$ is tangent to a geodesic lamination. %
    
    There is a tight map in every homotopy class \cite{DU:circle,GK:stretch}, and the intersection of the maximally stretched set (the set of points maximizing the local Lipschitz constant) over all homotopic tight maps is a non-empty geodesic lamination $\lambda_0 \subset \Sigmazero$ \cite{GK:stretch}.
    By \cite[Prop. 9.4]{GK:stretch} (also \Cref{L chain recurrent}) we know that $\lambda_0$ is chain recurrent 
    and that we can find a $\tau_0$ satisfying  $\Stretch (\tau_0) = \lambda_0$, and from now on, we assume that this is the case.
    In particular,
    \begin{equation}\label{eqn: L is Tlambda}
        \L_0 = \T^1\lambda_0
    \end{equation}
    holds.

    Note that $\lambda_0$ is oriented (by a choice of orientation on $\R/c\Z$).
    We use $\T^1_+\lambda_0$ to denote the points tangent to $\lambda_0$ in the positive direction and define $\T^1_-\lambda_0$ analogously. 
    Sometimes, we implicitly identify $\lambda_0$ with $\T^1_+\lambda_0$, which induces an $A$ action on $\lambda_0$.
    All of the same remarks apply to $\lambda = \pi_\Z\inverse (\lambda_0)\subset \Sigma$.
    
    Recall that $\Qm$ denotes the set of quasi-minimizing points in $\T^1\Sigma$ satisfying \eqref{eqn: qm}, and $\Qm = \Qm_+ \sqcup \Qm_-$, where $x \in \Qm_\pm$ means that $\tau(a_tx) \to \pm \infty$, as $t\to \infty$.
    The $\omega$-limit set mod $\Z$ of $\Qm$ in $\T^1\Sigma$ is the set 
    \[\Qm_\omega = \{x\in \T^1\Sigma: \exists y \in \Qm \text{ such that } \pi_\Z(a_ty) \text{ accumulates onto } \pi_\Z(x)\}.\]
    For our choice of $\tau_0$ satisfying \eqref{eqn: L is Tlambda}, by Theorem 1.4 of \cite{FLM} we have 
    \[\T^1\lambda =  \Qm_\omega.\]
	These equalities illustrate the relevance of the geometry of tight maps in our investigation of non-maximal horocycle orbit closures.

	\subsection{Sub-additive property of $\Zxy{x}{y}$}
    Recall the following definition from the introduction
    \begin{definition}
		For $x,y\in \T^1_+\lambda$ with $\tau(x) = \tau(y)$ define
		\[
		\Zxy xy = \{t: a_t y \in \overline{N x}\}. 
		\]
	\end{definition}
    Suppose $s \in \Zxy zy$ and $t\in \Zxy xz$.  
    Then $a_tz \in \overline{Nx}$, and so $a_t \overline{Nz} = \overline{Na_tz} \subset \overline {Nx}$.  In particular, $a_{t+s} y \in \overline {Nx}$, which proves the following useful property:
    \begin{equation}\label{eqn: subcontainment for Zxy}
        \Zxy zy + \Zxy xz \subset \Zxy xy.
    \end{equation}

	The above property implies in particular that $\Zxy xx$ is a semigroup, which we refer to as the \emph{recurrence semigroup} of $x$. It is is also the semigroup of sub-invariance of the associated horocycle orbit closure, that is,
	\[ \Zxy xx = \{t \in \R : a_t \overline{N x} \subseteq \overline{N x} \}. \]
	See \cite[\S7]{FLM} for more details.\footnote{In \cite{FLM} we used the notation $\Delta_x$ and considered it as a sub-semigroup of $A$ (or rather the centralizer of $A$, in higher dimensions).}

	\subsection{Terminology for asymptotic relations}
	
    We say that points $x$ and $y$ are \emph{forward asymptotic} if 
	\[ d(a_tx,a_ty)\to 0 \quad \text{as}\quad t\to\infty, \]
    i.e., $Nx = Ny$.
    
	Given a point $x \in \T^1\Sigma$ and a set $E \subset \T^1\Sigma$, we say that $x$ (also $A_+x$ or $Ax$) is \emph{forward asymptotic} to $E$ or \emph{asymptotic in forward time} to $E$ if for any $\varepsilon > 0$, there is a $T \in \R$ such that the geodesic ray $A_{[T,\infty)}x$ is contained in the $\varepsilon$-neighborhood of $E$. 

    Similarly, lines $Ax$ and $Ay$ (or rays $A_+x$ and $A_+y$) are \emph{forward asymptotic} or \emph{asymptotic in forward time} if there exists $b \in \R$ such that points $a_b x$ and $y$ are forward asymptotic, i.e., $ANx = ANy$.

    Finally, lines $Ax$ and $Ay$ are \emph{backward asymptotic} or \emph{asymptotic in backward time} if $A(-x)$ and $A(-y)$ are asymptotic in forward time, where $-: \T^1\Sigma \to \T^1\Sigma$ is the fiberwise antipodal involution.

    \subsection{Note on higher dimensions} \label{subsec: note about higher dimensions}
    
If the dimension of $\Sigma$ is greater than two, then the action of $N$ is defined only in the frame bundle of $\Sigma$, namely $G/\Gamma$. However, the horospheres themselves make sense as the leaves of a foliation of $\T^1\Sigma$ (namely the strong stable manifolds of the geodesic flow). Because most of this paper deals with dimension 2, we mostly elide this distinction. But to convert any of our discussions to higher dimension one can simply replace an orbit $Nx$ with ``the horospherical leaf containing $x$", and an expression like $y=nx, n\in N$ with ``$y$ is a point in the horospherical leaf containing $x$". This is relevant in sections 3 and 4, which can be carried out in any dimension.

	\section{Slack of paths}\label{Sec:Slack of Paths}

         As touched on in the introduction, our results rely on an analysis of the ``efficiency'' of different quasi-minimizing rays in $\T^1\Sigma$. In this section we make this notion precise and develop a few basic properties which will be used throughout the paper.
    
The {\em slack} of a path in $\Sigma$ measures the gap between its length and the
$\tau$-difference between its endpoints. For technical reasons we consider paths in
$\T^1\Sigma$ as well as $\Sigma$: 
    
	\begin{definition}\label{def: slack}

		Let $ \alpha: [a,b] \to \Sigma $ be a rectifiable curve. We define the \emph{slack of $ \alpha $} to be
		\[ \s(\alpha) = \mathrm{length}(\alpha) -
                (\tau(\alpha(b))-\tau(\alpha(a))). \]
                Similarly if $\hat\alpha:[a,b]\to\T^1\Sigma$ is rectifiable we define its slack to be the slack of its projection to $\Sigma$. 
Note that $\s$ is non-negative and additive under concatenation of paths, so if $I\subset \R$ is connected and $\alpha: I \to \T^1\Sigma $, we can define \[\s(\alpha) = \lim_{T \to \infty} \s\left(\alpha|_{I\cap [-T,T]}\right)
=\sup_{T >0} \s\left(\alpha|_{I\cap [-T,T]}\right).\] 
	\end{definition}
(We can define $\mathscr{S}_-$ by replacing $\tau$ with $-\tau$, and obtain a similar
discussion.)

If $\alpha$ is a geodesic flow line, of the form $A_{[s,t]}z$, we note that $\s(\alpha)$
is just $(t-s) - (\tau(a_t z)-\tau(a_s z))$. Using our choice of $\tau_0$ satisfying $\Stretch(\tau_0) = \lambda_0$, we have that 
\begin{equation}\label{eqn: zero slack in lambda}
\s(\alpha) = 0 \text{ if and only if } \alpha\subset
\T^1_+\lambda.  
\end{equation}

The following elementary consequence states that geodesic arcs that are not too close to the stretch lamination $\lambda_0$ have a definite amount of slack. It will be used in several places and we point out that it applies for $\Sigma_0$ a closed hyperbolic manifold of any dimension.

\begin{lemma}\label{small slack close to L}
Let $\tau_0:\Sigma_0\to\R/c\Z$ be a tight map for which $\lambda_0([\tau_0])=\Stretch(\tau_0)$, and
fix $b'>b>0$. For each $\delta$ there exists $\ep$ such that, if $\alpha \subset \T^1\Sigma$ is an oriented
geodesic arc whose length is in $[b,b']$ and $\s(\alpha) < \ep$, then $\alpha$ is in a
$\delta$-neighborhood of $\T^1_+\lambda_0$. 
\end{lemma}

\begin{proof} 
Suppose the statement fails, then there is $\delta>0$ and a sequence $\alpha_m$ with $\s(\alpha_m)\to 0$
such that $\alpha_m$ are not contained in a $\delta$-neighborhood of $\T^1_+\lambda_0$. A subsequence
converges to an arc $\alpha$ of zero slack, and because $\lambda_0$ is the full stretch locus of
$\tau_0$, this means that $\alpha$ is in $\T^1_+\lambda_0$, a contradiction. 
\end{proof}

\subsection*{Slack and orbit closures}
This lemma indicates the basic quantitative connection between slack and and $N$-orbit
closures. It relates limits of slack values with the sets $\Zxy{x}{y}$ defined in the introduction.

	\begin{lemma}\label{lem:slack gives points in Nxbar}
		For any $ x, y \in \T^1_+\lambda$ with $\tau(x) = \tau(y)$  and $ t \geq 0 $, we have $ a_t y \in \overline{Nx}$ if and only if there exists $ y_m \to y $
such that $Ay_m$ is asymptotic to $Ax$ in forward time, and $ \s(\alpha_m) \to t $. 
	\end{lemma}

        \begin{proof}          We first claim: If $x\in \T^1_+\lambda$ and $n\in N$ then
          \begin{equation}\label{eqn: slack and tau}
            \s(A_+ nx) = \tau(nx) - \tau(x).
          \end{equation}
          To prove this, since $a_snx$ and $a_sx$ are asymptotic as $s\to +\infty$, and
          $\tau$ is 1-Lipschitz, we have
          \begin{align*}
            \s(A_+ nx) &= \lim_{s\to\infty} s - \tau(a_snx) + \tau(nx) \\
            &= \lim_{s\to\infty} s - \tau(a_sx) + \tau(x) - \tau(x) + \tau(nx) \\
            &= \s(A_+ x) + \tau(nx)-\tau(x).
          \end{align*}
          Since $x\in \T^1_+\lambda$ we have $\s(A_+ x)=0$, which gives (\ref{eqn: slack and tau}). 

          \medskip
          
          Now if $a_ty \in \overline{Nx}$, there are $n_m\in N$ such that $n_m x \to a_t
          y$. Let $y_m = a_{-t} n_m x = n'_m a_{-t}x$  (where $n'_m\in N$ also). Then $y_m
          \to y$ and applying (\ref{eqn: slack and tau}) we have
               $$\s(A_+ y_m) =          \tau(y_m) - \tau(a_{-t}x). $$
This converges to $\tau(y) - \tau(a_{-t}x) = t$, 
          since $\tau(y)=\tau(x)$ and $x\in \T^1_+\lambda$. 
This gives one direction. 

          Conversely, suppose that $y_m\to y$ and $\s(A_+ y_m) \to t$, where $Ay_m$ is forward
          asymptotic to $Ax$. This means there is $s_m\in \R$ and $n_m\in N$ such that
          $y_m = n_m a_{s_m} x$. Again by (\ref{eqn: slack and tau}), we have
\begin{align*}
          \s(A_+ y_m) &= \tau(y_m) - \tau(a_{s_m} x) \\
          &= \tau(y_m) - \tau(x) - s_m. 
\end{align*}
          Since the left hand side converges to $t$ and since $\tau(y_m) \to \tau(y) = \tau(x)$, we
          conclude $-s_m \to t$. But this means that $ a_{-s_m} y_m \to a_t y$, and since
          $$ a_{-s_m} y_m = a_{-s_m} n_m a_{s_m} x  = n'_m x $$
          for some $n'_m\in N$, we have  $a_ty \in \overline{Nx}$. This gives the other  direction. 
        \end{proof}

\subsection*{Slack and the Bruhat decomposition}
\mbox{}

Suppose that $x,y,z\in \T^1\Sigma$ such that $Az\in \Axy xy$ -- that is, $Az$ is asymptotic to $Ax$ in forward time and $Ay$ in backward time. 
Up to the action of $\pi_1\Sigma_0$ there is a unique triple of lifts $\hat x, \hat y, \hat z$ to $\T^1\Hplane$ so that $A\hat z$ is asymptotic to $A\hat x$ in forward time and $A\hat y$ in backward time (starting with a lift of $Az$, for sufficiently large $T$ lift an arc $A_{[0,T]}z$ together with a path to $Ax$ shorter than the injectivity radius; and do the same for an arc $A_{[-T,0]}z$ and $Ay$).

We can obtain $\hat x$ from $\hat y$ by an expression like this: 
\begin{equation}\label{bruhat xy} \hat x = n a_t u \hat y
\end{equation}
where $n\in N$, $a_t\in A$ and $u\in U$ ($U$ being the unstable horospherical subgroup). Geometrically, following Figure \ref{fig: slack intro}, we are moving $\hat y$ along its unstable horocycle until we get to $Az$, moving along the geodesic $Az$, and then along a stable horocycle till we get to $\hat x$. Algebraically, we could get this by identifying $\hat x$ and $\hat y$ with $g$ and $h$ in $G$, respectively, and then $na_t u$ is the Bruhat decomposition of $gh^{-1}$.

We denote the  $A$ part of (\ref{bruhat xy}) by $\delta(gh^{-1})$, so that $t=\log(\delta(gh^{-1}))$. %
We have the following relationship between this quantity and our slack: 

    \begin{lemma}\label{lem: slack computes delta projection}
        With notation as above, suppose that $x, y \in \T^1_+\lambda$ and $\tau(x) = \tau(y)$, and $Az\in \Axy xy $.  Then 
        $$\s(Az) = \log(\delta(gh\inverse))\in\R_{\ge0}.$$
    \end{lemma}
    \begin{proof}
        The proof is equivalent to the proof of \cite[Lemma 9.4]{FLM} with
        small changes of notation.

        From the definitions, we have $z = n a_s x$ and $z = ua_t y$ for some $s, t \in \R$, $n\in N$ and $u \in U$.
        We see then that $gh\inverse = a_{-s}n\inverse ua_t$, so that $\log(\delta(gh\inverse)) = t-s$.

        Using (\ref{eqn: slack and tau}) from the proof of Lemma \ref{lem:slack gives
          points in Nxbar}, we have
        $$
        \s(A_+ z) = \tau(z) - \tau(a_s x) = \tau(z)-\tau(x) - s. 
        $$
        Reversing flow direction, the roles of $U$ and $N$, and the sign of $\tau$, the same identity yields
        $$
        \s(A_- z) = \tau(a_t y) - \tau(z) = \tau(y) -\tau(z) + t.
        $$
        Putting these together, and using $\tau(x)=\tau(y)$, we have
        \[\s(Az) =         \s(A_+ z)  +         \s(A_- z) = t-s.\]
    \end{proof}

    \begin{remarks*}
        \begin{enumerate}[leftmargin=*]
        	\item The notion of slack depends on our choice of $1$-Lipschitz tight map $\tau$, as does the requirement that $ \tau(x) = \tau(y)$. Together, these dependencies ``cancel each other out.''
        	\item Our definition of slack coincides with Sarig's notion of Busemann cocycle, \cite[\S4.1]{Sarig2010}, when considering two geodesics which are both backward and forward asymptotic in $\Sigma$. Sarig used this cocycle  to study the quasi-invariance properties of horocycle-flow-invariant Radon measures on $\T^1\Sigma$.
        \end{enumerate}
    \end{remarks*}

\subsection*{Estimating slack of broken paths}
\mbox{}
Let $\Sigma$ be a hyperbolic manifold of any dimension $m\ge 2$, and $\tau:\Sigma\to\R$ a 1-Lipschitz function with associated slack $\s$.  This lemma shows that a broken geodesic with ``small'' total jumps between its segments has a geodesic representative whose slack is estimated by the sum of the slacks of the pieces. For a smooth path $\alpha$ in $\Sigma$ we let $\T^1\alpha$ denote its tangent lift. 

	\begin{lemma}\label{lem:epsilon chain slack}
		For all $ c>0 $ there exist constants $\kappa_c,\varepsilon_0>0$ such that the following holds for all $ 0 < \varepsilon < \varepsilon_0 $. Let $ \alpha_i:[a_i,b_i] \to \Sigma$ for $ i=1,...,n $ be a sequence of geodesic arcs, each of length greater or equal to $c$, and satisfying
		\begin{equation}\label{eqn: sum jumps}
		\sum_{i=1}^{n-1} d_{\T^1\Sigma}(\T^1\alpha_i(b_i),\T^1\alpha_{i+1}(a_{i+1})) < \varepsilon
		\end{equation}
  and let $\bar\alpha$ denote an arc obtained from $\cup\alpha_i$ by joining each endpoint $\alpha_i(b_i)$ to $\alpha_{i+1}(a_{i+1})$ using arcs whose total length is less than $\varepsilon$. Then  there exists a geodesic arc $ \alpha $  homotopic rel endpoints to $\bar \alpha$ and satisfying
		\[ \left|\s(\alpha)- \sum_{i=1}^n \s(\alpha_i) \right| < \kappa_c \cdot \varepsilon. \]
		Moreover,  the Hausdorff distance between $ \alpha $ and $ \bar \alpha $ is smaller than $ \kappa_c \varepsilon $.
		
		The claim further holds with $ \alpha_n:[a_n,\infty) \to \Sigma
		$ and where $ \alpha $ is a geodesic ray from $ \alpha_1(a_1) $ 
  which is forward-asymptotic to $\alpha_n$; and similarly with
		$\alpha_1:(-\infty,b_1]\to\Sigma$. 
\end{lemma}

    \begin{proof}
By considering a lift of  $\bar\alpha$ to  $\bH^m$, and pulling back
the function $\tau$, we can reduce to the case that $\Sigma=\bH^m$.  In this case we choose $\alpha$ to be the unique geodesic joining the endpoints of $\bar\alpha$.

We may reduce to the case that $ \alpha_i(b_i)=\alpha_{i+1}(a_{i+1})$: 
we do this by moving the endpoints slightly in $\bH^m$, and the lower bound
$c$ on the lengths of the $\alpha_i$ maintains the control on the tangent vectors and hence on the sum (\ref{eqn: sum
  jumps}). In particular, 
letting $\theta_i\ge 0$ be the angle between $\T^1\alpha_i(b_i)$ and $\T^1\alpha_{i+1}(a_{i+1})$, 
for some $\kappa_1 = \kappa_1(c)$ we have
\begin{equation}\label{eqn: sum theta}
  \sum\theta_i \le \kappa_1\varepsilon.
\end{equation}

It is well-known (see e.g. \cite[Thm 4.2.10]{CEG})
that there is an $\ep_0$ and $\kappa_2$ so that the Hausdorff distance between $\alpha$ and $\cup\alpha_i$ is at most $\kappa_2\ep$. 

Let $r_i = d(\alpha_i(b_i),\alpha)=d(\alpha_{i+1}(a_{i+1}),\alpha)$. We
next claim that 

\begin{equation}\label{eqn: sum ri}
  \sum r_i < \kappa_3\sum \theta_i
\end{equation}
    for $\kappa_3=\kappa_3(c)$. 

Let $h:D\to \bH^m$ be a ``triangulated disk'' spanning the loop $\gamma = \alpha_1*\cdots
 *\alpha_n*\alpha^{-1}$. That is, choose a triangulation of a disk $D$ whose
 vertices are points $x_0,\ldots,x_n$ on the boundary, and choose $h$ 
 so that the segment between $x_{i-1}$ and $x_{i}$ maps to $\alpha_i$ for $i=1,\ldots,n$,
 the segment between $x_n$ and $x_0$ maps to $\alpha^{-1}$, and each triangle maps to a geodesic triangle
 in $\bH^m$.  See \Cref{fig: disk D}. Pulling back the hyperbolic metric via $h$ we obtain a hyperbolic metric on
 $D$ whose boundary is polygonal, and such that the angle subtended at $x_i$ is at least
 $\pi-\theta_i$ for $i=1,\ldots,n-1$, and is non-negative at $x_0$ and $x_n$. (This is
 obtained by considering the spherical distance between the incoming and outgoing tangent
 vector at each vertex -- see  \Cref{fig: angle}). 

 \begin{figure}[htb]
        \centering
        \includegraphics[width=1\linewidth]{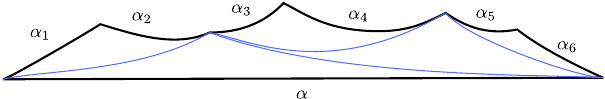}
        \caption{The triangulated disk in $\bH^m$ (which need not be embedded).}
        \label{fig: disk D}
    \end{figure}

The Gauss-Bonnet theorem tells us 
    $${\rm area}(D) \le \sum\theta_i.$$
    In the other direction, each vertex $\alpha_i(b_i)$ is distance $r_i$ from
    $\alpha$.
This means there is a triangle
    in $D$ with base on the $\alpha$ side of length a definite fraction of $c$ and height
    at least $r_i$, and all these triangles are disjoint. Summing the areas of these
    triangles we get
    $${\rm area}(D)\ge \sum \kappa_4 r_i,$$
for    $\kappa_4 = \kappa_4(\ep_0,c)$.     The claim (\ref{eqn: sum ri}) follows. 

\begin{figure}[htb]
        \centering
        \includegraphics{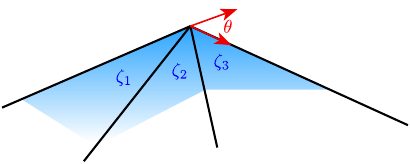}
        \caption{At a corner of the polygonal disk, the sum of the internal angles $\sum\zeta_j$ is at least $\pi-\theta$.}
        \label{fig: angle}
    \end{figure}

    Now to finish the lemma, let $y_i\in \alpha$ be the  closest point to $\alpha_i(b_i)$. Since $\tau$ is 1-Lipschitz we have $|\tau(y_i)-\tau(\alpha_i(b_i))| \le r_i$.  Letting $\beta_i$ be the segment of $\alpha$ between $y_{i-1}$ and $y_i$, we see that 
    $$|\s(\beta_i) - \s(\alpha_i)| < 2(r_{i-1}+r_i).$$
    Since slack along a path is additive, we obtain the conclusion of the lemma (in
    the finite case) by adding over the $\beta_i$ and using (\ref{eqn: sum theta}) and (\ref{eqn: sum ri}).

        The proof when one or both of $\alpha_1$ or $\alpha_n$ are rays is similar, or can be obtained from the finite case by taking limits. 
    \end{proof}

	\section{Finite Lamination case}\label{sec: finite lamination}
    In this section, we treat the case that $\lambda_0$ has finitely many leaves. 
    Our main Theorems \ref{thm:structure of Zxy} and \ref{thm: countable filtration Zxy} relate the slacks of edge-paths in the slack graph $\G$ with the structure of $N$-orbit closures.  %
    Our arguments are not dimension specific enabling us to establish Theorem \ref{thm: d-manifolds intro} regarding higher dimensional hyperbolic manifolds (see also Corollaries \ref{cor: multi-curve structure} and \ref{cor: depth omega} as well as \S\ref{subsec: note about higher dimensions}).

    In \S\S\ref{subsec: infinite cr leaf}--\ref{subsec: isolated multi-curve}, we explain how the techniques of this section get us closer toward understanding $N$-orbit closures for general $\lambda_0$ with infinite leaves.
    
    \subsection{$\lambda_0$ is a multi-curve}\label{subsec: multi-curve}
    We begin with the case that $\lambda_0$ is a disjoint union of finitely many simple closed geodesics, and assume that $\tau_0$ has been chosen, as in \S\ref{Subsec:Tight map summary of notations}, such that $\Stretch(\tau_0) = \lambda_0$ holds. %
    \medskip

    Consider the first return $\sigma : \tau_0\inverse(0) \cap \T^1_+\lambda_0 \to \tau_0\inverse(0)\cap \T^1_+\lambda_0$ under the geodesic flow in $\T^1\Sigmazero$.  Since $|\tau_0\inverse(0)\cap \T^1_+\lambda_0|$ is finite, and $\sigma$ is homeomorphic (i.e., bijective), some power $d$ of $\sigma$ is the identity.
    We replace $\Sigma_0$ with the regular $\Z/d\Z$-cover corresponding to the subgroup $dc\Z\subset c\Z = \pi_1(\R/c\Z)$, $\tau_0$ with the lifted map to $\R/dc\Z$, and the $\Z$-action with the action of $d\Z\cong \Z$.
    Abusing notation, all objects are labeled as before.
    \medskip
    
    The components of $\T^1_+\lambda$ are of the form $A x$, where $x\in \tau\inverse(0)\cap \T^1_+\lambda$.
    Let $\Axy xy $ denote the bi-infinite geodesics that are asymptotic to $A y$ in backward time and to $A x$ in forward time. 
    Let $\mathcal G$ be the directed graph satisfying
    \begin{itemize}
        \item the vertex set  $V(\mathcal G)$ is $\tau\inverse(0)\cap \T^1_+\lambda$. 
        \item the set of directed edges from $y$ to $x \in V(\mathcal G)$ is $\Axy xy$.
    \end{itemize}

    The \emph{fundamental semi-groupoid} $\Pi (\mathcal G)$ is the set of finite directed edge-paths in $\mathcal G$.
    The notion of slack from \S\ref{Sec:Slack of Paths} extends \[\s: \Pi(\mathcal G) \to \R_{\ge0}\]
    by the rule $\s(e_1\cdot e_2) = \s(e_1) + \s(e_2)$.
    For $x$ and $y \in V(\mathcal G)$, let $\Hom_{\mathcal G} (y,x)$ be those directed edge-paths from $y$ to $x$.

    \begin{theorem}\label{thm:structure of Zxy}
    When $\lambda_0$ is a multicurve, 
        $$\s(\Hom_{\mathcal G}(y, x))) = \Zxy xy.$$  
        Consequently, $\Zxy xy$ is countable.
    \end{theorem}

    Applying our previous work \cite{FLM}, we obtain a complete description of horocycle orbit closures. %
    \begin{cor}\label{cor: multi-curve structure}
        Whenever $\lambda_0$ is a multi-curve and $z\in \Qm_+$, then $\overline{Nz}$ is a countable union of horocycle orbits, hence has Hausdorff dimension $1$.
    \end{cor}
    \begin{proof}[Proof of \Cref{cor: multi-curve structure}]
    Every quasi-minimizing ray exiting the ``$+$'' end of $\Sigma$ is asymptotic to a leaf of $\T^1_+\lambda$, so
    \[\Qm_+ = \cup_{y\in V(\mathcal G)} P y, \]
    where $P=AN$.
    Then $\overline {Nz}$ is the union of its intersections with each of the finitely many $Py$.

    Assume that $Az$ is asymptotic in forward time to $Ax$.
    Since $\beta_+(x) = 0$ and $\beta_+$ is $N$-invariant, we have $Nz\cap Ax = \{a_{\beta_+(z)}x\}$, or more generally, $Nz\cap Ax = a_{\beta_+(z) - \beta_+(x)}x$.
    Using the definition of $\Zxy xy$ we obtain 
    \begin{equation}\label{eqn:countable horocycle orbit closure}
        \overline {Nz} = a_{\beta_+(z)}\overline{Nx} = a_{\beta_+(z)}\bigcup_{y \in V(\mathcal G)}N A_{\Zxy xy }y.
    \end{equation}
    Theorem \ref{thm:structure of Zxy} tells us that $A_{\Zxy xy }y$ is countable, hence \eqref{eqn:countable horocycle orbit closure} exhibits $\overline {Nz}$ as a countable union of horocycle orbits, which have Hausdorff dimension $1$, and proves that the Hausdorff dimension of $\overline {Nz}$ is $1$.
    \end{proof}

    \begin{proof}[Proof of Theorem \ref{thm:structure of Zxy}]
        We show that $\s(\Hom_{\mathcal G}(y,x))\subset \Zxy xy$ in two steps.
        First, we consider edgepaths of length one, that is, suppose that $Az \in \Axy xy$.
        By assumption, the first return $\sigma$ from  $F_0\cap \T^1_+\lambda_0$ to itself is the identity, so $k.Az\in \Axy xy$ for all $k \in \Z$.
        
        Observe that there are $y_k \in k.Az$ tending to $y$ as $k\to \infty$.
        Indeed, since $Az$ is asymptotic to $Ay$ in backward time, there is a $u \in U$ such that $uy\in Az$.
        Thus \[a_{-kc}u a_{kc} a_{-kc}y\in Az. \]
        With $u_k = a_{-kc}u a_{kc}$, we have that
         \[k.u_k a_{-kc}y\in k.Az.\]
        Since $k.a_{-kc}y = y$ and the $\Z$ action commutes with $U$, $y_k = u_k y \in k.Az$.  Since $\|u_k\|\to 0$, we get $y_k \to y$ as $k \to \infty$.

        Let  $\epsilon>0$ be given. %
        Additivity of slack, the fact that the slack of a path contained in $Ay$ is 0, and continuity  give that the geodesic ray
        \[\alpha_k: t \mapsto a_ty_k, ~ t\ge 0,\] has 
        \[|\s(\alpha_k) - \s(k.Az)|<\epsilon,\]
        for $k$ large enough.
        Note that $\s(k.Az) = \s(Az)$ for all $k$.
        Since $\epsilon$ was arbitrary and $\alpha_k$ is asymptotic to $Ax$ in forward time, Lemma \ref{lem:slack gives points in Nxbar} gives that $\s(Az) \in \Zxy xy$.

        In the second step, suppose that $\alpha_1 \cdots \alpha_n \in \Hom_{\mathcal G} (y,x)$.
        From the definition of $\s$, we have 
        \[ \s(\alpha_1\cdots \alpha_n) = \sum \s(\alpha_i).\]
        By the first step, $\s(\alpha_i) \in \Zxy{x_{i+1}}{x_{i}}$, where $x_1 = y$ and $x_n = x$. Using \eqref{eqn: subcontainment for Zxy} and induction, we find that
        \[\Zxy {x_2}{x_1} + \dots + \Zxy {x_n}{x_{n-1}} \subset \Zxy xy,\]
        hence conclude that
        \[\s(\alpha_1\cdots \alpha_n) \in \Zxy xy.\]
        This completes the proof that $\s(\Hom_{\mathcal G}(y,x))\subset \Zxy xy$.
        \medskip
	
        \medskip
        Now we show that $\s(\Hom_{\mathcal G}(y,x))\supset \Zxy xy$.  That is, we show that whenever $a_Ty \in \overline {Nx}$, then there is a  finite edgepath $\underline\alpha\in \Hom_ {\mathcal G}(y,x)$ such that $\s(\underline \alpha) = T$.

        Suppose then that $a_Ty \in \overline {Nx}$, and find a sequence $n_m \in N$ such that $n_m x \to a_Ty$ as $m\to \infty$. Furthermore, we can choose $n_m$ such that $An_m x$ is asymptotic to $Ay$ in backward time (this is an application of the Bruhat decomposition in a small neighborhood of the identity). %
        In other words, $An_m x\in \Axy xy$.
        Let $\alpha_m$ denote the path $t\mapsto a_tn_m x$.
        
        As in the proof of Lemma \ref{lem:slack gives points in Nxbar}, $\s(\alpha_m) \to T$, as $m\to \infty$.  What we have left to show is that $T = \s(\alpha^1)+ \cdots +\s(\alpha^i)$, where $\alpha^1 \cdots \alpha^i$ is a directed edgepath from $y$ to $x$ in $\mathcal G$.

        The finitely many $A$-orbits constituting $\T^1_+\lambda$ are uniformly isolated. 
        Find a positive $\epsilon_0$ smaller than the injectivity radius of $\T^1\Sigma$  such that the distance in $\T^1\Sigma$ between distinct components of $\T^1_+\lambda$ (lifted to $\T^1\mathbb H^2$) is at least $3\epsilon_0$. 
        In what follows, for a positive $\epsilon>0$, $\T^1_+\lambda^{(\epsilon)}$ denotes the $\epsilon$-neighborhood of $\T^1_+\lambda$ in $\T^1\Sigma$.

        Consider the components, listed in order along $\alpha_m$
        \[\alpha_m \setminus \T^1_+\lambda^{(\epsilon_0)} = \kappa_m^1 \cup \ldots \cup \kappa_m^{i_m}. \]
        Note that since $ \alpha_m $ is asymptotic to $ \T^1_+\lambda$ in both directions, and since the distance between the different components of $ \T^1_+\lambda^{(\varepsilon_0)} $ is at least $ \varepsilon_0 $ (and there is a shortest non-trivial loop starting and ending in any given component), then there are indeed only finitely many $ \kappa_m^j $'s for each $m$. 
        We think of $\kappa_m^j$ as an ``$\epsilon_0$-excursion'' taken by $\alpha_m$ away from $\T^1_+\lambda$.
        
        \begin{claim}\label{clm: slack bounded below}
        There is a $\delta>0$ such that $\s(\kappa_m^j)>\delta \ell(\kappa_m^j)\ge\delta\ep_0 $ for all $m$ and $j$.
        \end{claim}

	\begin{proof}[Proof of Claim \ref{clm: slack bounded below}]
 By choice of $\ep_0$, $\kappa_m^j$ has length at least $\ep_0$. We can cut it up into $\lfloor \ell(\kappa_m^j)\rfloor$ segments of length in $[\ep_0,2\ep_0]$, and apply Lemma \ref{small slack close to L} to each of them, obtaining the desired inequality.
	\end{proof}
	
	Additivity of the slack and Claim \ref{clm: slack bounded below} produces a uniform upper bound on $i_m$, the number of $\epsilon_0$-excursions, and their total length.
	To see this, observe that 
	\[T+1 \ge \s(\alpha_m) > \sum_{j = 1}^{i_m} \s(\kappa_m^j)\ge \sum_{j = 1}^{i_m} \delta\ell(\kappa_m^j) \ge \delta\ep_0 i_m\]
	holds for $m$ large enough.
	Thus, up to taking a subsequence, we may assume that $i_m = i_0$ is constant and $(T+1)/\delta \ge \sum \ell(\kappa_m^j)$.
	 
        Choose points $p_m^j \in \kappa_m^j$ for all $m$ and $j =1 , ..., i_0$.
        By compactness of $\T^1\Sigmazero$, we may find $p_j \in \T^1\Sigmazero$ and a subsequence  such that $\lim_{m\to \infty} \pi_\Z(p_m^j) = \pi_\Z(p^j)\in \T^1\Sigmazero$ for all $j$. 
        After a further subsequence we may assume that for all $j,k$, either $d(p_m^j,p_m^k)$ is bounded or $d(p_m^j,p_m^k)\to\infty$.

        Boundedness of $d(p_m^j,p_m^k)$ as $m\to \infty$ is an equivalence relation on the upper indices, for which equivalence classes are intervals in $\N\cap [1, ..., i_0]$. 
        Let $1\le M\le i_0$ be the number of such equivalence classes.
        Now we choose a representative $q_m^j$ for $j = 1, ..., M$ for each equivalence class of the upper indices (so that for each $j$, $q_m^j = p_m^{k_j}$ for some $k_j$).
        Thus $d(q_m^j,q_m^k)\to\infty$ for each $1\le j< k\le M$, and $\pi_\Z(q_m^j) \to q^j$ as $m \to \infty$.

        Let $\beta^j=Aq^j$, and note that, as pointed geodesics, $\pi_\Z(\alpha_m,q^j_m)$ converges to $\pi_\Z(\beta^j,q^j)$. 

        \begin{claim}\label{clm: betas are a path}
           Each $\beta^j$ is an edge of $\G$, and $\beta^1\cdots \beta^M$ is a path in $\G$. We have
           \begin{equation}\label{eqn: slack sum inequality}
           \sum_j\s(\beta^j) = \lim_{m\to\infty}\s(\alpha_m)=T.
           \end{equation}
         \end{claim}

        We call such $(\beta^1,\ldots, \beta^M)$ a {\em geometric limit chain} for the sequence $\alpha_m$.
        
        \begin{proof}[Proof of Claim \ref{clm: betas are a path}]
        Each $j$ represents a finite collection of adjacent $\epsilon_0$-excursions in each $\alpha_m$ of bounded total length, which remain a bounded distance from each other, and are adjacent on both sides to segments contained in $\T^1_+\lambda^{(\epsilon_0)}$
        whose length goes to $\infty$ with $m$. Thus, each limiting $\beta^j$ is asymptotic at both ends to (a leaf of) $\T^1_+\lambda$ and so represents an edge of $\G$. For each adjacent pair $\beta^j,\beta^{j+1}$, $\beta^j$ is forward-asymptoic to the same component of $\T^1_+\lambda$ to which $\beta^{j+1}$ is backward-asymptotic. This makes $\beta^1\cdots\beta^M$ a path in $\G$. 

        Let $\epsilon>0$ be given.  
        We may choose $m$ large enough that $\alpha_m$ is well approximated by the finite union of $\cup \beta^j \setminus \T^1_+\lambda^{(\epsilon)}$ with jumps of size $\epsilon$ and long segments of leaves of $\T^1_+\lambda$ in between.
        Lemma \ref{lem:epsilon chain slack} now implies that the finite sum $\sum_j \s(\beta^j)$ is an $O(\epsilon)$-good approximation of $\s(\alpha_m)$. Since $\epsilon$ was arbitrary, we conclude \eqref{eqn: slack sum inequality}.
        \end{proof}
        This concludes the proof of the theorem.
    \end{proof}
 
    \subsection*{No miracles lemma} The following technical point will be useful for describing the structure of $\Zxy xy$ (Theorem \ref{thm: countable filtration Zxy}) as well as in the next subsection, when we consider the case that $\lambda_0$ has an infinite leaf (Theorem \ref{thm: chain recurrent leaf}). It says that the slack of a composition of two edges in $\Hom_{\mathcal G}(x,z)$ can be obtained as the limit of a {\em non-constant} sequence of slacks of single edges. 
    \begin{lemma}\label{lem: non-trivial accumulation}
        For any $x, z \in V(\mathcal G)$, given $\underline \alpha = \alpha^1\cdot \alpha^2 \in \Hom_{\mathcal G}(x,z)$, there exists a sequence %
        $\alpha_m \in \Axy zx$ with the properties that $ |\s(\alpha_m) - \s(\underline \alpha) | \to 0$ and $\s(\alpha_m) \not = \s(\underline \alpha)$ for all large $m$.
    \end{lemma}
        
	\begin{proof}
    The algebraic perspective will be more helpful; we follow the proof of \cite[Proposition 7.20]{FLM}.

    We have $\alpha^1 \in \Axy yx$ and $\alpha^2 \in \Axy zy$.
    There is a bijective correspondence between $\Axy yx$ and relative homotopy classes of paths in $\T^1\Sigma$ joining $Ax$ and $Ay$: for an arc $\eta$ joining $Ax$ to $Ay$, obtain $\alpha_\eta \in \Axy yx$ by dragging the initial and terminal endpoints of $\eta$ to infinity along $A_-x$ and $A_+y$, respectively.  
    Let $\eta_i$ be such that $\alpha^i = \alpha_{\eta_i}$ for $i = 1,2$. 
    We will also consider the arc $m.\eta_2$.

    Join $Ax$ to $Ay$ via $\eta_1$ and lift this simply connected $1$-complex to $G$, where the lift of $\eta_1$ joins $Ag_x$ to $Ag_y$ with $g_x, g_y\in G$ lifting $x$ and $y$, respectively.
    The slack of $\alpha_{\eta_1}$ can be computed  as $\log(\delta(g_yg_x\inverse))$, i.e, \[\text{ if  $n_1\ell_1u_1 g_x = g_y $, then $\s(\alpha_{\eta_1}) = \log(\ell_1)$}.\] See Lemma \ref{lem: slack computes delta projection} and the text preceding it.

    Similarly, the slack of $\alpha_{\eta_2}$ can be computed as $\log(\delta(g_zg_y\inverse))$, where $g_z\in G$ is the lift of $z$ determined by $g_y$ and $\eta_2$.
    Then $g_zg_y\inverse = n_2\ell_2u_2 \in NAU$, and $\s(\alpha_{\eta_2}) = \log(\ell_2)$.

    Now consider the path $\eta_1 \ast m.\eta_2$ determined by joining $\eta_1$ to $m.\eta_2$ along a segment of $Ay$.
    Take $\alpha_m =\alpha_{\eta_1\ast m.\eta_2} \in \Axy zx$. %
    There are $\gamma_m \in \Gamma$ such that \[\s(\alpha_{\eta_1 \ast m.\eta_2})=\log(\delta(g_z\gamma_mg_x\inverse))= \log(\delta(g_z\gamma_mg_y\inverse g_yg_x\inverse)).\]

    From the definitions (see also Figure \ref{fig: lifting arcs}), we have
    \begin{equation}\label{eqn: lifting arcs}
        n_2\ell_2u_2a_{mc}g_y = a_{mc}g_z\gamma_m,
    \end{equation}
    so that $g_z\gamma_m g_y\inverse = a_{-mc}n_2\ell_2u_2a_{mc}$.

    \begin{figure}
        \centering
        \includegraphics[width=1\linewidth]{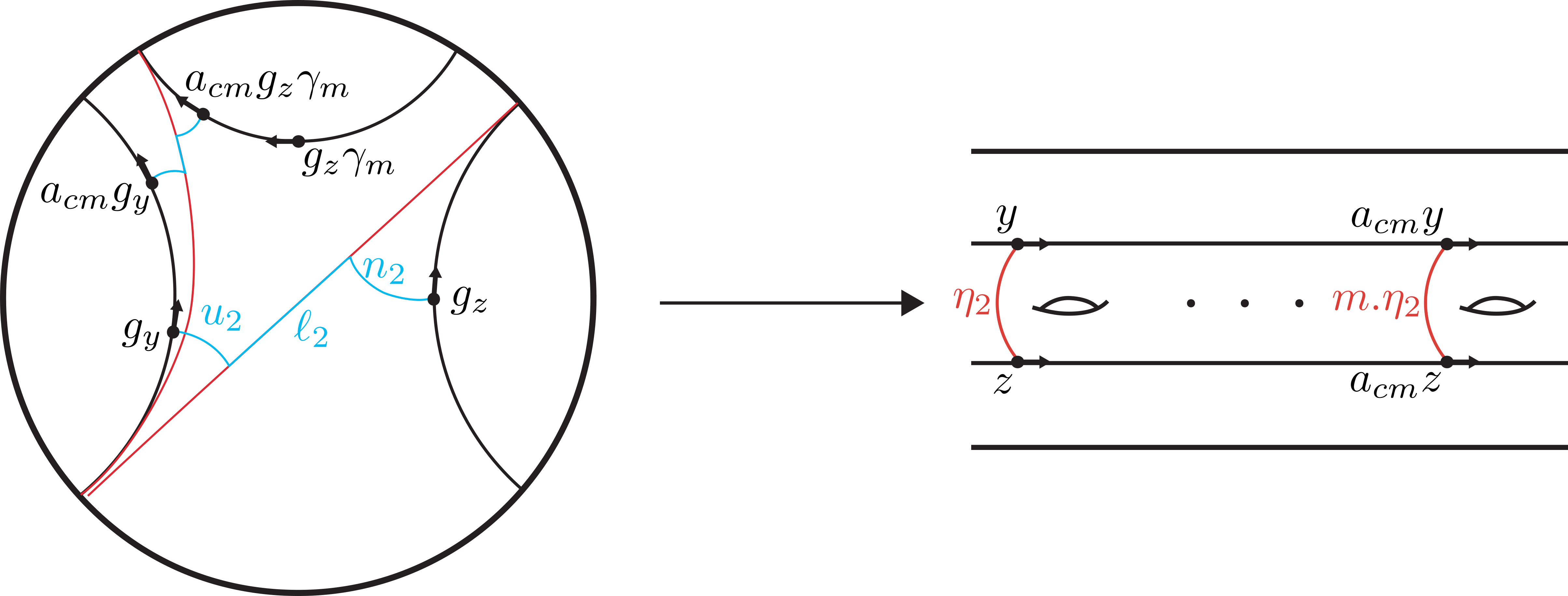}
        \caption{Lifting the arcs $m.\eta_2$ to $G$ gives \eqref{eqn: lifting arcs}}
        \label{fig: lifting arcs}
    \end{figure}

    Note that none of $n_1, n_2\in N$ and $u_1, u_2 \in U$ is the identity, because none of the lines $Ax$, $Ay$, and $Az$ is asymptotic to any other in either forward or backward time.
    \medskip

    Now we compute
    \begin{align*}
        (g_z\gamma_mg_y\inverse)( g_yg_x\inverse) & = a_{-mc}n_2\ell_2u_2a_{mc}n_1\ell_1u_1\\
        & = (a_{-mc}n_2a_{mc})\ell_2 a_{-mc}u_2a_{mc} n_1 \ell_1 u_1,
    \end{align*}
    so
    \[\delta(g_z\gamma_mg_x\inverse) = \delta(\ell_2 a_{-mc}u_2a_{mc} n_1 \ell_1).\]

    Since $a_{-mc}u_2a_{mc} \to e$, as $m\to \infty$, we can write 
    \[a_{-mc}u_2a_{mc} n_1 = n_m'\ell_m'u_m'\in NAU\]
    for $m$ large enough.
    Then 
    \[\ell_2 a_{-mc}u_2a_{mc} n_1 \ell_1 = \ell_2 n_m'\ell_m'u_m'\ell_1 = n_m'' \ell_1 \ell_2 \ell_m' u_m''.\]
    Hence $\s(\alpha_{\eta_1 \ast m.\eta_2}) = \s(\alpha_{\eta_1}) +\s(\alpha_{\eta_2}) +\log(\ell_m')$.
    
    We can now see that $ u_2, n_1 \not= e$ implies that $\ell_m' \not= e$ for all $m$ large by an explicit matrix computation; see the proof of the Claim in \cite[Proposition 7.20]{FLM} for an argument in $ \SO^+(d,1) $.
    Indeed, if $u,n \not=e$ then 
    \begin{align*}
    	a_{-t} u a_t n &= \begin{pmatrix}
    		e^{-t/2} & \\
    		& e^{t/2}
    	\end{pmatrix}\begin{pmatrix}
    		1 & x\\
    		& 1
    	\end{pmatrix}
    	\begin{pmatrix}
    		e^{t/2} & \\
    		& e^{-t/2}
    	\end{pmatrix}\begin{pmatrix}
    		1& \\
    		y & 1
    	\end{pmatrix} \\
    	&=
    	\begin{pmatrix}
    		1+e^{-t}xy & e^{-t}x \\
    		y & 1
    	\end{pmatrix},
    \end{align*}
    for $x,y\in \R\setminus \{0\}$.
    Then $\log(\delta(a_{-t} u a_t n)) = 2\log(1+e^{-t}xy)\not= 0$.\footnote{In fact, an explicit computation shows $ \log(\delta(M)) = 2\log(|M_{1,1}|) $ for any $ M \in \PSL_2(\R) $.}

    This completes the proof that $\s(\alpha_m) \not= \s(\underline \alpha)$ for all large $m$, but $\s(\alpha_m) \to \s(\underline \alpha)$.
	\end{proof}

    \subsection*{Depth}\label{subsec: Depth}
    Recall that for a set $S\subset \R$, the \emph{derived set} $S^{(1)}$ is obtained from $S$ by removing the isolated points from $S$.
    Inductively, $S^{(i)}$ is the derived set of $S^{(i-1)}$. We say that $S$ {\em has depth $d\in \N$ } if 
    $S^{(i)} \not = \emptyset$ for all $i<d$, and $ S^{(d)} = \emptyset$. We say that {\em $S$ has depth $\omega$} if $S^{(i)} \not = \emptyset$ for all $i$ and $\cap_{i \in \mathbb N} S^{(i)} = \emptyset$.
    
    For an edgepath $\underline \alpha = \alpha^1 \cdots \alpha^i \in \Hom_{\mathcal G}(x,y)$, let $\ell(\underline \alpha) = i$ denote its combinatorial length.
    Let $\Hom_{\mathcal G}^{(i)}(x,y)$ denote those $\underline \alpha$ with $\ell(\alpha) \ge i$.
    The following structural result for the shift set $\Zxy yx$ says that its accumulations are filtered by the combinatorial length of paths, via the slack map $\s$.
    It is essentially a corollary of the the proof of Theorem \ref{thm:structure of Zxy} and the technical Lemma \ref{lem: non-trivial accumulation}. 
    
    \begin{theorem}\label{thm: countable filtration Zxy}
        $({\Zxy yx})^{(i)} = \s(\Hom_{\mathcal G}^{(i+1)}(x,y))$ for all  $i \geq 0$.
    \end{theorem}

    \begin{proof}
        
        Note that for $i=0$ this is just $\Zxy yx = \s(\Hom_\G(x,y))$, which is Theorem \ref{thm:structure of Zxy}. 
        For readability, denote $Z^i = ({\Zxy yx})^{(i)}$ and $H^i=\s(\Hom_{\mathcal G}^{(i)}(x,y))$ for the rest of the proof.

        We first argue that every point of $H^{i+1}$ is an accumulation point of $H^i$. Indeed, for $i=1$, Lemma \ref{lem: non-trivial accumulation} gives for any $\alpha\cdot\beta\in\Hom_{\mathcal G}^{(2)}(x,y) $ a sequence $\gamma_n\in \Hom_\G^{(1)}(x,y)$ such that $\s(\gamma_n) \to \s(\alpha\cdot\beta)$ nontrivially (not eventually constant). Now for $\ell(\underline\alpha)=i+1>2$ we just apply Lemma \ref{lem: non-trivial accumulation} to two successive edges in $\underline\alpha$. 
        
        This implies that no point of $H^2$ is isolated in $H^1=Z$, so $H^2\subset Z^1$. Arguing by induction we find
        \begin{equation}\label{eqn: H in Z}
        H^{i+1} \subset Z^i,
        \end{equation}
        where the inductive step is that no point of $H^{i+2}$ is isolated in $H^{i+1}$, and hence in $Z^i$, so that $H^{i+2}$  must be in $Z^{i+1}$.

        To prove the other inclusion, we need the following.
        \begin{claim}\label{clm: nontrivial accumulation}
        For each $i\ge 1$, $H^i\setminus H^{i+1}$ is isolated in $H^i$.
        \end{claim}
	
	\begin{proof}[Proof of the claim]
    We need to prove that if $\ell(\underline \alpha) = i$ and $\s(\underline \alpha) \not = \s(\underline \beta)$ for all $\underline \beta$ with $\ell(\underline \beta) >i$, then  $\s(\underline \alpha)$ is isolated in $\s(\Hom_{\mathcal G}^{(i)}(x,y))$.
    Let $\underline \alpha_m$ be a sequence of paths with $\ell(\underline \alpha_m) \ge i$ and $\s(\underline \alpha_m) \to \s(\underline \alpha)$. 
    Using Claim \ref{clm: slack bounded below}, we see that $\ell(\underline \alpha_m)$ is uniformly bounded from above.  Up to taking a subsequence, we can assume that $\ell(\underline \alpha_m) = i_0\ge i$.
    Thus $\underline \alpha_m = \alpha_{m,1}\cdots \alpha_{m,i_0}$. 

    After restricting to a subsequence, the proof of Theorem \ref{thm:structure of Zxy} gives a geometric limit chain $\beta_k^1\cdots \beta_k^{M_k}$ for each  sequence $(\alpha_{m,k})_m$, which we concatenate to a path $\underline\gamma = \beta_1^1\cdot\beta_1^2\cdots \beta_{i_0}^{M_{i_0}}$  in $\G$ satisfying $\s(\underline\gamma) = \s(\underline\alpha)$. 

    If the length of $\underline\gamma$ is bigger than $i$ then we have contradicted the
    hypothesis that $\s(\underline\alpha) \notin H^{i+1}$.  Thus $\ell(\underline\gamma)=i$, which means that $\ell(\underline\alpha_m)\equiv i$ and each geometric limit chain for $(\alpha_{m,k})_m$ is composed of a single element $\beta_k^1$. 

    We claim now that in fact $\alpha_{m,k} = \beta_k^1$ (up to the $\Z$ action) for  large enough $m$.  To see this, decompose $\beta_k^1$ into a compact interval $K$
    and two rays contained in a regular neighborhood of $\T^1_+\lambda$. For large $m$, $\alpha_{m,k}$ contains an interval $K_m$ following $K$ very closely, and the rest of $\alpha_{m,k}$ must consist of rays in the regular neighborhood of $\T^1_+\lambda$, because any exit from that neighborhood would lead to a second component of the geometric limit chain.

    We conclude that the subsequence we've extracted from $\underline \alpha_m$ is eventually constant. In particular, for every sequence
    in $\s(\Hom_{\mathcal G}^{(i)}(x,y))$ converging to $\s(\underline\alpha)$ there is a constant subsequence. This implies that $\s(\underline\alpha)$ is isolated in $\s(\Hom_{\mathcal G}^{(i)}(x,y))$.
    \end{proof}

   Now we can prove that $Z^i=H^{i+1}$  by induction: 
   For $i=0$ this is Theorem \ref{thm:structure of Zxy}, as above.  Suppose we have the equality for $i\ge 0$. Now any point $z$ in $Z^{i+1}$ is by definition not isolated in $Z^i$ which is $H^{i+1}$. By Claim \ref{clm: nontrivial accumulation}, this implies that $z$ is in $H^{i+2}$. Thus 
    $Z^{i+1}\subset H^{i+2}$, and by the inclusion (\ref{eqn: H in Z}) they are equal. 
    \end{proof}

    \begin{corollary}\label{cor: depth omega}
    The depth of $\Zxy yx$ is $\omega$.
    \end{corollary}       

\begin{proof}
    We need to show that each $Z^i\ne \emptyset$, and that $\cap_i Z^i = \emptyset$. The first of these comes from $Z^i=H^{i+1}$ and the fact that the $H^i$ are nonempty by definition. If the second fails then $\cap_i H^i\ne \emptyset$, so there is a sequence of paths $\underline \alpha_m$ with $\ell(\underline\alpha_m)\to\infty$, and $\s(\underline\alpha_m)$ bounded (in fact constant). By Claim \ref{clm: slack bounded below}, this is impossible.
\end{proof}

	\subsection{Finite component with an infinite leaf}\label{subsec: infinite cr leaf}

    We have now understood the structure of $N$-orbit closures when the minimizing lamination $\lambda_0\subset \Sigmazero$ consists only of (a finite collection of) simple closed curves.
    Now we consider the case that an arbitrary $\lambda_0$ contains a connected component $\mu_0$ with finitely many leaves, not all of them closed.
	\begin{theorem}\label{thm: chain recurrent leaf}
        Suppose $\mu_0\subset\lambda_0$ is a connected component with finitely many leaves, at least one of which is an infinite leaf.
        Suppose $x\in \T_+^1\mu$. Then $ \Zxy xx=[0,\infty) $.
	\end{theorem}

    \begin{proof}
    Consider those leaves $\T^1_+\mu^{\per} \subset \T_+^1\mu$ that project to periodic orbits in $\T^1\Sigmazero$ and the directed graph $\mathcal G^{\per}$ whose vertex set is $\T^1_+\mu^{\per}\cap \tau\inverse(0)$; the directed edges joining $y$ to $z$ are the elements of $\Axy zy$.
    Any leaf of $\T_+^1\mu$ is forward asymptotic to $Ay$ for some $y \in V(\mathcal G^{\per})$.  In particular, $Nx\cap Ay = a_ty$ for some $t$ and $y \in V(\mathcal G^{\per})$, and it suffices to compute $\Zxy xx = \Zxy {a_ty}{a_t y} = \Zxy yy$.
    Since $\Zxy xx$ is a closed semi-group, $\Zxy xx= [0,\infty)$ if and only if $\Zxy xx = \Zxy yy$ contains arbitrarily small positive values.

    Since $\T_+^1\mu$ contains the preimage of an infinite chain recurrent leaf and $y\in \T_+^1\mu$, there is an $\underline \alpha \in \Hom_{\mathcal G^{\per}}(y,y)$ with $\s(\underline \alpha) = 0$  (see \eqref{eqn: zero slack in lambda}).
    Since $\s(\underline \alpha \cdot \underline \alpha) =2 \s(\underline \alpha) =0$, by  replacing $\underline \alpha$ with $\underline \alpha\cdot \underline \alpha$, we may assume that $\ell(\underline \alpha)\ge 2$.
    
    Apply Lemma \ref{lem: non-trivial accumulation} to obtain a sequence $\underline \alpha_m\in \Hom_{\mathcal G^{\per}}(y,y)$ satisfying 
    \begin{itemize}
        \item the combinatorial length satisfies $\ell(\underline\alpha_m) = \ell(\underline \alpha)-1 \ge 1$;
        \item $\s(\underline \alpha_m) >0$ for all $m$; and 
        \item $\s(\underline \alpha_m) \to  \s(\underline \alpha) = 0$.
    \end{itemize} 
    The proof of Theorem \ref{thm:structure of Zxy} applies to see that $\Zxy yy\supset \s(\Hom_{\mathcal G^{\per}}(y,y))$, which  contains arbitrarily small positive values.
    This is what we wanted.
    \end{proof}

    \subsection{Moving toward general laminations}\label{subsec: isolated multi-curve}
    In Section \ref{sec: structure of horo orbit closures}, we will address the structure of $\Zxy xy$  where $x$ and $y$ are tangent to general chain recurrent laminations $\lambda_0$, which may have uncountably many leaves.
    In this section, we extract a lemma from the proof of Theorem \ref{thm:structure of Zxy} for use later on.

    In general, each connected component of $\lambda_0$ is either an isolated closed leaf or contains an infinite leaf.  
    Denote by $\lambda_0^{\imc}$ the \emph{isolated multi-curve} part of $\lambda_0$, which is just the union of the isolated closed leaves. 
    Denote by $\lambda_0^\infty$ the union of the components that contain an infinite leaf; %
    it is equal to $\lambda_0\setminus \lambda_0^{\imc}$.  

    We define a directed graph $\Gimc$ in a similar fashion as in the beginning of the section as follows.
    Denote by $\T^1_+\lambda^{\imc} \subset \T^1\Sigma$ as the preimage under $\pi_\Z$ of the tangents to $\lambda_0^{\imc}$ exiting the `$+$' end, and define $\T^1_+\lambda^\infty$ analogously.
    The vertex set $V(\Gimc)$ of $\Gimc$ is $\T^1_+\lambda^{\imc} \cap \tau\inverse(0)$. %
    The directed edge set from $y$ to $x$ is $\Axy xy$.

    \begin{lemma}\label{lem: isolated multi-curve slack - arcs outside nbhd of non-imc}
        Let $x, y \in V(\Gimc)$ and  $T \in \Zxy xy$.
        Suppose there is a positive $\epsilon>0$ and a sequence $n_m \in N$ such that $n_m x \to a_Ty$ as $m\to \infty$ and $A_+n_mx$ avoids $(\T^1_+\lambda^\infty)^{(\epsilon)}$ for all $m$.
        Then there is $\underline \alpha \in \Hom_{\Gimc}(y,x)$ such that $T = \s(\underline \alpha)$.
    \end{lemma}

    \begin{proof}
        The proof follows verbatim the proof of inclusion $\Zxy xy\subset \s(\Hom_{\mathcal G}(y,x))$ from Theorem \ref{thm:structure of Zxy} with the following changes: Here, $\T^1_+\lambda^{\imc}$ plays the role of $\T^1_+\lambda$, and $\epsilon_0$ from the proof of Theorem \ref{thm:structure of Zxy} should be taken smaller than $\epsilon$ from the statement of the lemma.
    \end{proof}
      
	\section{Chain proximality on minimal components}\label{sec:chain prox min}
    
    Suppose $\lambda$ is an oriented minimal geodesic lamination on a closed hyperbolic surface $S$ with more than one leaf.
    Let $\gamma$ be an oriented $C^1$ transversal to $\lambda$ without backtracking, i.e., $\gamma$ is transverse to $\lambda$ with the same sign everywhere.

    Let $X = \lambda \cap \gamma$ and $\sigma: X \to X$ be the first return for the geodesic flow tangent to $\lambda$ in the forward direction.
    Note that $X$ is a compact metric space with Hausdorff dimension $0$  and that $\sigma$ is a bi-Lipschitz homeomorphism.
    The latter fact is due to the classical observation that the map sending a point $x\in X$ to its forward unit tangent vector along $\lambda$ is bi-Lipschitz onto its image, the geodesic flow is smooth, and the first return time along the flow is a continuous function on $X$.

    The central notion of this section is that of \emph{chain proximality} (see Lemma
    \ref{lem:chain prox implies inclusion in Nxbar} for the connection between this notion and our $N$-orbit closures). 
    \begin{definition}\label{def: chain proximal}
        Let $X$ be a metric space and $\sigma: X 
        \to X$ be a map.
        For $x,y \in X$, we say that $x$ is \emph{chain proximal} to $y$ and write $x \chprox y$ if, for every $\ep>0$, there exists a sequence
          $x=x_0,x_1,\ldots,x_m$ such that
         \begin{equation}\label{eqn: strong ep chain}
             \sum_{i=0}^{m-1} d(x_{i+1},\sigma(x_i)) < \ep
         \end{equation}
 and $ x_m = \sigma^m(y)$.
 We call such a sequence an \emph{$\ep$-interception of $y$ by $x$} and say that $x$ $\ep$-intercepts $y$.
 If $x \chprox y$ and $y\chprox x$, we write $x\chproxb y$.
    \end{definition}
    Clearly, $\chprox$ is a reflexive relation on $X$.
    That $\sigma$ is bi-Lipschitz implies that $\chprox$ is also transitive.  
    A priori, $x\chprox y$ need not imply $y\chprox x$.
    
    \begin{remark}
        A \emph{strong $\ep$-chain} from $x$ to $y$ would be a sequence 
        $x= x_0, x_1, ...,$ $ x_m=y$ satisfying \eqref{eqn: strong ep chain}.  This notion seems to have been introduced by Easton \cite{Easton:strongchain} following work of Conley \cite{Conley:chain_recurrent}.
        Note that in our definition of chain proximality, an $\ep$-interception of $y$ by $x$ is a strong $\ep$-chain of length $m$ from $x$ to $\sigma^m(y)$.  That is, an $\ep$-interception of $y$ by $x$ starts at $x$, closely follows $\sigma$-orbits making summable jumps, and eventually catches up with the orbit of $y$ in a synchronous fashion.
    \end{remark}

    A neighborhood $\mathcal N$ of $\lambda$ in $S$ is called \emph{snug} if each component of $S\setminus \mathcal N$ is a deformation retract of the component of $S\setminus \lambda$ containing it.
    Suppose $\mathcal N$ is snug for $\lambda$ on $S$, and denote by $\{\gamma_i\}$ the set of connected components of $\mathcal N \cap \gamma$.
    Consider the partition $\{L_i = \gamma_i\cap X\}$ of $X=\lambda \cap \gamma$ by closed and open sets.

    Our first main result in this section is the following characterization of the chain proximality relation on $X$.

    	\begin{theorem}\label{thm: chain prox invt equiv relation}
    		Let $\lambda \subset S$ be a minimal oriented geodesic lamination and let $\gamma$ be a $C^1$  transversal to $\lambda$ without backtracking. Let $X=\lambda \cap \gamma$, with $\sigma$ the first return map to $X$. Then chain proximality is a $\sigma$-invariant equivalence relation on $X$ with finitely many equivalence classes $M_1, ..., M_s$.  
    		
    		Moreover, these equivalence classes are closed subsets of $X$ which are finite unions of members of the partition $\{L_i\}$.      
    \end{theorem}

    The proof of this theorem occupies the remainder of this section. 
    \medskip
    
    We call a component $J \subset \gamma_i\setminus X$ not containing an endpoint of $\gamma_i$ a \emph{gap}.  Say that a gap $J=(x,y)$ is \emph{shrinking in forward (resp. backward) time} if $d(\sigma^n(x), \sigma^n(x)) \to 0$ as $n\to +\infty$ (resp. $n \to -\infty)$.
    \begin{lemma}\label{lem: gaps shrinking}
        Every gap $J = (x,y)\subset \gamma_i \setminus X$ is shrinking in either forward or backward time.
    \end{lemma}

    \begin{proof}
        Let $S'$ be the metric completion of the component of $S\setminus \lambda$ containing $J$.  Then the closure of $J$ in $S'$ joins two of its boundary components.
        Since $J\subset \mathcal N$, and $\mathcal N$ is snug and since $S$ (hence $S'$) is of finite area, these two boundary components must be asymptotic, which implies the lemma.
    \end{proof}

    Let $\nu$ be a $\sigma$-invariant probability measure with full support on $X$.
    \begin{theorem}
    \label{thm: chain prox ae gaps}
        There is a subset $X^\dagger \subset X$ of full $\nu$-measure such that $x\chprox y$ for every $x\in X$ and $y \in X^\dagger$ in the same component $\gamma_i$ as $x$.
    \end{theorem}

    \begin{proof}
        For each $m\ge1$, we consider the diagonal action of $\sigma$ on $X^m$.  Say that an $m$-tuple $\underline x \in X^m$ is \emph{recurrent} if $\sigma^n(\underline x)$ accumulates on $\underline x$ as $n\to \infty$. %
        By Poincar\'e Recurrence, $\nu^m$-a.e. $\underline x \in X^m$ is recurrent. 
        By Fubini, there is a set $X_m \subset X$ of full $\nu$-measure such that for all $y \in X_m$ and for $\nu^{m-1}$-a.e. $\underline x \in X^{m-1}$, the tuple $y\times \underline x \in X^m$ is recurrent.
        Changing $X_m$ by at most a $\nu$-null set, we may assume that $X_m$ is $\sigma$-invariant, i.e., $\sigma(X_m) = X_m$ for all $m\ge 1$. 

        Then $X^\dagger = \cap_{m\ge 1} X_m$ is a $\sigma$-invariant set of full $\nu$-measure. Consider $y\in X^\dagger$, let $x\in X\cap \gamma_i$, and suppose $I\subset \gamma_i$ is an interval with endpoints $x$ and $y$.
        Denote by $\le$ the linear order on $I$, oriented positively from $x$ to $y$.

        Since $X$ has length $0$ in $\gamma_i$ we have $\ell(I) = \ell(I \setminus X)$.  Let $\ep>0$ be given, and find a finite collection of gaps $J_1< J_3 ...< J_{2k-1}$ of $X$ contained in $I$ such that $\ell(I \setminus \cup J_i) <\ep$.
        We have $J_i = (p_i, p_{i+1})$, so that 
        \[x=p_0\le p_1 <p_2<  ... < p_{2k-1}<p_{2k}\le p_{2k+1}=y,\]
        and take $m = 2k+1$.
        Notice that $I\setminus \cup J_i = \cup_{i = 0}^{k} [p_{2i}, p_{2i+1}]$, which has total length at most $\epsilon$.
        
        \begin{claim}\label{clm: making a chain}
            There are $q_0, ...., q_m=y \in X$ and $\epsilon_i > 0$ with $\sum \epsilon_i \le 3\epsilon$ such that $d(\sigma(q_0), \sigma(x))<\epsilon$ and such that for each $i = 0, ..., m-1$ there exist infinitely many positive values of $n$ satisfying
            \[d(\sigma^n(q_i), \sigma^n(q_{i+1}))<\epsilon_i.\]
        \end{claim}

        \begin{proof}[Proof of Claim \ref{clm: making a chain}]
        Recall $ m=2k+1 $. The following defines an open subset of $ X^m $ --- Let $ (q_0,\ldots,q_{m-1}) $ satisfy:
    		\begin{enumerate}[label=(\alph*)]
    			\item $ d(\sigma(q_0),\sigma(x))<\varepsilon $;
    			\item $ q_{2i},q_{2i+1} \in (p_{2i},p_{2i+1}) $ for all $ i=0,\ldots,k $ and $ q_{2k} \in (p_{2k},y) $; and
    			\item for each odd $i$, since the gap $J_i=(p_i, p_{i+1})$ is shrinking in either forward or backward time, we know there exists $m_i\in \Z$ such that $d(\sigma^{m_i}(p_i),\sigma^{m_i}(p_{i+1}))<\epsilon/4k$. Choose $ q_i,q_{i+1} $ so close to $ p_i $ and $ p_{i+1} $ respectively so as to ensure that $ d(\sigma^{m_i}(q_i),\sigma^{m_i}(q_{i+1}))<\epsilon/k $. See \cref{fig: chain prox jumps}.
    		\end{enumerate}
    		
    		\begin{figure}
    			\centering
    			\includegraphics[width=0.8\linewidth]{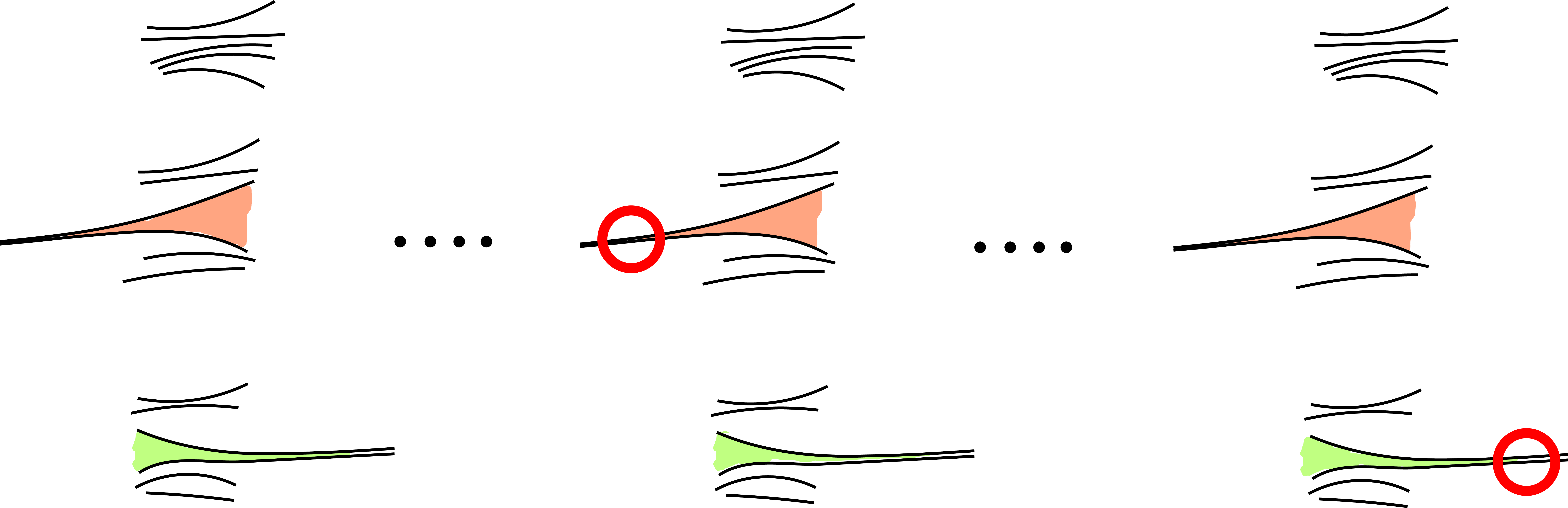}
    			\caption{In this illustration one can see a recurring transection of the lamination (in black) with forward asymptotic gaps (shaded in green) and backward asymptotic gaps (shaded in red). Red circles mark the position of intended jumps across gaps, which are from $\sigma^{m_i}(q_i) \sim \sigma^{m_i}(p_i) $ to $ \sigma^{m_i}(p_{i+1}) \sim \sigma^{m_i}(q_{i+1}) $ as described in (c) of the proof of Claim \ref{clm: making a chain}.}
    			\label{fig: chain prox jumps}
    		\end{figure}
    		
    		In case $ x=p_0=p_1 $, take $ q_0=q_1 $ at distance less than $ \varepsilon/2 $ from $ x $. Similarly whenever $ y=p_m=p_{m-1} $. In such cases, these conditions define an open set in $ X^{m-1} $ or $ X^{m-2} $.
    		
    		Since $ \sigma $ is bi-Lipschitz we are ensured that the open set defined above is non-empty. The measure $ \nu $ having full support in $ X $ and the fact that $ y \in X^\dagger $ imply that there exists a recurrent tuple $ (q_0,\ldots,q_{m-1},y) $ in $ X^{m+1} $ satisfying (a)-(c). Let $ n_l \to \infty $ be a sequence of times for which $ \sigma^{n_l}((q_0,\ldots,q_{m-1}, y)) $ tends to $ (q_0,\ldots,q_{m-1}, y) $.
    		
    		Set $ q_m=y $. For all $ i=1,\ldots,k $, and all large enough $ l\geq 1 $ we know that both $ \sigma^{n_l}(q_{2i}),\sigma^{n_l}(q_{2i+1}) \in (p_{2i},p_{2i+1}) $ and hence
    		\[ d(\sigma^{n_l}(q_{2i}),\sigma^{n_l}(q_{2i+1})) < d(p_{2i},p_{2i+1}). \] 
    		
    		On the other hand, for each odd $ 1 \leq i \leq 2k+1 $, we know that 
    		\[ d(\sigma^{n_l+m_i}(q_i),\sigma^{n_l+m_i}(q_{i+1}))<\epsilon/k \quad \text{for all large }l. \]
    		
    		Set $ \varepsilon_{2i}=d(p_{2i},p_{2i+1}) $ and $ \varepsilon_{2i+1}=\varepsilon/k $ for all $ i=0,\ldots,k $. Hence for each $ i=0,\ldots,m-1 $ there exist infinitely many $ n $'s where $ d(\sigma^n(q_i), \sigma^n(q_{i+1}))<\epsilon_i $ and
    		\[ \sum_{j=0}^{m-1} \varepsilon_j < \varepsilon + \sum_{i = 0}^k d(p_{2i},p_{2i+1}) + \sum_{i = 1}^k \varepsilon/k < \varepsilon + \ell\left(\cup_{i = 0}^{k} [p_{2i}, p_{2i+1}]\right)+\varepsilon<3\varepsilon, \]
    		proving the claim.
        \end{proof}

        With the claim established, we can now construct a $4\epsilon$-interception of $y$ by $x$. %
        Let $x_0 = x$ and let $x_1 = \sigma(q_0)$.
        Now choose a sequence of times $1 = n_0 < n_1 < ... <n_m$ inductively by choosing $n_{i+1}>n_i$ satisfying 
        \[d(\sigma^{n_{i+1}}(q_i), \sigma^{n_{i+1}}(q_{i+1}))<\epsilon_i,\]
        which is possible by Claim \ref{clm: making a chain}.

        Define, for $n_i\le j <n_{i+1}$ \[x_j = \sigma^j(q_i),\]
        and finally when $j = n_m$, we let 
        \[x_j = x_{n_m} = \sigma^{n_m}(q_m) = \sigma^{n_m}(y).\]

        These points $x=x_0, x_1, ..., x_{n_m} = \sigma^{n_m}(y)$ follow $\sigma$-orbits except at the ``jump'' times $n_i$, where the jump distance is controlled by the claim.
        Summing up the errors we conclude
        \[\sum d(x_{j+1}, \sigma(x_j))<4\epsilon.\]
        Letting $\epsilon$ tend to $0$ proves that $x\chprox y$.
    \end{proof}

    For the proof of Theorem \ref{thm: chain prox invt equiv relation}, we will use the following lemma regarding $\delta$-proximal pairs.
    \begin{lemma}\label{lem: almost proximal}
        With $\nu$ as before, let $x \in X$.  For every $\delta>0$, there is a set $F\subset X$ with $\nu(F)>0$ such that for all $z\in F$, 
        \begin{equation}\label{eq: delta-proximality}
        	\liminf_{n\to \infty} d(\sigma^n(x), \sigma^n(z))<\delta.
        \end{equation}
    \end{lemma}
    \begin{proof}
        Suppose not.  Then there is a $\delta>0$ such that for $\nu$-a.e. $z\in X$, there is an $N_z<\infty$, such that for $n\ge N_z$,
        \[d(\sigma^n(x), \sigma^n(z))\ge \delta.\]
        Since $\nu$ has no atoms, the function $X \ni y\mapsto \nu(B_\delta(y))$ is continuous. Since $\nu$ has full support, which is compact, there is a $b>0$ such that $\nu(B_\delta(y))>b$ for all $y\in X$, where $B_\delta(y)$ is the ball of radius $\delta$ around $y$ in $X$.
        Since $N_z$ is finite for $\nu$-a.e. $z$, there is some $N$ such that $F_N = \{z: N_z<N\}$ has measure greater than $1-b$. %

        Thus for $n>N$, we find that $\sigma^n(F_N)$ is disjoint from $B_\delta(\sigma^n(x))$.  Since $\sigma$ is a homeomorphism preserving $\nu$, we see that $\nu(\sigma^n(F_N)) +\nu(B_\delta(\sigma^n(x)))>1-b +b >1$, which is a contradiction.   
    \end{proof}

    \begin{proof}[Proof of Theorem \ref{thm: chain prox invt equiv relation}]
        Note first that, since $\chprox$ is transitive and reflexive, the relation $\chproxb$ is (tautologically) an equivalence relation. 
        Recall the definition of $ L_i= \lambda \cap \gamma_i $, where $ \gamma_i $ are the connected components of $ \gamma \cap \cal N$. By Theorem \ref{thm: chain prox ae gaps}, each $L_i\cap X^\dagger$ is contained in an equivalence class of $\chproxb$. Therefore the equivalence classes in $X^\dagger$ of $\chproxb$ can be written as $M_j\cap X^\dagger$, where each $M_j$ is a union of some subcollection of $L_i$. Note that the partition  $X=\cup M_j$ is invariant by $\sigma$, since the relation $\chproxb$ is invariant by $\sigma$ and $X^\dagger$ is dense in $X$. 

        Now consider $x,y\in M_i$ and let us show that $x\chprox y$. 
        Let $b>0$ be the minimum distance between $M_j$ and $M_k$ for all $j\not=k$, and let $\ep>0$ be given. 
         Let $F$ be the set from Lemma \ref{lem: almost proximal} for $\delta = \min\{\epsilon/2,b\}$, and for the point $y$.
    Then $\nu(F\cap X^\dagger)>0$; take $y' \in F\cap X^\dagger$, so that $y'$ is $\delta$-proximal to $y$ in the sense of \eqref{eq: delta-proximality}. This implies $y'\in M_i$ as well, since otherwise the distance between $\sigma^i(y)$ and $\sigma^i(y')$ for $i>0$ is bounded below by $b$.

    We can approximate $x$ as closely as we'd like by $x' \in M_i\cap X^\dagger$, where we already know  $x'\chproxb y'$.
    Thus, we have an $\ep/2$-interception of $y'$ by $x'$.  Using Lemma \ref{lem: almost proximal}, there is $m'>m$ such that $d(\sigma^{m'}(y'), d(\sigma^{m'}(y))<\ep/2$. By concatenation, this produces an $\ep$-interception of $y$ by  $x'$.  Since $\ep$ was arbitrary, we have $x'\chprox y$.

    Since $x'$ can be made arbitrarily close to $x$, we conclude that $x\chprox y$ (by prepending to the chain a jump from $\sigma(x)$ to $\sigma(x')$). Arguing symmetrically, $y\chprox x$. On the other hand, if $x\in M_i$ and $y\in M_j$ for $i\ne j$, then their orbits remain at least $b$ apart for all time, and so $x\chprox y$ cannot hold. 
    Thus $\chprox$ is equal to $\chproxb$, so it is an equivalence relation and the $M_i$ are the equivalence classes. This concludes the proof of Theorem \ref{thm: chain prox invt equiv relation}.
    \end{proof}

	\section{A synchronous Transversal}\label{Section:Synchronous Transversal}
    In this section, we study the chain proximality relation for the first return mapping to a $\tau_0$ fiber for the geodesic flow tangent to $\lambda_0$ and explain how chain proximality allows us to conclude  containments of $N$-orbit closures in $\T^1\Sigma$ (Lemma \ref{lem:chain prox implies inclusion in Nxbar}).
    In order to apply  the results of the previous section describing the chain proximality relation, we construct a \emph{synchronized} $C^1$ transversal $\gamma$ to $\lambda_0$ (Lemma \ref{lem: nice transversal}), i.e., one that meets every minimal component  in the same $\tau_0$-fiber. 
    This good transversal will in fact be contained in leaves of Thurston's {\em horocyclic
      foliation}, which we describe below.    In particular, the construction of $\gamma$ actually gives some insight into the
    structure of \emph{every} tight map in a neighborhood of $\lambda_0$ (Corollary
    \ref{cor: structure of tight maps}).

    We then give a satisfying classification of the chain proximality equivalence classes in terms of the connected components of the preimage of $\lambda_0$ in a finite cover (Theorem \ref{thm: chprox equiv chain recurrent}) and in terms of the weak components of $\lambda$ (Corollary \ref{cor: chprox and weak components}).
    Finally, in \S\ref{subsec: arcs and shifts revisit}, we return to our discussion relating slacks and shifts for arbitrary $x,y \in \T^1_+\lambda$ in the same $\tau$-fiber.  

    \subsection{Chain proximality and orbit closures}\label{subsec: why chain prox}

The reason we are interested in the chain proximality relation is that it is tightly
connected to containment of orbit closures. %

    Let $\Y = \T^1_+\lambda \cap \tau\inverse(0)$, let $\Y_0 = \pi_\Z(\Y)\subset \T^1\Sigma_0$, and 
    denote by 
    \[\sigma_0: \Y_0 \to \Y_0\]
    the first return mapping, i.e., $\sigma_0(x) = a_{c} x$.

    The notion of chain proximality also makes sense applied to the time $c$ map for the geodesic flow restricted to $\T^1_+\lambda$.
    \begin{lemma}\label{lem: chprox in Sigma vs Sigmazero}
        For $x$ and $y \in \Y$, $x \chprox y $ for $a_c$ if and only if $\pi_\Z(x) \chprox \pi_\Z(y)$ for $\sigma_0$.    
    \end{lemma}

    \begin{proof}
        Assume that $x\chprox y$ for $a_c$.  Using the relation that $\pi_\Z(a_cz) = \sigma_0 (\pi_\Z(z))$ and that $\pi_\Z$ is $1$-Lipschitz, any $\epsilon$-interception of $y$ by $x$ using $a_c$ descends to an $\epsilon$-interception of $\pi_\Z(y)$ by $\pi_\Z(x)$ using $\sigma_0$, demonstrating that $\pi_\Z(x) \chprox \pi_\Z(y)$ for $\sigma_0$.

        For the other direction, observe that 
        \[\pi_\Z\inverse(\Y_0) = \sqcup_{m \in \Z} a_{cm}\Y,\]
        so that $\pi_\Z$ restricts to a bijection $a_{cm}\Y \to \Y_0 $, for each $m$.
        Let $ \pi_\Z(x)= x_0, x_1, ..., x_N = \sigma_0^N(\pi_\Z(y))$ be an $\epsilon$-interception of $\pi_\Z(y)$ by $\pi_\Z(x)$, where $\epsilon$ is smaller than half the injectivity radius of $\T^1\Sigmazero$ and define $y_i \in a_{ci}\Y$ by $\pi_\Z(y_i) = x_i$.
        Since $a_{[0,c]}y_i\subset \T^1_+\lambda$, we have $\tau(a_cy_i) = \tau(y_i) +c = ci+c$, which implies that $a_c y_i \in a_{c(i+1)}\Y$.  
        Since $\pi_\Z$ is locally isometric, and $\epsilon$ is smaller than half the injectivity radius of $\T^1\Sigmazero$, we have 
        \[\sum d(y_{i+1}, a_c y_i) =\sum d(x_{i+1}, \sigma_0(x_i)) <\ep.\]
        Thus $x = y_1, ..., y_N = y$ is an $\epsilon$-interception of $y$ by $x$, which proves the lemma.
    \end{proof}
    
    The first return map $ \sigma_0 $ allows us to correctly relate $ \sigma_0 $-chain-proximality on $\Y_0$ with the geodesic flow along leaves of $\T^1_+ \lambda $ in $ \T^1\Sigma $, and consequently conclude horocycle orbit accumulation relations.

    \begin{lemma}\label{lem:chain prox implies inclusion in Nxbar}
		Let $ x,y \in \mathcal Y $ with $ x \chprox y $ then $ x \in \overline{Ny} $. Consequently, if $ y \chproxb x $ then $ \overline{Ny}=\overline{Nx} $.
	\end{lemma}

	\begin{proof}
    Let $ \pi_\Z(x)=x_0,\ldots,x_m=\sigma_0^m(\pi_\Z(y)) = \pi_\Z(a_{mc}y) \in \Y_0$ be an $ \varepsilon $-interception of $ \pi_\Z(y) $ by $ \pi_\Z(x) $. 
    As in the proof of Lemma \ref{lem: chprox in Sigma vs Sigmazero}, we have corresponding points $y_i \in a_{ci}\Y$ that determine an $\ep$-interception of $y$ by $x$ for $a_c$ in $\T^1_+\lambda$.
    
    Consider geodesic arcs $\alpha_0, \ldots, \alpha_{m-1}$ with $\alpha_i(t) = a_ty_i$, for $t \in [0,c]$, and the ray $ \alpha_m(t)=a_{mc+t} y = a_t y_m $ for $ t \in [0,\infty) $. Hence we have  $ \ell(\alpha_i) \geq c $ for all $ i=0,\ldots,m $ and 
		\[ \sum_{i=0}^{m-1} d_{\T^1\Sigma}(\alpha_i(c),\alpha_{i+1}(0)) < \varepsilon. \]
		By \Cref{lem:epsilon chain slack}, we conclude that for all small enough $ \varepsilon $, there exists a geodesic ray $\alpha^\varepsilon$ that is asymptotic to $Ay$ in forward time satisfying
		\[ \s(\alpha^\varepsilon) \leq \sum_{i=0}^{m} \s(\alpha_i) + \left|\s(\alpha^\varepsilon)- \sum_{i=0}^{m} \s(\alpha_i) \right| < \kappa_c\varepsilon, \]
        where $\alpha^\varepsilon = A_+x^\varepsilon$.
		Moreover, as $ \varepsilon \to 0 $ we have $ x^\varepsilon \to x $. Applying \Cref{lem:slack gives points in Nxbar}, with $ t=0 $, we conclude that $ x \in \overline{Ny} $, as claimed.
		
		The last implication follows from the fact that $ x \in \overline{Ny} \Rightarrow \overline{Nx} \subset \overline{Ny} $.
	\end{proof}

    \subsection{The horocyclic foliation}
    Consider the metric completion $\Sigma'$ of $\Sigmazero\setminus \lambda_0$, which is a finite area hyperbolic surface with totally geodesic boundary.  In the universal cover of each of the finitely many connected components, there is shortest positive distance between any two non-asymptotic boundary geodesics.  Let $\delta_0>0$ be smaller than $1/4$ of the minimum such distance, and denote by $\mathcal N_0$ the closure of the $\delta_0$-neighborhood of $\lambda_0$ on $\Sigmazero$.
    Then $\mathcal N_0$ is a snug neighborhood of $\lambda_0$.
        
    Pairs of asymptotic leaves in $\lambda_0$ correspond to ends of non-compact boundary components of $\Sigma'$ called \emph{spikes}.
    Each spike has a maximal closed neighborhood contained in $\mathcal N_0$ that is foliated by horocyclic segments facing the end and joining the boundary components, meeting them orthogonally.  
    Denote by $\mathcal S_0$ the closure of the union of these foliated spike neighborhoods, together with the isolated, closed leaves of $\lambda_0$.
    Note that $\mathcal S_0$ is a closed set containing $\lambda_0$; see Figure \ref{fig: horofoliation}.
    
\begin{figure}
    \centering
    \includegraphics[width=0.75\linewidth]{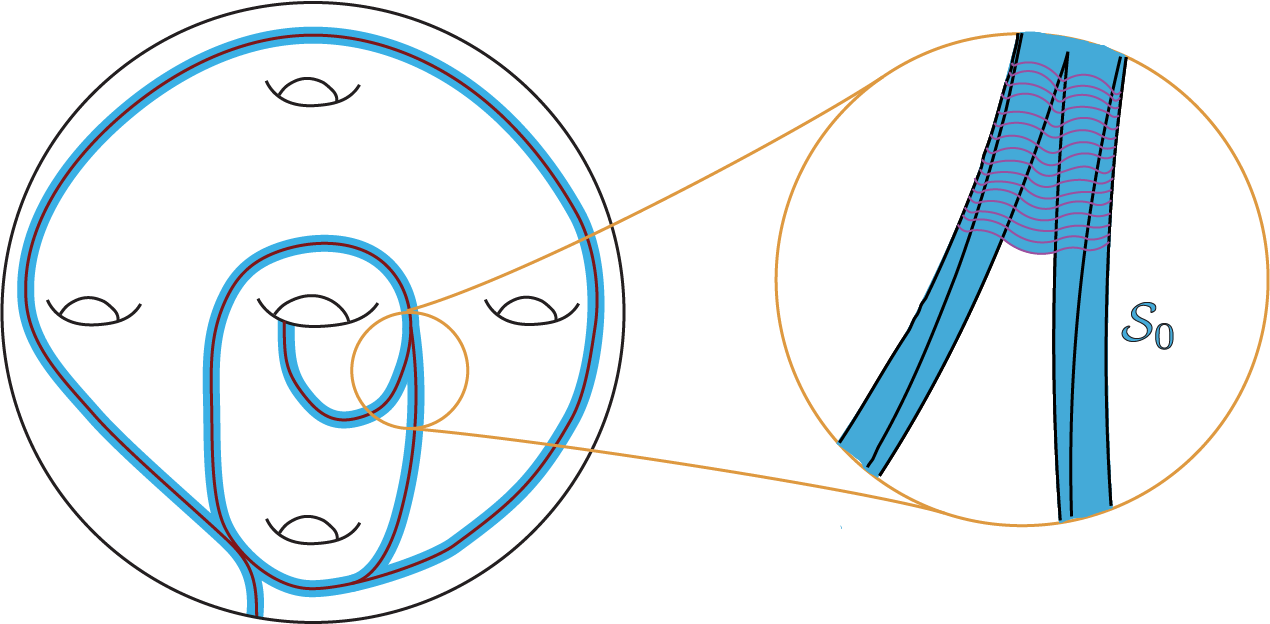}
    \caption{The spike neighborhood $\mathcal S_0$ inside of a snug neighborhood for $\lambda_0$ and leaves of the horocycle foliation.}
    \label{fig: horofoliation}
\end{figure}
    The partial foliation of $\mathcal S_0$ by horocyclic segments extends across the leaves of $\lambda_0$ to a $C^1$ foliation of $\mathcal S_0$ %
    called the \emph{horocyclic foliation}, which was defined by Thurston in a
    neighborhood of $\lambda_0$ (see \cite[\S8.9]{Thurston:notes} and
    \cite[\S4]{Thurston:stretch}, or \cite{CF:SHSH,CF:SHSHII} for a related
    construction). Each leaf of this foliation is the closure of a union of both stable and unstable horocyclic arcs. 
    We regard the isolated leaves of $\lambda_0$ as being foliated by their points.

    \subsection{Horocycle  transversals are synchronous}\label{subsec: const synchronous transversal}

    The following lemma supplies us with a nice transversal to $\lambda_0$ whose intersection with $\lambda_0$ is contained in a $\tau_0$-fiber.  
    An oriented $C^1$ transversal $\gamma$ to $\lambda_0$ is \emph{complete} if no component backtracks and $\gamma$ meets every minimal sublamination of $\lambda_0$.

    \begin{lemma}\label{lem: nice transversal}
        For any fiber $F_0$ of $\tau_0: \Sigmazero \to \R/c\Z$, there is an oriented $C^1$
        complete transversal $\gamma$ to $\lambda_0$ with $\gamma \cap
        \lambda_0 = F_0\cap \lambda_0$. Each arc of $\gamma$ is
        contained in a leaf of the horocyclic foliation of $\lambda_0$. 
    \end{lemma}

    \begin{proof}
    Define $Y_0 = F_0\cap \lambda_0$.
    Given $x \in Y_0$, denote by $\gamma_x$ the leaf of the horocycle foliation in $\mathcal S_0$ containing $x$.
    We claim that $\gamma_x \cap \lambda_0\subset Y_0$.
    Indeed, for any horocyclic segment $J \subset \gamma_x$ joining points $y, z \in \lambda_0\cap \gamma_x$, it must be that $\tau_0(y) = \tau_0(z)$.  The easiest way to see this is to lift the situation to $\Sigma$, where the the corresponding leaves $g_y$ and $g_z$ of $\lambda = \pi_\Z\inverse (\lambda_0)$ are asymptotic in one direction (because they are joined by a horocyclic arc). %
    Since $\tau$ is continuous and isometric along leaves of $\lambda_0$ it follows that $\tau$-values of endpoints of horocyclic segments joining $g_y$ to $g_z$ coincide.
    Since $\gamma_x$ is a $C^1$ transversal and $\lambda_0$ has zero $2$-dimensional Lebesgue measure, Fubini gives that $\gamma_x \cap \lambda_0$ has $1$-dimensional Lebesgue measure $0$.
    Given $\epsilon>0$, there are finitely many horocyclic segments $J_1 \le ... \le J_m$ (for a linear order on $\gamma_x$) such that $\ell(\gamma_x\setminus \cup J_i) <\epsilon$.
    Let $t_i$ be the value of $\tau_0$ at the endpoints of $J_i$.
    Since $\tau_0$ is $1$-Lipschitz, $\sum_{i = 1}^{m-1} |t_i - t_{i+1}|<\epsilon$.
    Since $\epsilon$ was arbitrary, the claim that $\gamma_x \cap \lambda_0$ is contained in the same $\tau_0$-fiber follows.

    Consider the collection $\{\gamma_x \}_{x\in Y_0}$.  If $y \in Y_0\cap \gamma_x$, then $\gamma_y = \gamma_x$, and $\gamma_x \cap Y_0$ is open in $Y_0$.  By compactness, there are only finitely many such arcs and points, and we can take $\gamma$ to contain all such.
    \end{proof}

    The proof establishes the following structural result for tight maps.
    \begin{cor}\label{cor: structure of tight maps}
        Any tight map $\Sigmazero \to \R/c\Z$ with canonically stretched lamination $\lambda_0$ is homotopic in a snug train track neighborhood of
        $\lambda_0$ to a map whose fibers are leaves of the horocycle foliation, and
        the homotopy can be chosen to be constant on $\lambda_0$.
    \end{cor}

    \subsection{Chain proximality in a finite cover}

    Denote by $\lambda_0^{\min}$ the union of the minimal sublaminations of $\lambda_0$ so that $\lambda_0\setminus \lambda_0^{\min}$ consists of finitely many infinite chain recurrent leaves that spiral onto the minimal components.  

    Recall that for each $d\ge 1$, there is a $d$-sheeted cover $\pi_d:\Sigma_d\to \Sigmazero$ and a $1$-Lipschitz tight map $\tau_d : \Sigma_d \to \Z/dc\Z$  induced by the degree $d$ map $\Z/dc\Z \to \Z/c\Z$.
    Denote by $\lambda_d^{\min}\subset \lambda_d$ the minimal part of the canonical maximally stretched lamination for $\tau_d$. 

    \begin{lemma}
        For all $d\ge 1$, we have $\lambda_d^{\min} = \pi_d\inverse(\lambda_0^{\min})$ and $\lambda_d = \pi_d\inverse(\lambda_0)$.
    \end{lemma}
    \begin{proof}
    The maximally stretched lamination $\lambda_d$ for $\tau_d$ is the preimage under $\pi_d$ of the maximally stretched lamination $\lambda_0$ for $\tau_0$, because the coverings are locally isometric and being maximally stretched is a local condition.  
    Each component of $\lambda_d^{\min}$ maps to a component of $\lambda_0^{\min}$, and each component of $\lambda_0^{\min}$ has preimage that is a union of components of $\lambda_d^{\min}$.
    \end{proof}

    With $\Y\subset \T^1\Sigma$ as in \S\ref{subsec: why chain prox},
    for every $d\ge 1$, denote by $\Y_d \subset \T^1\Sigma_d$ the image of $\Y$ under $\pi_{d\Z}: \Sigma \to \Sigma_d$.
    We denote by 
    \[\sigma_d: \Y_d \to \Y_d\]
    the first return mapping for the geodesic flow, i.e., $\sigma_d(x) = a_{dc} x$.
    By fiat, $\Y_0 = \Y_1$, $\Sigmazero = \Sigma_1$, $\sigma_0 = \sigma_1$, and so on.

    Define also, for each $d\ge 0$, $\mathcal Y_d^{\min}\subset \Y_d$ as the subset tangent to $\lambda_d^{\min}$.
    \medskip
    
    Lemma \ref{lem: nice transversal} gives us a nice transversal $\gamma$ to $\lambda_0$ with $\lambda_0 \cap \gamma = \lambda_0 \cap \tau_0\inverse(0)$. 
    Since the natural map $\Y_0 \to \lambda_0 \cap \gamma$ is a bi-Lipschitz homeomorphism,\footnote{That this map is bi-Lipschitz uses the properties of geodesic laminations in dimension $2$.} we can apply Theorem \ref{thm: chain prox invt equiv relation} to obtain, for each $\sigma_0$-minimal closed invariant set $\mathcal X_0 \subset \Y_0^{\min}$, a description of the $\sigma_0|_{\mathcal X_0}$-chain proximality equivalence classes on $\mathcal X_0$.  In particular, there are finitely many, and the corresponding finite partition of $\mathcal X_0$ is left invariant by $\sigma_0|_{\mathcal X_0}$.

    Since there are only finitely many components of $\lambda_0^{\min}$, the following is essentially a direct consequence.
    \begin{cor}\label{cor: finite cover}
        There is a $d\ge 1$ such that the connected components of $\lambda_d^{\min}\subset \Sigma_d$ are in bijection with the $\sigma_d|_{\Y_d^{\min}}$-chain proximal equivalence classes  given by Theorem \ref{thm: chain prox invt equiv relation}, i.e., for $x$ and $y \in \Y_d^{\min}$, $x$ is $\sigma_d$-chain proximal to $y$ if and only if they are tangent to the same component of $\lambda_d^{\min}$.  %
    \end{cor}

    \begin{proof}
        There is a definite distance between components of $\Y_0^{\min}$, so each 
chain proximal equivalence class for $\sigma_0|_{\Y_0^{\min}}$ is contained in a minimal
component, and hence is equal to one of the
chain proximal equivalence classes for $\sigma_0$ when restricted to that component. 
        Theorem \ref{thm: chain prox invt equiv relation} asserts that there are finitely
        many such, and they are permuted by $\sigma_0$.  
        For a suitable choice of $d$, $\sigma_0^d$ fixes each equivalence class. 
        
        Note that $\sigma_d: \Y_d^{\min} \to \Y_d^{\min}$ is isomorphic to $\sigma_0^d:
        \Y_0^{\min}\to \Y_0^{\min}$, and $\lambda_d^{\min}$ is the tangent projection of
        the suspension of $\sigma_d|_{\Y_d^{\min}}$. 
        Thus, each $\sigma_d$ chain proximality equivalence class in $\Y_d^{\min}$
        suspends to a sublamination of $\lambda_d^{\min}$, hence a union of minimal
        components. On the other hand as above each equivalence is contained in a
        component, so in fact its suspension is exactly a minimal component of  $\lambda_d^{\min}$.
    \end{proof}

        \begin{remark}\label{rmk: minimal filling lamination}
        If $\lambda_0^{\min}$ is {\em filling}, i.e., its complementary components are disks, then any lift to a finite cover is filling, hence also minimal.
        In that case it follows from Corollary \ref{cor: finite cover} that {\em all} pairs are chain-proximal. %
    \end{remark}

    \begin{remark}
		We note that existence of a $ d \geq 1 $ for which some minimal component $ (\mathcal X_0,\sigma_0) $ of $ (\Y_0^{\min},\sigma_0) $ lifts to more than one minimal component of $ (\Y_d^{\min},\sigma_d) $ is equivalent to the existence of a continuous rational eigenfunction for $ (\mathcal X_0,\sigma_0) $.
	\end{remark}
    
    \subsection{Isolated chain recurrent leaves}
    Now that we have understood the chain proximality relation on $\Y_0^{\min}$, we consider the role of the isolated leaves in $\lambda_0$.
    This will be easier to do  in the finite cover guaranteed by Corollary \ref{cor: finite cover} where every pair $x$ and $y$ in the same component of $\Y_d^{\min}$ is $\sigma_d$-chain proximal.
    
    The presence of isolated leaves in the maximal-stretch lamination $\lambda_d$ allows for the possibility that chain proximal equivalence classes for $\Y_d^{\min}$ could merge when considering the chain proximality relation for the first return mapping on all of $\Y_d= \T_+^1\lambda_d\cap \tau_d\inverse(0)$. 
    
    The following theorem asserts that $\sigma_d$-chain proximality is an equivalence relation on $\Y_d$ whose equivalence classes correspond to connected components of $\lambda_d$.

    \begin{theorem}\label{thm: chprox equiv chain recurrent}
        $\sigma_d$-chain proximality is an  equivalence relation on $\Y_d$. The corresponding partition by equivalence classes is $\{\Y_d \cap \T^1\mu_i\}$, where $\mu_1, ..., \mu_k$ are the connected components of $\lambda_d$.  
    \end{theorem}

    \begin{proof}
        If $\Y_d^{\min} = \Y_d$, i.e., $ \lambda_0^{\min}=\lambda_0$, then this is just Corollary \ref{cor: finite cover}.

        There is a finite directed graph whose vertices are the connected components of $ \lambda_d^{\min}$, and there is a directed edge from $\mu^-$ to $\mu^+$ if there is an isolated leaf  $g\subset \lambda_d$ whose past accumulates onto $\mu^-$ and whose future accumulates onto $\mu^+$.
        Since $\lambda_d$ is chain recurrent, the connected components of $\lambda_d$ correspond to (directed, recurrent) components of this graph (see \cite[\S8]{Thurston:stretch}).

        For such an isolated, infinite leaf  $g$, let $\mathcal Z =\T^1g \cap \Y_d$, $\mathcal Z^- = \T^1\mu^-\cap \Y_d$,  and $\mathcal Z^+ = \T^1\mu^+\cap \Y_d$.
        By the previous paragraph, to show that $x\chproxb y$ whenever $x$ and $y\in \Y_d$ project to the same connected component of $\lambda_d$, it suffices to show that $x\chprox y$ whenever
        \begin{enumerate}
            \item $x \in \mathcal Z$ and $y \in \mathcal Z^+$: in this case $x$ is $\sigma_d$-proximal to some $z \in \mathcal Z^+$, hence $x \chproxb z$. Since $z \chproxb y$, we get $x \chproxb y$.
            \item $x \in \mathcal Z$ and $y \in \mathcal Z$: there is some $m \in \Z$ such that $y = \sigma_d^m(x)$.
        Then $x$ is proximal to some $z \in \mathcal Z_+$ and so $y$ is proximal to $\sigma_d^m(z)$.  Then $x\chproxb z \chproxb \sigma_d^m(z) \chproxb y$.
      \item
$x \in \mathcal Z^-$ and $y \in \mathcal Z$: let $\epsilon>0$ be given. Since $\mu_-$ is
        minimal, $\{\sigma_d^{-m}(y):m\ge 0\}$ is dense in $\mathcal Z^-$. Let $m\ge 1$ be
        such that $d(x,\sigma_d^{-m}(y))<\epsilon/2$. Since by (2) we have
        $\sigma_d^{-m}(y)\chprox y$, we can add one step from $x$ to an interception of
        $y$ by $\sigma_d^{-m}(y)$, and conclude $x\chprox y$.
        \end{enumerate}

    This proves that for every connected component $\mu_i$ of $\lambda_d$, every pair $x,y \in \T^1\mu_i\cap \Y_d$ satisfies $x\chproxb y$.
    Since there is some definite distance between two connected components of $ \lambda_d $ in $ \Sigma_d $, no two $x\in \T^1\mu_i\cap \Y_d$ and $y \in \T^1\mu_j\cap \Y_d$ can have $x\chprox y$ if $i\not=j$, concluding the proof of the theorem.
    \end{proof}

    \subsection{Chain-proximality in $\Sigma$ revisited} 
    Recall that a \emph{weak component} $\mu\subset \lambda \subset \Sigma$ is a sublamination with the property that the $\epsilon$-neighborhood of $\mu$ is connected for every $\epsilon>0$. 
    Recall from \S\ref{subsec: why chain prox} that the chain proximality relation is defined on $\T^1_+\lambda$ for the time $c$ map $a_c$ for the geodesic flow, and Lemma \ref{lem: chprox in Sigma vs Sigmazero} says, for points $x, y \in \Y$, that ``$x\chprox y$ for $a_c$'' is equivalent to ``$\pi_\Z(x) \chprox \pi_\Z(y)$ for $\sigma_0$.''
		
	\begin{cor}\label{cor: chprox and weak components}
			Two points in $\Y=\T^1_+\Sigma \cap \tau^{-1}(0)$ are chain-proximal if and only if they are contained in the same weak connected component of $\T^1\lambda$.
	\end{cor}

	\begin{proof}
        With $d$ as in Theorem \ref{thm: chprox equiv chain recurrent} and Corollary \ref{cor: chprox and weak components}, 
        it suffices to show that $\pi_{d\Z}:\Sigma \to \Sigma_d$ induces a bijection between the weak connected components of $\lambda$ and the connected components of $\lambda_d$. 

		Clearly, one weak connected component of $\lambda$ cannot project onto two components of $\lambda_d$. On the other hand, suppose $x,y \in \Y$ are two points which project into the same connected component of $\lambda_d$. By \Cref{thm: chprox equiv chain recurrent}, $x \chproxb y$ and therefore for all $\varepsilon>0$ there exists an $\varepsilon$-interception of $\pi_{d\Z}(x)$ by $\pi_{d\Z}(y)$, consisting of $m$ geodesic arcs, each of length $cd$, starting at $\pi_{d\Z}(y)$ and ending at $\pi_{d\Z}(\sigma_d^m(x))=\pi_{d\Z}(a_{cdm}x)$. Lift this quasi-orbit to $\T^1\Sigma$ beginning at $y$ and terminating at $a_{cdm}x$, as in Lemma \ref{lem: chprox in Sigma vs Sigmazero}. 
		
		Using \Cref{lem:epsilon chain slack}, there exists a geodesic arc $\alpha^\varepsilon$, beginning $\kappa_c\varepsilon$-close to $y$ and ending $\kappa_c\varepsilon$-close to $a_{cdm}x$, which is completely contained in the $\kappa_c \varepsilon$-neighborhood of $\T^1\lambda$.
		Therefore, $x$, $a_{cdm}x$, and $y$ are in the same connected component of the $\kappa_c \varepsilon$-neighborhood of $\T^1\lambda$. Since $\varepsilon>0$ can be taken arbitrarily small we conclude that $x$ and $y$ lie in the same weak connected component of $\lambda$, concluding the proof.
		\end{proof}

	\subsection{Arcs and shifts revisited}\label{subsec: arcs and shifts revisit}

	In \S\ref{Sec:Slack of Paths} we described the relation between the slack of geodesic rays and the shift in the geodesic direction of horocycle accumulation points. We may give now a first complete description of these shift sets. %

	\begin{prop}\label{prop:Zxy closure of s(Axy)}
		For all $ x,y \in \T^1_+\lambda $ with $ \tau(x)=\tau(y) $ 
		$$ \Zxy{x}{y}=\overline{\s(\Axy{x}{y})}. $$
	\end{prop}
	\begin{proof}		
		First, note that applying $ a_{-\tau(x)} $ to $ x $ and $ y $ does not change $ \Zxy{x}{y} $ or $ \Axy{x}{y} $ and hence we may assume that $ x,y \in \tau^{-1}(0) $. 
		
		Let $ d $ be the degree of the finite cover discussed in \Cref{cor: finite cover}. Let us assume at first that the points $x$ and $ y $ project into minimal components $ \mu_x $ and $\mu_y$ of $ \lambda_d $, respectively. Denote by $\T^1\tilde{\mu}_x$ and $\T^1\tilde{\mu}_y$ the corresponding lifts to $\T^1\Sigma$ containing $x$ and $y$, and note that $ d\Z $ preserves these lifts.
        \medskip

        The basic idea is to use minimality of $\sigma_d$ on $\mu_y\cap \tau_0\inverse(0)$ and $\sigma_d$-chain proximality between any two points on $\mu_x \cap \tau_0\inverse(0)$ to chain together the past of $y$ with the future of $x$ using a large segment of a large translate of $Az$ (accounting for most of its slack); see Figure \ref{fig: Zxy=s(Axy)}.  Lemma \ref{lem:epsilon chain slack} gives $A$-orbits in $\Axy xy$ with slack approaching $T$ and containing points approaching $y$, so applying Lemma \ref{lem:slack gives points in Nxbar} gives that $T \in \Zxy xy$.
        \medskip
        
		More precisely, let $ Az \in \Axy{x}{y} $ be any arc and denote $T=\s(Az)$. Fix some $ \varepsilon > 0 $. 
		By the minimality of $ \mu_y $ we know that
		\[ \liminf_{n \in d\Z \;,\; n\to \infty}d\left(n.(a_{-cn}y),y\right)= 0. \]
		Since $ Az $ is backward asymptotic to $ Ay $ there exist arbitrarily large $ s_1 >0 $ and $n \in d\Z $ satisfying
		\[ d\left(n.(a_{-s_1}z),y\right)<\varepsilon. \]
		Pick sufficiently large $ s_1,s_2 > 0 $ such that additionally
		\[ \left| \s(A_{[-s_1,s_2]}z)-T \right| < \varepsilon. \]
		Since $ Az $ is forward asymptotic to $ Ax $, one may require that $ s_2 $ additionally satisfies 
		\[ d(a_{s_2}z,a_{cm}x)<\varepsilon \quad \text{for some large } m \in d\Z. \]
				
		Notice that both $ (-m).a_{cm}x $ and $(-(n+m)).a_{c(n+m)}x$ are contained in $ \T^1\tilde{\mu}_x\cap \tau^{-1}(0) $. By \Cref{cor: finite cover}, we know $ (-m).a_{cm}x \chproxb (-(n+m)).a_{c(n+m)}x $. This implies that there exists an $ \varepsilon $-interception from $ (-m).a_{cm}x $ to $(-(n+m)).a_{c(n+m)}x$. \Cref{lem:epsilon chain slack} implies there exists a geodesic ray $ \beta $ beginning $ \kappa_c \varepsilon $-close to $ (-m).a_{cm}x $, asymptotic to $ (-(n+m)).A_+x $ and having 
		slack $ 0 \leq \s(\beta) < \kappa_c\varepsilon $. 		
		
		Now consider the geodesic ray $ \alpha $ constructed by connecting $ n.A_{[-s_1,s_2]}z $ with $ (n+m).\beta $. Notice that the terminal point of $ n.A_{[-s_1,s_2]}z $ is $n.a_{s_2}z$ which is $(1+\kappa_c)\varepsilon$-close to $n.a_{cm}x$, the initial point of $(n+m).\beta$, see \cref{fig: Zxy=s(Axy)}. Invoking \Cref{lem:epsilon chain slack} once more, we are ensured that $ \alpha $ begins $ \kappa_c(1+\kappa_c) \varepsilon $-close to $ y $, is asymptotic to $ A_+x $ and has slack satisfying
		\[ \left| \s(\alpha)-\s(n.A_{[-s_1,s_2]}z)-\s((n+m).\beta) \right|<\kappa_c\varepsilon, \]
		and thus also
		\[ \left| \s(\alpha)-T \right|<(1+2\kappa_c+\kappa_c^2)\varepsilon. \]
		Having $ \varepsilon $ arbitrary, we conclude from \Cref{lem:slack gives points in Nxbar} that $ T \in \Zxy{x}{y} $. Since $\Zxy{x}{y}$ is closed we thus have $\Zxy{x}{y} \supseteq \overline{\s(\Axy{x}{y})}$.
		
		\begin{figure}
			\centering
			\includegraphics[width=0.98\linewidth]{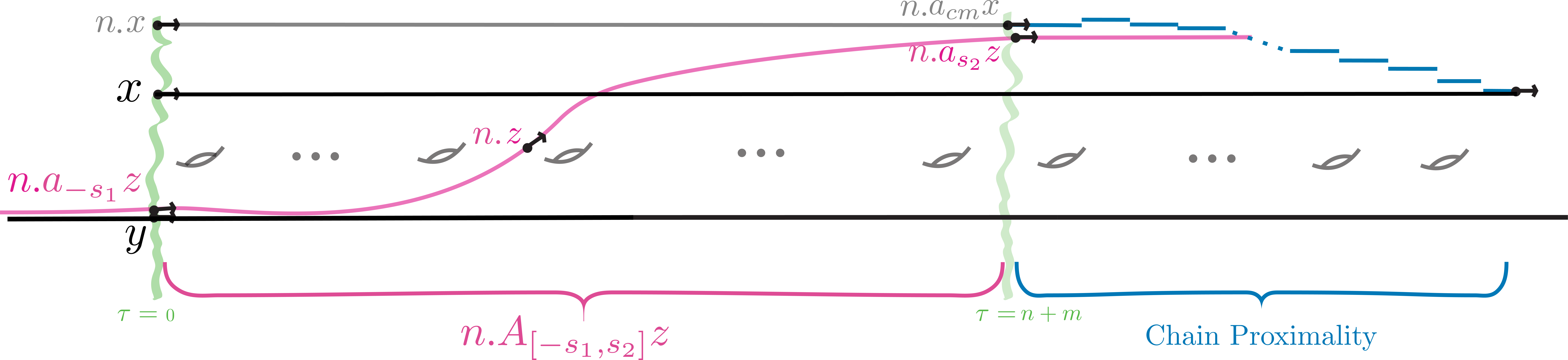}
			\caption{}
			\label{fig: Zxy=s(Axy)}
		\end{figure}
		
		The other inclusion, $ (\subseteq) $, follows from \Cref{lem:slack gives points in Nxbar} by extending the one-sided geodesic rays $ \alpha_m $ back towards $ A_-y $ making them into elements of $ \Axy{x}{y} $ with slack equal to $\s(\alpha_m)$ up to an arbitrarily small error as $ m\to 0 $. 
		\medskip
		
		Now we consider the case where either $x$ or $y$ do not  project into a minimal component of $\lambda_d$. Recall that such a case corresponds to points which are asymptotic, in both forward and backward time, to leaves of minimal components. We claim that if $y'$ is any point in $\T^1_+\lambda \cap \tau^{-1}(0)$ which is backward asymptotic to $Ay$, and $x'$ any point $\T^1_+\lambda\cap \tau^{-1}(0)$ which is forward asymptotic to $Ax$ then
		\[ \Zxy{x'}{y'}=\Zxy{x}{y} \quad \text{and}\quad \Axy{x'}{y'}=\Axy{x}{y}, \]
		thus reducing the proof to the case already proven. 
		
		The fact that $\Axy{x'}{y'}=\Axy{x}{y}$ follows from the definition, since both sets contain those arcs connecting the past of $y$ to the future of $x$. 
		
		For the other identity, since $\tau(a_tx')=\tau(a_tx)=\tau(x)+t$ for all $t \in \R$ and $A_+x'$ is asymptotic to $A_+x$, we conclude that $x$ and $x'$ are asymptotic, that is, $x' \in Nx$. Hence $\Zxy{x'}{y'}=\Zxy{x}{y'}$. On the other hand, $Ay'$ and $Ay$ project onto two leaves of the same minimal component in $\lambda_d$, therefore by \Cref{cor: finite cover} and \Cref{lem:chain prox implies inclusion in Nxbar} we know that $\overline{Ny}=\overline{Ny'}$ and hence $\overline{Na_ty}=\overline{Na_ty'}$ for all $t$. This implies that $a_ty \in \overline{Nx}$ if and only if $a_ty' \in \overline{Nx}$, or in other words that $\Zxy{x}{y'}=\Zxy{x}{y}$. 
	\end{proof}

    \begin{cor}\label{Cor: Zxx is ray if infinite leaf}
		Let $ x \in \Y $ be any point tangent to a weak component $\mu$ of $\lambda$ that is not a periodic line. %
        Then $ \Zxy{x}{x}=[0,\infty) $.
	\end{cor}

	\begin{proof}
		Recall that the case where $ \mu $ has countably many leaves was covered in \Cref{thm: chain recurrent leaf}.
		
		Assume that $\mu$ has uncountably many leaves and choose $x$ on a lift of a minimal sublamination $\mu'$ that is not a closed curve.  Since $ \Zxy{x}{x} $ is a closed semigroup, it suffices to show that $ \Zxy{x}{x} $ contains arbitrarily small positive elements. 
		
		Since $ \mu' $ is the preimage of a minimal sublamination with no isolated leaves, the point $ x $ is not isolated in $ \Y $ and moreover there exists a point $ y \in \Y $ at distance $ d(x,y)<\varepsilon $ which is neither forward nor backward asymptotic to $ x $ in $ \T^1\Sigma $. This implies that the geodesic $ \alpha $ constructed by connecting $ A_-x $ to $ A_+y $ (along the shortest path between $ x $ and $ y $) is not contained in $ \lambda $. In particular, $ \s(\alpha) > 0 $. By \Cref{lem:epsilon chain slack} we moreover know that $ \s(\alpha)< \kappa_c\varepsilon $. By \Cref{prop:Zxy closure of s(Axy)}, we conclude that $ \s(\alpha) \in \Zxy{y}{x} $. 
		
		Applying the argument above, but switching the roles of $ x $ and $ y $, we conclude that there exists a geodesic $ \beta \in \Axy{x}{y} $ with $ \s(\beta) < \kappa_c \varepsilon $ and $ \s(\beta) \in \Zxy{x}{y} $.
		By \eqref{eqn: subcontainment for Zxy},
		\[ \Zxy{x}{x} \supset \Zxy{y}{x}+\Zxy{x}{y}, \]
		and hence conclude that $ \Zxy{x}{x} $ contains 
		\[ 0<\s(\alpha)+\s(\beta)<2\kappa_c\varepsilon, \]
		implying $\Zxy{x}{x}=[0,\infty)$.
		\medskip
		
        Now for arbitrary $x\in \T^1_+\mu$, take $y$ tangent to $\mu'\subset \mu$ as above.
        By \Cref{thm: chprox equiv chain recurrent} and \Cref{lem:epsilon chain slack} there exist arcs in $\Axy{y}{x}$ and in $\Axy{x}{y}$ each having slack smaller than $\varepsilon$, for an arbitrary $\varepsilon>0$. By the previous proposition and \eqref{eqn: subcontainment for Zxy} we thus have
		\[ \Zxy{x}{x} \supset \Zxy{y}{x}+\Zxy{y}{y}+\Zxy{x}{y} \supset [2\varepsilon,\infty), \]
		concluding the proof.
	\end{proof}
    
	\section{Structure of Horocycle Orbit Closures}\label{sec: structure of horo orbit closures}
	
	In this section we integrate all previous results into a complete description of horocycle orbit closures for any hyperbolic metric on $\Sigma_0$. As described in the introduction, the structure of horocycle orbit closures is read off of a directed graph with associated weights --- the \emph{Slack Graph}:
	
	\begin{definition}
		Let $\mu_1,...,\mu_k$ be the weak connected components of $\lambda$, as discussed in \Cref{thm: chprox equiv chain recurrent} and \Cref{cor: chprox and weak components}. 
		Fix a choice of $ x_i \in \T^1_+\mu_i \cap \tau\inverse(0)$, and define the following directed graph $ \mathcal{G} $, having finitely many vertices and infinitely many edges:
		\begin{itemize}%
			\item the vertex set  $V(\mathcal G)$ is $\{x_1,...,x_k\} \subset \mathcal{Y}$. 
			\item the set of directed edges from vertex $y$ to $x$ is $\Axy xy$, the set of bi-infinite geodesics that are asymptotic to $ Ay $ in backward time and to $ Ax $ in forward time.
		\end{itemize}
	\end{definition}
	
	As before, let $ \tau: \Sigma \to \R $ be a $1$-Lipschitz tight map whose maximal stretch locus is equal to $\lambda$, see \Cref{chain recurrence always}. %
	
	\subsection{Marked Busemann function}
	
	A first step in our reduction is identifying horocycle orbit closures of quasi-minimizing points according to the value of a marked Busemann function. 
	
	Recall our definition of slack of a geodesic ray $A_+x$:
	\[ \s(A_+x)=\lim_{t \to \infty} t - (\tau(a_t x)-\tau(x)). \]
	
	\begin{definition}
		Given a point $x \in \T^1\Sigma$ we define $\beta_+:\T^1\Sigma \to \R$ by
		\[ \beta_+(x)=\tau(x)-\s(A_+x).\]
	\end{definition}
	
	This function was discussed in \cite[Section 6]{FLM}\footnote{Note that our definition here of $\beta_+$ agrees with $\beta_+(x)=\lim_{t\to \infty} \tau(a_t x)-t$ used therein.} where we proved that $\beta_+$ is $N$-invariant, upper semi-continuous and satisfies
	\[ \beta_+(x) > -\infty \quad \text{if and only if} \quad x \in \Qm_+. \]
	Moreover, $ \beta_+(a_t x)=\beta_+(x)+t $ for all $x \in \T^1\Sigma$ and $ t \in \R $.
	\medskip
	
	Recalling \cite[Thm. 3.4]{FLM}, we know that every quasi-minimizing point $x \in \Qm$ has its geodesic ray asymptotic to $\T^1\lambda$. The decomposition of $\lambda$ into weakly connected components ensures that every $x \in \Qm$ is asymptotic to exactly one component among $\mu_1,...,\mu_k$, leading to the following well-defined function:
	\begin{definition}
		To every $x \in \Qm$ we associate $\v(x)=x_i \in V(\G)$, where $x_i$ is the corresponding representative of the unique weakly connected component $\mu_i$ to which $A_+x$ is asymptotic.
	\end{definition}
	It is easy to see that the function $\v:\Qm \to V(\G)$ is also $N$-invariant.
	
	\newcommand{\bhatp}{\hat{\beta}_+}
	We are now set to discuss the (positive) \emph{marked Busemann function} as defined in the following theorem:
	
	\begin{theorem}\label{Thm: the marked Busemann function}
		The function $\bhatp: \Qm_+ \to \R \times V(\G)$ given by
		\[ \bhatp(x)=(\beta_+(x),\v(x)) \]
		uniquely identifies the horocycle orbit closure of $x$. That is, for all $x,y \in \Qm_+$
		\[ \overline{Nx}=\overline{Ny} \quad 
		\text{if and only if}\quad \bhatp(x)=\bhatp(y). \]
		In particular,
		\[ \overline{Nz}=a_{\beta_+(z)} \overline{N\v(z)}. \]
	\end{theorem}

	One can see that the coordinate $\v$ determines a family of orbit closures, up to translation by $A$, whereas $\beta_+$ determines how ``deep'' the particular orbit closure is positioned.

	\begin{proof}
		Fix $x \in \Qm_+$ with $\v(x)=x_i$. We will show that 
		\begin{equation}\label{eq:Reduction from Qm to representatives x_i}
		\overline{N x}=a_{\beta_+(x)}\overline{N x_i}.
		\end{equation}
		Actually, by the $ A $-equivariance of $ \beta_+ $ and the fact that $ \overline{N a_t x}=a_t\overline{Nx} $ it suffices to prove the case where $ \beta_+(x)=0 $.
		
		Recall that two points $ y_1,y_2 \in \T^1\Sigma $ are $ A $-proximal if there exists a sequence $ t_n \to \infty $ for which $ d(a_{t_n}y_1,a_{t_n}y_2) \to 0 $. In \cite[Corollary 8.3]{FLM}, we showed that whenever two points are $A$-proximal in $ \T^1\Sigma $ then they have equal horocycle orbit closures. Moreover, the proof of \cite[Proposition 8.5]{FLM} implies that if $ a_t x $ is asymptotic to $ \T^1_+ \mu_i $ then there exists $ z \in \T^1_+\mu_i $ which is $ A $-proximal to $ x $. It is easy to see, from the definition of $ \beta_+ $ and the $ 1 $-Lipschitz property of $ \tau $, that any two $ A $-proximal points have the same $ \beta_+ $-value. Since $\beta_+(x)=0$ and $\beta_+=\tau$ on $\T^1\lambda$, we conclude that there exists $ z \in \Y \cap \T^1_+\mu_i $ satisfying $ \overline{N x}=\overline{Nz} $.
		
		Now consider the two points $ z $ and $ x_i $ in $ \Y \cap \T^1_+\mu_i $. By \Cref{thm: chprox equiv chain recurrent,cor: chprox and weak components} these two points are chain-proximal. 
  By \Cref{lem:chain prox implies inclusion in Nxbar} we conclude that $ \overline{N z} = \overline{N x_i} $.
		
		This in particular implies that if $\bhatp(x)=\bhatp(y)$ then $\overline{Nx}=\overline{Ny}$.
		\medskip
		
		In the other direction, assume that $ \overline{Nx}=\overline{Ny} $. Since $ \beta_+ $ is $ N $-invariant and upper semi-continuous we know that the set $ (\beta_+)^{-1}\left([\beta_+(x),\infty)\right) $ is closed and $ N $-invariant and hence contains $ \overline{Nx} $ and $ y $. In other words, $ \beta_+(y) \geq \beta_+(x) $. Symmetry of this argument implies that $ \beta_+(x)=\beta_+(y) $.
		
		Now assume in contradiction that $ \v(x)=i\neq j= \v(y) $. Let $ x_i $ and $ x_j $ be the representatives of $ \mu_i $ and $ \mu_j $ in $ V(\G) $. 
		
		Since there is a definite $ \varepsilon > 0 $ distance between $ \T^1_+\mu_i $ and $ \T^1_+\mu_j $ in $\T^1\Sigma$, any arc in $ \Axy{x_i}{x_j} $ has slack greater than some $ \delta >0 $ (see Lemma \ref{small slack close to L}) %
        By \Cref{prop:Zxy closure of s(Axy)}, this implies that $ 0 \notin \Zxy{x_i}{x_j} $ and hence $ x_j \notin \overline{Nx_i} $. By \eqref{eq:Reduction from Qm to representatives x_i} we conclude that $ y \notin \overline{N x} $, contradicting our assumption.
	\end{proof}

	In fact, the above theorem reduces the entire analysis to the finite set of representatives $V(\G)$:
	
	\begin{cor}\label{Cor:Orbit closure via representatives}
		For any $ x_i \in V(\G) $
		\[ \overline{Nx_i}=\bhatp^{-1} \left( \bigcup_{x_j \in V(\G)} \left(\Zxy{x_i}{x_j}\times\{x_j\}\right) \right). \]
	\end{cor}
	
	\begin{proof}
		Notice the following simple observation --- If $ \overline{Nx}=\overline{Ny} $ then for any~$ z $
		\begin{equation}%
		x \in \overline{N z} \iff y \in \overline{N z},
		\end{equation}
		which follows from the fact that $ x \in \overline{Nz} $ implies $ \overline{Nx} \subset \overline{Nz} $, and similarly for $ y $. 
		
		Now fix some $ 1\leq i \leq k $ and let $ y \in \Qm_+ $ be any point. Denote $\bhatp(y)=(T,x_j)$. Since 
		\[ \overline{N y} = a_T\overline{N x_{j}} = \overline{N a_T x_j}, \]
		we deduce that $ y \in \overline{N x_i} $ if and only if $ a_T x_j \in \overline{N x_i} $. In other words 
		\[ y \in \overline{N x_i} \iff T \in \Zxy{x_i}{x_j}. \]
		Hence the $ \bhatp $ values completely determine the inclusion in $ \overline{N x_i} $, and therefore we may rewrite the above equivalence as 
		\[ y \in \overline{N x_i} \iff \bhatp(y) \in \Zxy{x_i}{x_j}\times \{x_j\}, \]
		proving the corollary.
	\end{proof}

	\subsection{The `$-$'-end}\label{subsec: minus end}
	
	We have thus far focused our attention on the `+' end of $\Sigma$, analyzing those orbit closures contained in $\Qm_+$. This sign attribution is obviously arbitrary and all previous theorems would have applied just as well to the negative end if one had considered $(-\tau)$ as the tight map instead of $\tau$.
	
	Leaving $\tau$ unchanged, the following definitions of the slack and Busemann functions reflect this simple observation:
	\[ \mathscr{S}_-(\alpha)=\length_\Sigma(\alpha)-(\tau(\alpha(a))-\tau(\alpha(b))) \]
	for any rectifiable curve $\alpha: [a,b] \to \T^1\Sigma$, and respectively
	\[ \beta_-(x)=-\tau(x)-\mathscr{S}_-(x) \quad \text{for any } x \in \Qm_-.\footnote{Note that this definition of $\beta_-$ differs in sign from the one given in \cite{FLM}.} \]
	
	Given a point $x \in \T^1\Sigma$ we'll denote by $-x$ its involution, that is, the element with the same basepoint but antipodal direction ( in particular, $x \mapsto -x$ flips the the orientation of all geodesics). 
  Note that if $Az \in \Axy{x}{y}$ with $x,y \in \T^1_+\lambda$, then $z\in \Qm_+$ and $ - z \in \Qm_-$; additionally, $A(-z) \in \Axy{- y}{- x}$ where $- x, - y \in \T^1_-\lambda $.
  Moreover,
	\[ \s(Az)=\mathscr{S}_-(A(-z)). \]
	
	Therefore, we may consider the slack graph $\G_-$ whose vertices are $- x_i$, for $i=1,...,k$ and whose edges are exactly those edges of $\G$ but with flipped orientation. Corresponding slacks of edges are unchanged after this reorientation and use of $\mathscr{S}_-$.

	One amusing consequence is that given any point $ y \in \Y $, its involution $ - y $ facing the opposite direction would satisfy
	\[ \Zxy{y}{y}=\Zxy{- y}{- y}.\footnote{A very different proof of a similar statement was given as part of \cite[Prop. 8.6]{FLM}.} \]

	We may also draw the following corollary:
	\begin{cor}
		Let $k$ be the number of weakly connected components of $\lambda$. Then up to $A$-translation, the number of distinct $N$-orbit closures in $\T^1\Sigma$ is equal to $2k+1$.
	\end{cor}

	\begin{proof}
		By \Cref{Thm: the marked Busemann function}, $\Qm_+$ and $\Qm_-$ each provide us with $k$ distinct families. Dense horocycles provide us with the last type of orbit closure.
	\end{proof}

	\subsection{Reading the structure of $\Zxy{x_i}{x_j}$ off of the slack graph}\label{subsec: slack graph G}
	As in \S\ref{subsec: isolated multi-curve}, consider $ \Gimc $, the isolated multi-curve part of $ \G $, that is, the induced subgraph of $ \G $ on those vertices $ x_i $ for which $ \mu_i $ in $ \Sigma $ is a lift of a closed geodesic in $\Sigma_0$. Let $ \s: \Hom_{\mathcal G} \to \R $ be the homomorphic extension of the slack function defined in \S\ref{sec: finite lamination}.
	
	We may read off the structure of the shift sets $ \Zxy{x}{y} $ from the graph $ \G $:
	
	\begin{theorem}\label{Thm:the ray part in Zxy}
		Given $ x_i,x_j \in V(\G) $ let
		\[ \rho_{j,i} = \inf\{ \s(\underline{\alpha}) : \underline{\alpha} \in \Hom_{\G}(x_j,x_i)\smallsetminus \Hom_{\Gimc}(x_j,x_i) \}, \]
		be the infimal slack value over all edge-paths in $ \G $ from $ x_j $ to $ x_i $ which pass through a vertex outside of $ \Gimc $. Then
		\[ \Zxy{x_i}{x_j} = \s\left(\Hom_{\Gimc}(x_j,x_i)\right) \cup [\rho_{j,i},\infty). \]
	\end{theorem}

	The combination of \Cref{Cor:Orbit closure via representatives} and \Cref{Thm:the ray part in Zxy} shows that the graph $ \G $ together with the associated slack values hold all the information needed to describe the structure of all horocycle orbit closures. 
    
    As a preliminary result we state the following:
    \begin{prop}\label{prop:closure of s(Hom xy) is Zxy}
    	For any $ x_i,x_j \in V(\G) $ we have
    	\[ \Zxy{x_i}{x_j}=\overline{\s\left(\Hom_{\G}(x_j,x_i)\right)}. \]
    \end{prop}

	\begin{proof}
		By \Cref{prop:Zxy closure of s(Axy)}, we know that 
		\[ \Zxy{x_i}{x_j}=\overline{\s(\Axy{x_i}{x_j})}. \]
		Therefore the inclusion $ \Zxy{x_i}{x_j} \subseteq \overline{\s\left(\Hom_{\G}(x_j,x_i)\right)} $ holds by definition. 
		
		For the other inclusion, let $ \alpha_1\cdots \alpha_n $ be an edgepath in $ \Hom_{\G}(x_j,x_i) $ with 
		\[ x_j=y_0,y_1,\dots,y_{n-1},y_n=x_i \]
		being the vertices along the path and $ \alpha_k \in \Axy{y_k}{y_{k-1}} $. Using \Cref{prop:Zxy closure of s(Axy)} once again we know that
		\[ \s(\alpha_1\cdots \alpha_n)=\s(\alpha_1)+\cdots+\s(\alpha_n) \in \sum_{k=1}^n \Zxy{y_k}{y_{k-1}}. \]
		By \eqref{eqn: subcontainment for Zxy} we conclude $ \s(\alpha_1\cdots\alpha_n) \in \Zxy{x_i}{x_j} $, as claimed.
	\end{proof}
	
	\begin{proof}[Proof of \Cref{Thm:the ray part in Zxy}]
		Fix $ x_i,x_j \in V(\G) $ and let $ \rho_{j,i} $ be as in the statement. Note that the case where $ \Hom_{\G}(x_j,x_i)\smallsetminus \Hom_{\Gimc}(x_j, x_i) = \emptyset $, and hence $ \rho_{j,i}=\infty $, was proven in \Cref{thm:structure of Zxy}.
		
		Now assume $\Hom_{\G}(x_j,x_i)\smallsetminus \Hom_{\Gimc}(x_j, x_i)\neq \emptyset$, and hence $ \rho_{j,i} < \infty $. The following two claims prove the required statement.
		
		\medskip
		
		\noindent \underline{\emph{Claim 1:}} $ [\rho_{j,i},\infty) \subset \Zxy{x_i}{x_j} $. 
		\medskip
		
		Given $ \delta>0 $, there exists an edgepath $ \alpha_1\cdots\alpha_n $ passing through the vertices $ x_j=y_0,\dots,y_n=x_i $, having $ \s(\alpha_1\cdots\alpha_n) \in [\rho_{j,i},\rho_{j,i}+\delta) $ and with $ y_{k_0} $ a vertex outside of $ \Gimc $, for some $0\leq k_0 \leq n$.
  
		In particular we know that
		\[ \s(\alpha_1\cdots\alpha_n) \in \sum_{k=1}^n \Zxy{y_k}{y_{k-1}} \subseteq \Zxy{x_i}{x_j}. \]
		Recall that since $ y_{k_0} \notin \Gimc $ then $ \Zxy{y_{k_0}}{y_{k_0}}=[0,\infty) $ by \Cref{Cor: Zxx is ray if infinite leaf}. Hence
  (again by \eqref{eqn: subcontainment for Zxy})
		\[ \s(\alpha_1\cdots\alpha_n)+[0,\infty)\subseteq \sum_{k=1}^{k_0}\Zxy{y_k}{y_{k-1}}+\Zxy{y_{k_0}}{y_{k_0}}+\sum_{k=k_0+1}^n\Zxy{y_k}{y_{k-1}} \subseteq \Zxy{x_i}{x_j}. \]
        Since $ \delta $ was arbitrary and $ \Zxy{x_i}{x_j} $ is closed we conclude the claim.
		\medskip
		
		\noindent \underline{\emph{Claim 2:}} $ \Zxy{x_i}{x_j} \cap [0,\rho_{j,i})= \s\left(\Hom_{\Gimc}(x_j,x_i)\right) \cap [0,\rho_{j,i}) $.
		\medskip
		
		First note that if either $x_i$ or $x_j$ are contained in $\G\smallsetminus\Gimc$ then $\rho_{j,i}=\inf \s(\Hom_{\G}(x_j,x_i))$. By \Cref{prop:closure of s(Hom xy) is Zxy} we thus have $\rho_{j,i}=\min \Zxy{x_i}{x_j}$ and by the previous claim $\Zxy{x_i}{x_j}=[\rho_{j,i},\infty)$. This proves Claim 2 in this case as $\Hom_{\Gimc}(x_j,x_i)=\emptyset $.
		
		Now assume $x_j,x_i \in V(\Gimc)$. Note that $ \s\left(\Hom_{\Gimc}(x_j,x_i)\right) \subseteq \Zxy{x_i}{x_j} $ follows immediately from \Cref{prop:closure of s(Hom xy) is Zxy} above.
		
		Recall the notation $ \T^1_+\lambda^{\infty}$ for the subset of $ \T^1_+\lambda $ which is the lift of all the components of $ \lambda_0 $ containing an infinite leaf. Suppose $ s_0 \in \Zxy{x_i}{x_j} \cap [0,\rho_{j,i}) $. By \Cref{lem:slack gives points in Nxbar}, there exists a sequence of geodesic rays $ \alpha_m $ beginning at $ \alpha_m(0) $, with $ \alpha_m(0)\to x_j $, and forward asymptotic to $Ax_i$,  and having slacks $ \s(\alpha_m) \to s_0 $. 

  \medskip 
  \noindent \underline{\emph{Claim 2a:}} the rays $ \alpha_m $ avoid some $ \varepsilon $-neighborhood of $\T^1_+\lambda^\infty $ for all large enough $ m $.

  \medskip
  \Cref{lem: isolated multi-curve slack - arcs outside nbhd of non-imc} tells us that if Claim 2a holds then $ s_0 \in \s\left(\Hom_{\Gimc}(x_j,x_i)\right) $, implying Claim 2; so it remains to establish 2a.
		
		Assume in contradiction that for all $ \varepsilon >0 $ there exist arbitrarily large $ m $ for which $ \alpha_m $ intersects $ (\T^1_+\lambda^\infty)^{(\varepsilon)} $. Fix $ \varepsilon < \frac{\rho_{j,i}-s_0}{6\kappa_c} $ and let $ m $ be large enough to satisfy 
		\[ \s(\alpha_m)<\frac{\rho_{j,i}+s_0}{2}, \quad \alpha_m \cap (\T^1_+\lambda^\infty)^{(\varepsilon)} \neq \emptyset \quad \text{and} \quad d(\alpha_m(0),x_j)<\varepsilon.  \]
		In particular, let $ y \in \T^1_+\lambda^\infty $ and $ T>0 $ satisfy $ d(\alpha_m(T),y)<\varepsilon $. 
		
		Consider the geodesic $ \eta_1 $ constructed by connecting and straightening
		\[ A_-x_j \cup \alpha_m|_{[0,T]} \cup A_+y, \]    
		and the geodesic $ \eta_2 $ constructed from
		\[ A_-y \cup \alpha_m|_{[T,\infty)}. \]
		Thus $ \eta_1 \in \Axy{y}{x_j} $ and $ \eta_2 \in \Axy{x_i}{y} $.
		
		By  \Cref{lem:epsilon chain slack}, we are ensured that
		\[ \left|\s(\eta_1)-\s(A_-x_j) - \s(\alpha_m|_{[0,T]})\right|<2\kappa_c\varepsilon \] 
		and
		\[\left| \s(\eta_2)-\s(A_-y) - \s(\alpha_m|_{[T,\infty)}) \right|<\kappa_c\varepsilon. \]
		Since $ \s(A_\pm y)=\s(A_-x_j)=0 $, we conclude that
		\[ \left| \s(\eta_1)+\s(\eta_2)-\s(\alpha_m) \right|<3\kappa_c\varepsilon<\frac{\rho_{j,i}-s_0}{2}, \]
		implying in particular that
		\[ \s(\eta_1)+\s(\eta_2)<\s(\alpha_m)+\frac{\rho_{j,i}-s_0}{2}<\rho_{j,i}, \]
		by our choice of $ \varepsilon $ and $ m $.
		
		At this point, if we knew that $ Ay\cap \tau^{-1}(0) \in V(\G) $, then we would have obtained an edgepath $ \eta_1\cdot \eta_2 $ connecting $ x_j $ to $ x_i $, passing outside of $ \Gimc $ and having slack strictly smaller than $ \rho_{j,i} $, contradicting the definition of $ \rho_{j,i} $. Nonetheless, if $ \{y'\} = Ay\cap \tau^{-1}(0) $ and if $ z \in V(\G) $ is the representative of the component of $ y $ and $ y' $, then 
		\[ \Zxy{y'}{x_j}=\Zxy{z}{x_j} \quad \text{and} \quad \Zxy{x_i}{y'}=\Zxy{x_i}{z}, \]
		e.g.~by \Cref{Cor:Orbit closure via representatives} and the fact that $ \beta_+(y')=\beta_+(z) $. Hence by \Cref{prop:Zxy closure of s(Axy)} we know there are arcs $ \eta_1' \in \Axy{z}{x_j} $ and $ \eta_2' \in \Axy{x_i}{z} $ with slacks arbitrarily close to $ \s(\eta_1) $ and $ \s(\eta_2) $, leading again to a contradiction. This proves Claim 2a and hence Claim 2 and the theorem.
	\end{proof}
	
	In light of \Cref{thm: countable filtration Zxy} we draw the following corollary:
	\begin{cor}\label{cor:depth at the beginning of Zxy}
		Using the notation of \Cref{Thm:the ray part in Zxy} we have
		\[ \s(\Hom_\G^{(i+1)}(x_j, x_i))\cap (0,\rho_{j,i}) = (\Zxy {x_i}{x_j})^{(i)} \cap (0,\rho_{j,i}), \]
		where $\Hom_\G^{(i+1)}$ corresponds to edgepaths in $\G$ of length$\geq i+1$ and where $(\Zxy {x_i}{x_j})^{(i)}$ is the $i$-th derived set of $\Zxy {x_i}{x_j}$.
	\end{cor}
	
	\begin{proof}
		The proof of this corollary is identical to the one given for \Cref{thm: countable filtration Zxy}, with one added ingredient --- Claim 2a above. That is, whenever a sequence of rays $(\alpha_m)$ has slack $\lim_{m\to \infty}\s(\alpha_m)<\rho_{j,i}$ then there exists an $\varepsilon>0$ such that $\alpha_m$ avoids $(\T^1_+\lambda^\infty)^{(\varepsilon)}$ for all large $m$. This in turn implies that all geometric limit chains extracted from the sequence avoid this neighborhood too and hence correspond to an edgepath in $\Gimc$.
	\end{proof}
	
	\begin{remark}
		Under the assumption that $ \tau $ has maximal stretch locus equal to $ \lambda $, we conclude that there exists a uniform $ \delta>0 $ such that for all $ i\neq j $ 
		\[ \Zxy{x_i}{x_j} \subseteq [\delta,\infty). \]
		This is because all arcs in $ \Axy{x_i}{x_j} $ have to spend some definite amount of time a definite distance away from $ \T^1\lambda $. This in particular implies the following statements:
		\begin{enumerate}[label=(\roman*)]
			\item $ \Zxy{x_i}{x_i}=[0,\infty) $ if and only if $ \mu_i $ contains an infinite leaf.
			\item By \Cref{cor:depth at the beginning of Zxy}, the depth of $ \Zxy{x_i}{x_i}\cap [0,\rho_{j,i}) $ is bounded by $ \lceil \frac{\rho_{j,i}}{\delta} \rceil $.
		\end{enumerate}
	\end{remark}

\subsection{Dichotomy}\label{subsec: dichotomy}
	Another facet of the dichotomy stated in \Cref{thm: dichotomy intro}  has to do with the notion of Garnett points:
	\begin{definition}[e.g.~\cite{Sullivan}]
		A point $\xi $ in the limit set of a Fuchsian group $\Gamma$ is called Garnett if there is a maximal closed horoball centered at
		$\xi$ in $\Hplane$ disjoint from the $\Gamma$ orbit of a point $p \in \Hplane$ and any larger horoball contains infinitely many $\Gamma$-orbits.
	\end{definition}
	
	Recall the classical Busemann function $B:\partial\Hplane \times \Hplane \times \Hplane \to \R$ defined by
	\[ B_\xi(z,w)=\lim_{t \to \infty} d_{\Hplane}(z,\alpha(t))-d_{\Hplane}(w,\alpha(t)), \]
	where $\alpha: [0,\infty) \to \Hplane$ is any geodesic ray ending at $\xi$. An equivalent definition of a limit point being Garnett is that
	\[ B_\xi(z,\gamma.z) < \sup_{\gamma' \in \Gamma} B_\xi(z,\gamma'.z) < \infty \quad \text{for all } z\in \Hplane \text{ and } \gamma \in \Gamma.  \]
	In light of this definition, one can readily verify that a quasi-minimizing point $\xi$ is Garnett if and only if there does not exist a minimizing geodesic ray in $\Hplane/\Gamma$ whose lift ends at $\xi$.
	\medskip
	
	We are now set to fully state and prove the dichotomy:
    \begin{theorem}
    	There is a dichotomy.
    	\begin{enumerate}[label = (\alph*)]
    		\item $\lambda_0$ is a simple multi-curve: for all $x \in \Qm$, $\overline {Nx}$ is a countable union of horocycles, hence has Hausdorff dimension $1$.
    		The set of endpoints of quasi-minimizing rays in $\partial\Hplane$ is countable, and contains no Garnett points. 
    		\item $\lambda_0$ contains an infinite leaf: for all $x\in \Qm$, $\overline{Nx}$ has Hausdorff dimension $2$ and $\overline {Nx} \cap A_+x$ contains a ray.
    		The set of endpoints of quasi-minimizing rays in $\partial\Hplane$ is an uncountable set with Hausdorff dimension 0 which contains uncountably many Garnett points.
    	\end{enumerate}
    \end{theorem}  
    
    \begin{proof}
    	The case where $ \lambda_0 $ is an isolated multi-curve was covered in \Cref{cor: multi-curve structure}. Additionally, since all quasi-minimizing points in this case are asymptotic to a leaf of $ \lambda $, all such limit points have a corresponding minimizing ray implying they are not Garnett.
    	
    	Now assume $ \lambda_0 $ is not an isolated multi-curve. Then by \Cref{Cor:Orbit closure via representatives} and \Cref{Thm:the ray part in Zxy}, for any quasi-minimizing point $ y \in \Qm $ the recurrence semigroup $ \Zxy{y}{y} $ contains a ray. This implies that $ \overline{Ny} $ contains a subset of the form $ A_{(t_1,t_2)}Ny $ which has Hausdorff dimension 2. On the other hand, by \cite[Cor. 1.5]{FLM} we know that $ \Qm $, which contains $ \overline{N y} $, has Hausdorff dimension 2. This implies $ \dim \overline{N y}=2 $, as claimed.
    	
    	By \Cref{Cor: Zxx is ray if infinite leaf}, there exists a point $ x \in \T^1\lambda $ with $ \Zxy{x}{x}=[0,\infty) $. In particular, there exist arcs in $ \Axy{x}{x} $ having arbitrarily small positive slack. Fix some sequence $ (\alpha_m) \subseteq \Axy{x}{x} $ having summable slacks, that is $ \sum_{m \in \N} \s(\alpha_m) < \infty $. Making use of close returns of $ A_+x $ to $ \Z.x $ allows us to chain together countably many long intervals from  such arcs, generating a quasi-minimizing ray which is not asymptotic to any leaf of $ \lambda $. Such a ray corresponds to a Garnett limit point, having no minimizing representative. Clearly, one can generate in such a way uncountably many distinct Garnett points (e.g.~by permuting the elements of the sequence $ (\alpha_m) $).
    \end{proof}
    
\subsection{Examples}

    The goal of this subsection is to exhibit further \emph{non-rigidity} properties of $N$-orbit closures in $\Z$-covers as we vary the metric on the closed surface, downstairs.
    Theorem 5.6 of \cite{FLM} provides a convergent sequence of marked hyperbolic hyperbolic structures $\Sigma_{0,m}  \to \Sigmazero$ with corresponding $\Z$-covers $\Sigma_m$ and $\Sigma$ with the properties that the minimizing laminations $\lambda_{0,m} \subset \Sigma_{0,m}$ have finitely many leaves, while $\lambda_0\subset \Sigma_0$ is minimal and filling with uncountably many leaves. %
    The main results in this paper give that every non-maximal $N$-orbit closure in $\T^1\Sigma_m$ has Hausdorff dimension $1$, while in $\T^1\Sigma$, non-maximal $N$-orbit closures have Hausdorff dimension $2$.  

    In this subsection, we modify the construction from \cite[\S5.5]{FLM} slightly to produce a convergent sequence $\Sigma_{0,m}\to \Sigma_0$ with corresponding $\Z$-covers $\Sigma_m$ and $\Sigma$ where $\dim \overline {Nx} = 2$ for all quasi-minimizing points $x \in \Sigma_m$ and $x\in \Sigma$, but where the initial part of the recurrence semi-group $\Zxy {x_m}{x_m}$ has arbitrarily large, finite depth for certain $x_m \in \T^1_+\lambda_m \subset \T^1\Sigma_m$.
    Meanwhile $\Zxy xx = [0,\infty)$ for all $x \in \T_+\lambda\subset \T^1 \Sigma$.
    Furthermore, the number of distinct $N$-orbit closures (up to $A$-action) in $\T^1\Sigma_m$ grows without bound, while in $\T^1\Sigma$, there are exactly $4$, up to translation by $A$.

    As in our previous paper, these examples are produced from certain interval exchange transformations via the \emph{orthogeodesic foliation} construction and its continuity properties studied in \cite{CF:SHSH, CF:SHSHII}. 

	\begin{remark}\label{rmk: extend construction of stretch laminations}
		Note that the constructions described in \cite[Theorem 5.3]{FLM} were only stated for laminations supporting a transverse measure of full support. However, as chain recurrent laminations are Hausdorff limits of the supports of measured laminations, and the \emph{orthogeodesic foliation construction} from \cite{CF:SHSH,CF:SHSHII} is continuous in the Hausdorff topology, a limiting argument gives the following statement:
        
        Let $S_0$ be a closed, oriented surface, let $c>0$, let $\varphi\in H^1(S_0,c\Z)$, and let $\lambda_0$ be an oriented chain recurrent geodesic lamination on $S_0$.  Suppose $\varphi$ is Poincar\'e dual to a multicurve $\alpha$ with positive $c\Z$-weights that meets $\lambda_0$ transversely and positively and such that $S_0\setminus (\lambda_0 \cup \alpha)$ is a union of pre-compact disks.  Then there is a hyperbolic metric $\Sigma_0$ on $S_0$ and a  $1$-Lipschitz tight map $\tau_0: \Sigmazero \to \R/c\Z$ inducing $\varphi$ on homology with stretch set equal to $\lambda_0$. 
        
	\end{remark}

    For a closed, oriented surface $S_0$, we denote by $\mathcal T(S_0)$ the Teichm\"uller space of homotopy classes of marked hyperbolic structures on $S_0$.
    For the purpose of the following theorem, say that a geodesic lamination is \emph{perfect} if it is minimal and has no isolated leaves.
	\begin{theorem}\label{thm: examples}
	    Given any non-trivial homotopy class $S_0 \to S^1$ with corresponding $\Z$-cover $S \to S_0$, there is a sequence $\Sigma_{0,m} \in \mathcal T(S_0)$ converging to $\Sigma_0\in \mathcal T(S_0)$ with corresponding locally isometric $\Z$-covers $\Sigma_m\to \Sigma_{0,m}$ and $\Sigma \to \Sigma_0$ satisfying the following properties.
     \begin{enumerate}
         \item The minimizing lamination $\lambda_{0,m} \subset \Sigma_{0,m}$ consists of a minimal perfect component and a union of boundedly many simple closed curves.
         In $\Sigma_m$, the perfect component lifts to a weakly connected component of $\lambda_m$, but the number of uniformly isolated leaves in $\lambda_m$ grows without bound.
         \item $\lambda_0\subset \Sigmazero$ consists of $2$ minimal, perfect components, and $\lambda\subset \Sigma$ has $2$ weakly connected components. 
         \item We have convergence $\lambda_{0,m} \to \lambda_0$ in the Hausdorff topology on closed subsets of $S_0$ (with respect to an auxiliary negatively curved metric).
         \item There is a $\rho>0$ such that for all $m$ and for all $y_m$ forward-tangent to uniformly isolated leaves in $\lambda_m$, 
         $\Zxy {y_m}{y_m} \cap [0,\rho]$ is countable with finite depth, which goes to infinity with $m$.
     \end{enumerate}
     In particular, up to $A$-translation, the number of distinct $N$-orbit closures facing the `$+$'-end in $\T^1\Sigma_m$ grows without bound, but  in $\T^1\Sigma$ there are $2$.
	\end{theorem}

    \begin{remark}
        Item $(4)$ could be strengthened to say that for any $y_m$ and $z_m$ forward tangent to uniformly isolated leaves in $\lambda_m$ on the same $\tau_m$-fiber, $\Zxy {z_m}{y_m} \cap [0,\rho]$ is also countable with finite depth, tending to infinity with $m$.  The argument is more elaborate than we care to include, here.
    \end{remark}
    \begin{proof}
        We only give a sketch of the construction, referring the reader to \cite[\S5]{FLM} for more details.
        Let $T: I \to I$ be a weakly mixing interval exchange transformation,\footnote{The only condition that we are using is that all positive powers of $T$ are ergodic for the Lebesgue measure.} which exist for every irreducible permutation, e.g., by \cite{AF:weak_mixing}.
        Let 
        \[T' = T\sqcup T : I \sqcup I \to I \sqcup I\]
        be the (reducible) interval exchange transformation obtained by stacking two intervals one next to the other and applying $T$ to each.
        Consider the singular flat surface obtained by suspending $T'$ with constant roof function $c>0$ and gluing the remaining two edges by an orientation preserving isometry (consult Figure 8 in \cite{FLM}).
        Non-singular leaves of the  horizontal foliation of this singular flat structure $\omega$ correspond to orbits of $T'$.  Collapsing the leaves of the vertical foliation to points yields a (harmonic) map to the circle $\R/c\Z$.
        We can always find a $T$ such that this singular flat surface $\omega$ is topologically equivalent to $S_0$,\footnote{For certain small complexity examples, one must modify this construction slightly using irrational circle rotations, rather than a weakly mixing IET.} and the mapping class group $\text{Mod}(S_0)$ acts transitively on primitive integer cohomology classes $H^1(S_0, \Z)$, so we can assume that the homotopy class of maps to $\R/c\Z$ is a given one $S_0\to S^1$.
        \medskip

        Let $\lambda_0$ be the measured geodesic lamination obtained by straightening the leaves of the horizontal foliation of the singular flat metric given by $\omega$ from the previous paragraph.
        It has two components, each corresponding to a copy of $T: I \to I$; call them $\mu_1$ and $\mu_2$.
        Using \cite{CF:SHSH}, there is a unique hyperbolic metric $\Sigma_0\in \mathcal T(S_0)$ such that the orthogeodesic foliation $\cO_{\lambda_0}(\Sigma_0)$ is isotopic and measure equivalent to the vertical foliation on $\omega$.
        Collapsing the leaves of $\cO_{\lambda_0}(\Sigma_0)$ yields a tight map $\tau_0 : \Sigma_0 \to \R/c\Z$ with $\Stretch (\tau_0) = \lambda_0$.
        On $\Sigma_0$, there is a positive distance $d_0$ between the two components $\mu_1$ and $\mu_2$ of $\lambda_0$.  
        Since $T$ is weak-mixing, the chain proximality equivalence relation on $\lambda_0$ intersected with a $\tau_0$-fiber has two classes corresponding to the two $\mu_1$ and $\mu_2$; this follows from \S\S\ref{sec:chain prox min}--\ref{Section:Synchronous Transversal} and the fact that, for the first return system to a $\tau_0$ fiber intersected with $\T^1_+\lambda_0$, for a given point $x \in \T^1_+\mu_i$, the set of points $y$ that are proximal to $x$ is dense in $\T^1_+\mu_i$, for $i = 1,2$ (see \cite[Theorem 9.2]{FLM}).

        As in the proof \cite[Theorem 5.6]{FLM}, we can find a sequence of weighted multi-curves $\gamma_m$ contained in a snug train track neighborhood of $\mu_2$ that converge both in the Hausdorff topology and the measure topology to $\mu_2$.
        Furthermore, there are corresponding periodic interval exchange transformations $T_m : I \to I$ that converge to $T$ as $m \to \infty$ such that
        for 
        \[T_m' = T\sqcup T_m: I \sqcup I  \to I \sqcup I , \]
        the corresponding constant roof function $c>0$ suspension $\omega_m$ with its singular flat structure satisfies: 
        \begin{itemize}
            \item the horizontal foliation of $\omega_m$ is measure equivalent to $\lambda_{0,m}= \mu_1 \sqcup \gamma_m$.
            \item collapsing the vertical foliation $\omega_m \to \R/c\Z$ represents $S_0 \to S^1$.
            \item $\omega_m \to \omega$ as $m \to \infty$ (in the natural topology, e.g., that $\omega_m$ and $\omega$ are holomorphic $1$-forms on Riemann surfaces homeomorphic to $S_0$).
        \end{itemize}
        We have corresponding hyperbolic metrics $\Sigma_{0,m} \in \mathcal T(S_0)$ with $\cO_{\lambda_{0,m} }(\Sigma_{0,m})$ isotopic and measure equivalent to the vertical foliation of $\omega_m$ as well as $1$-Lipschitz tight maps $\tau_{0,m} : \Sigma_{0,m} \to \R/c\Z$.

        By \cite[Theorem A]{CF:SHSHII}, $\Sigma_{0,m}$ converges to $\Sigma_0$ in $\mathcal T(S_0)$ as $m\to \infty$.

        \medskip
        This completes the construction and establishes items (1) -- (3).  The only thing left to explain is item (4).
        Let $x_1$ and $x_2 \in \T^1_+\lambda$ be on leaves projecting to $\mu_1$ and $\mu_2 \subset \lambda_0$, respectively, in the same $\tau$-fiber.
        Define
        \[\rho_0 = \inf_{Az \in \Axy {x_2}{x_1}} \s(Az) + \inf_{Az \in \Axy {x_1}{x_2}} \s(Az).\]
        Since the distance between $\mu_1$ and $\mu_2$ is $d_0>0$ and $\Stretch (\tau_0) = \lambda_0$, we can conclude that $\rho_0>0$ (Lemma \ref{small slack close to L}).
        Define $\rho$ (from the statement of the theorem) as $\rho_0/2$.
        Using the results in \S\ref{subsec: slack graph G}, $\rho_0$ does not depend on the choices of $x_1$ or $x_2$.
        \medskip

        Let $y_m$ be a point forward tangent to a uniformly isolated leaf of $\lambda_m$.
        Since $\Sigma_{0,m}$ converges to $\Sigma_{0}$ as $m\to \infty$ and the laminations $\lambda_{0,m} \to \lambda_0$ converge in the Hausdorff topology,
        up to subsequence, there is a point $y'$ forward tangent to the weak component of $\lambda$ corresponding to $\mu_2$ such that 
        the triples
        \[(\T^1\Sigma_m, \T^1_+\lambda_m , y_m) \text{ converge geometrically to } (\T^1\Sigma,\T^1_+\lambda , y'), \text{ as $m\to \infty$. }\]
        In other words, near $y'$, the geometry of $\lambda \subset \Sigma$ is well approximated by the geometry of $\lambda_m \subset \Sigma_m$.

        By geometric convergence of triples, the uniformly isolated leaves of $\lambda_m \subset \Sigma_m$ get closer  to one another (but remain distance at least $d_0/2$ from the component corresponding to $\mu_1$ for large enough $m$) and grow in number as $m \to \infty$.
        Using Theorem \ref{Thm:the ray part in Zxy} and geometric convergence, we have that 
        \[\Zxy {y_m}{y_m} \cap [0, \rho]\]
        is a countable set for $m$ large enough.
        That is, although $\Zxy {y_m}{y_m}$ contains a ray, this ray does not begin until after $\rho$ (in fact its beginning is close to $\rho_0$, if $m$ is large enough).

        Now we find a small slack path joining the past of $y_m$ to its future.
        Given $\ep>0$ and $m$ large enough, there is a uniformly isolated leaf of $\lambda_m$ within $\ep$ of $y_m$.
        A path backwards asymptotic to $y_m$ that jumps to this nearby leaf, and then jumps back to $Ay_m$ when it is close has positive slack of size $O(\epsilon)$ (Lemma \ref{lem:epsilon chain slack}).  
        That this leaf does indeed come back close to $Ay_m$ follows from periodicity of $T_m'$.
        This proves that there is a generator of the semi-group $\Zxy {y_m} {y_m}$ smaller than $\ep$ for $m$ large enough.
        In particular, for $i$ such that $i<\rho/\epsilon$, $\s(\Hom_{\Gimc}^{(i)}(y,y))\cap [0,\rho]$ is non-empty, as it contains positive elements of size smaller than $i\epsilon<\rho$.
        
        By \Cref{cor:depth at the beginning of Zxy}, \[(\Zxy {y_m}{y_m})^{(i-1)}\cap [0,\rho] = \s(\Hom_{\Gimc}^{(i)}(y,y))\cap [0,\rho],\]
        and so $\Zxy {y_m}{y_m}\cap [0,\rho]$ has depth at least $\rho/\epsilon -1$ for large enough $m$.
        Since $\ep$ was arbitrary, this establishes item (4).
        \medskip

        A detailed argument would be rather cumbersome to write down and distract from the main argument, so we conclude our discussion, here.
    \end{proof}

	\subsection{Non-regularity of orbit closures}\label{subsec: non-regularity of orbit closures}
	
	In this subsection we briefly argue that non-maximal horocycle orbit closures are never (topological) submanifolds of $\T^1\Sigma$. We highlight several forms of irregularities, some orbit closures may exhibit more than one.
	
	First notice that whenever the orbit closure is a countable union of horocycles then it is not a manifold as locally, in small compact neighborhoods, it is a countable disjoint union of one-dimensional arcs.
	
	Otherwise, the distance minimizing lamination $\lambda \subset \Sigma$ contains a non-periodic leaf. Suppose $x \in \Qm_+$ and for simplicity assume $\beta_+(x)=0$, hence $\overline{Nx}=\overline{N\v(x)}$. There are two cases --- if $x_i=\v(x)$ corresponds to a point in $\Gimc$ (i.e.~it is asymptotic to a periodic geodesic in $\lambda$) then by \Cref{Thm:the ray part in Zxy} and the remark thereafter, we know that a small neighborhood of $x_i$ intersects $\overline{Nx}$ in a one-dimensional arc, whereas other parts of the orbit closure contain a two dimensional plane (e.g.~around the point $a_sx_i$ where $s>\rho_{i,i}$).
	
	If, on the other hand, $x_i=\v(x)\notin\Gimc$ then we know that $x_i \in \T^1_+\mu_i$ where $\mu_i$ is a weakly connected component of $\lambda$  containing a non-periodic leaf. In particular, $\mu_i$ contains infinitely many leaves (at least countably many of which are isometric copies of the non-periodic leaf). This implies that emanating from any basepoint in $\T^1_+\mu_i$ there are infinitely many quasi-minimizing rays having slack $<\varepsilon$, for any arbitrary $\varepsilon>0$, and which are asymptotic to $\mu_i$. 
	
	Recall that $\Zxy{x_i}{x_i}=[0,\infty)$ and consider the point $a_s x_i \in \overline{Nx}$, for some $s>0$. Consider the Iwasawa decomposition of $\PSL_2(\R)=NAK$ where $K\cong \operatorname*{PSO}(2)$ is the group of rotations around the basepoint ($A$ and $N$ as before). Since there are infinitely many quasi-minimizing rays emanating from points of the form $k a_s x_i$ and having slack$<s/2$ we conclude that these points have $\bhatp$-value in $[s/2,\infty)\times \{x_i\}$ which implies that they are contained in $\overline{Nx_i}$. Moreover, the entire $A_+N$-orbit of these points is contained in $\overline{Nx_i}$.
	
	Hence locally around $a_s x_i$ we have witnessed infinitely many two-dimensional half-planes. As we know that the set of quasi-minimizing directions does not contain an interval (it is in fact 0 Hausdorff dimensional) we conclude that $\overline{N x_i}$ is locally not a manifold.

\appendix
\section{Chain recurrence of the stretch lamination}

\newcommand{\Lf}{\Lambda_+}
\newcommand{\Lb}{\Lambda_-}
\newcommand{\hatLf}{\hat\Lambda_+}
\newcommand{\hatLb}{\hat\Lambda_-}
\newcommand{\Lbf}{\Lambda_\pm}

In this appendix we give the proof of Theorem \ref{L chain recurrent}, which we recall
states that for a 
closed hyperbolic $d$-manifold  $\Sigma_0$ and   a nontrivial homotopy class of maps $\Sigma_0\to
\R/\Z$, the stretch lamination $\lambda_0$ is chain-recurrent. 

The proof for $d=2$ in \cite{GK:stretch} uses the structure theory of geodesic
laminations on surfaces, but on the other hand applies to any dimensional target. 

Before we start let us recall the definition of chain-recurrence. If $\mu$ is an oriented
lamination in a hyperbolic manifold and $x,y\in \mu$ then a {\em $(b,\ep)$-chain}
from $x$ to $y$, where $b,\ep>0$, is a sequence of directed subsegments $\alpha_0,\ldots,\alpha_k$ of
$\mu$, each of length at least $b$, such that $\alpha_0$ begins at $x$, $\alpha_k$ terminates at $y$, and for each
$i<k$ the terminal point of $\T^1\alpha_i$ is within $\ep$ of the initial point of
$\T^1\alpha_{i+1}$. We say that $x\in \mu$ is chain-recurrent if there exists $b>0$
such that for every $\ep>0$ there is a $(b,\ep)$-chain from $x$ to itself. One can check that the
set of chain-recurrent points must be a sublamination of $\mu$, and if it is all of
$\mu$ then we say $\mu$ is chain-recurrent.

\begin{proof}[Proof of Theorem \ref{L chain recurrent}]
Let $\tau_0: \Sigma_0 \to \R/c\Z$ be a tight map in the given homotopy class, where $c>0$ has
been chosen so that the Lipschitz constant of $\tau_0$ is 1. We may assume \cite[Theorem 1.3]{GK:stretch} that $\tau_0$
has been chosen so that $\lambda_0$ is the entire locus where the local Lipschitz constant is
1. Recall that $\lambda_0$ is the intersection of the maximal stretch sets over all maps in our
given homotopy class.

If $\lambda_0$ is not chain-recurrent, let $x\in \lambda_0$ be a non chain-recurrent point. We will find
a homotopic $\tau'$ whose maximal stretch set does not include $x$, thus obtaining a
contradiction.

The oriented lamination $\lambda_0$ admits an $A$ action by geodesic flow, by lifting it to
$\T^1_+\lambda_0\subset \T^1\Sigma_0$, and applying $A$ there. We adopt this notation throughout, so that for
$x\in\lambda_0$, $Ax$ is the lamination leaf through $x$. 

Let $\Lf=\Lf(x)$ and $\Lb=\Lb(x)$  be the following sets:

We let $y\in \Lf$ if $y\in \lambda_0 \ssm Ax$, and if there exists $b>0$ so that for each $\ep>0$ there is a $(b,\ep)$-chain from
$x$ to $y$. Similarly, $y\in \Lb$ if $y\in \lambda_0\ssm Ax$, and there exists $b>0$ so that for each $\ep>0$ there is a
$(b,\ep)$-chain from $y$ to $x$.  

We note the following:
\begin{enumerate}
\item $\Lf$ and $\Lb$ are disjoint.
  If $y$ is in the intersection, then there exists $b>0$ so that for each $\ep$ we have a $(b,\ep)$-chain from $x$ to
  $y$ and back to $x$, which contradicts the choice of $x$ as non-chain-recurrent.
  
\item $\Lf$ and $\Lb$ are $A$-invariant:

  Let $y\in\Lf$ and consider $a_ty$ for
  $t\in\R$. Fixing $b>0$, for each $\ep$ consider a $(b,\ep)$-chain $c_\ep$ from $x$ to $y$.
  If $t\ge 0$ then we can extend the last segment of the chain to reach $a_ty$, and
  conclude $a_ty\in \Lf$.

If $t<0$, first note that if the total
  length of $c_\ep$ is bounded as $\ep\to 0$ then $y\in Ax$ which contradicts the
  definition. Thus the length goes to $\infty$, and so for small enough $\ep$ we can flow
  by $t$ along the chain and adjust the segments to obtain a $(b,\ep')$-chain that goes from $x$ to $a_ty$, where
  $\ep'\to 0$ as $\ep\to 0$.

  The argument for $\Lb$ is the same. 

\item $\Lf$ and $\Lb$ are closed.

  Let $y_n\in \Lf$ converge to $y\in \Sigma_0$. Assume first that $y\notin Ax$. 
For each $y_n$ we have $b_n>0$ such that
  for all $\ep>0$ there is a $(b_n,\ep)$ chain from $x$ to $y_n$. If these chains all had
  bounded total length then $y_n$ would be in $A_{[0,T]}x$ for some fixed $T$ and then so
  would $y$, a contradiction. Thus the lengths  are unbounded. Note that by replacing blocks of $k$-consecutive segments in a $(b',\varepsilon)$-chain by one long segment we get a $(kb',\varepsilon')$-chain. Thus even in the case where $b_n \to 0$, since $\varepsilon$ can be taken arbitrarily small we are always ensured that all $y_n$ are reachable by a $(b,\varepsilon)$-chain, for arbitrary $\varepsilon>0$. Hence so is the point $y$, implying $y\in \Lf$.

  The possibility remains that  $y\in Ax$. But then $y=a_tx$ for some $t$, so applying $A$-invariance
$a_{-t}y_n\in \Lf$ converge to $x$, and this implies that $x$ is chain-recurrent, again a
contradiction. The proof for $\Lb$ is the same. 
\end{enumerate}

Let $\tau:\Sigma\to \R$ be the lift of $\tau_0$ to the $\Z$-cover $\Sigma$. Let
$\hatLf$ and $\hatLb$ be the preimages of $\Lf$ and $\Lb$ in $\Sigma$. 
Next we claim:
\begin{equation}\label{eqn: epsilon room}
  \inf\{ d(y,z) - (\tau(y)-\tau(z)) : y\in\hatLb, z\in\hatLf\} > 0.
\end{equation}

Note that the $1$-Lipschitz property of $\tau$ means that this infimum is non-negative. If
it is zero, let $y_n\in\hatLb,z_n\in\hatLf$ be such that
$$d(y_n,z_n) - (\tau(y_n)-\tau(z_n)) \to 0.$$
Note that this quantity is just the slack $\s(\gamma_n)$ where $\gamma_n$ is a
distance-minimizing geodesic from $z_n$ to $y_n$.

Now because $\Lbf$ are compact and disjoint, the lengths of $\gamma_n$ are uniformly
bounded below. Thus, applying Lemma \ref{small slack close to L}, we obtain  some uniform $b>0$ and $\delta_n\to 0$ so that $\gamma_n$
can be cut into pieces of size roughly $b$  whose lifts to $\T^1\Sigma_0$ lie in
$\delta_n$-neighborhoods of $\T^1\lambda_0$.
In particular this gives us $(b,\delta_n)$-chains from
$z_n$ to $y_n$. These descend to chains in $\Sigma_0$ from $\bar z_n$ to $\bar y_n$, where
$\bar z_n \in \Lf$ and $\bar y_n\in \Lb$. 

Combining these with the chains from $x$ to $\bar z_n$ and from $\bar y_n$ to $x$ given by the
definition of $\Lbf$, we find that $x$ is chain-recurrent, again a contradiction. This
completes the proof that inequality (\ref{eqn: epsilon room}) holds.

\medskip

Now let $\ep>0$ be smaller than the infimum in (\ref{eqn: epsilon room}), and smaller than
$c$. Define
$\tau':\hatLf\cup\hatLb \to \R$ as follows:
\begin{align*}
\tau'|_{\hatLb} &= \tau|_{\hatLb} \\
\tau'|_{\hatLf} &= \tau|_{\hatLf} - \ep.
\end{align*}
We check that $\tau'$ is $1$-Lipschitz:  if $y,z\in \hatLf$ 
or $y,z\in \hatLb$ 
then $\tau'(y)-\tau'(z) = \tau(y)-\tau(z)$ so the fact that $\tau$ is $1$-Lipschitz
suffices. If $z\in\hatLf$ and $y\in\hatLb$ then
by (\ref{eqn: epsilon room}) and the choice of $\ep$ we have
$$\tau'(y)-\tau'(z) = \tau(y)-\tau(z) + \ep < d(y,z).$$
On the other hand
$$\tau'(z)-\tau'(y) = \tau(z)-\tau(y) - \ep < \tau(z)-\tau(y) \le d(y,z)$$
Thus $|\tau'(y)-\tau'(z)| < d(y,z)$ so indeed $\tau'$ is 1-Lipschitz.

By a classical theorem of McShane \cite{McShane}, the function 
\[ w \mapsto \inf\{ \tau'(z)+d_\Sigma(z,w) : z \in \hatLf\cup\hatLb \} \]
is a $1$-Lipschitz extension of $\tau'$ to all of $\Sigma$. We shall denote this extension by $\tau'$ as well. One can easily verify from the formula that equivariance of $\tau'$ on $\hatLf\cup\hatLb$ implies that the extended function is also $\Z$-equivariant. Thus $\tau'$ descends to a $1$-Lipschitz function $\tau'_0: \Sigma_0\to\R/c\Z$, which is in the same homotopy class as $\tau_0$.

However, $Ax$ cannot be in the stretch locus of $\tau'_0$. Suppose that it were. Then,
lifting $x$ to $\hat x\in \Sigma$ we would have $\tau'(a_t\hat x) = \tau'(\hat x) + t$ for all
$t\in\R$. We may assume for convenience that $\tau'(\hat x)=0$.
Consider a sequence $n_i\to\infty$ such that $a_{cn_i} x$ converges to $z$ --
then $z\in \Lf$. Note that $\tau'_0(a_{cn_i}) = 0\mod c\Z$ so $\tau'_0(z)=0\mod c\Z$. But
$z\in\Lf$ implies that $\tau'_0(z) = -\ep \mod c\Z$, which is a contradiction (we chose $0<\ep<c$), and we
conclude that $Ax$ is not in the stretch locus of $\tau'_0$.

But this contradicts the hypothesis that $Ax$ is in $\lambda_0$, the common stretch locus of all maps
homotopic to $\tau_0$, so the proof of the theorem is complete. 

\end{proof}

\bibliography{references}{}
\bibliographystyle{amsalpha.bst}

\end{document}